%% file: TASEP-variants-final-v3.tex
\documentclass[letterpaper,reqno,11pt,oneside,final]{amsart} 
\pdfoutput=1

\usepackage{amsmath,amsthm,amsfonts,amssymb}
\usepackage[extension=pdf]{hyperref}
\usepackage{enumitem,comment}
\usepackage[T1]{fontenc}
\usepackage{times}
\usepackage{soul}
\usepackage{mathtools}
\usepackage{mathrsfs}
\usepackage{cases}
\usepackage{needspace}
\usepackage{bm}
\usepackage[scr=boondoxo,scrscaled=1.05]{mathalfa}
\usepackage[dvipsnames]{xcolor}
\usepackage[multiple]{footmisc}
\usepackage{microtype}
\usepackage{xspace}
 \usepackage{tikz}
 \usetikzlibrary{matrix,backgrounds,decorations.pathreplacing}
\usepackage{rotating}
\usepackage[totalheight=9.6in,totalwidth=6.15in,centering]{geometry}
\usepackage{graphics,graphicx}
\usepackage{tocvsec2} 
\usepackage{mathtools}
\usepackage{xr}
\usepackage{ifthen}
\usepackage[color,notref,notcite]{showkeys}
\usepackage[style=alphabetic,natbib=true,maxnames=99,isbn=false,doi=false,url=false,firstinits=true,hyperref=auto,arxiv=abs,backend=bibtex]{biblatex}
\AtEveryBibitem{%
\clearlist{language}}
\DeclareFieldFormat[article,inbook,incollection,inproceedings,patent,thesis,unpublished]{title}{#1\isdot}
\renewbibmacro{in:}{%
\ifentrytype{article}{}{%
  \printtext{\bibstring{in}\intitlepunct}}} 
\defbibheading{apa}[\refname]{\section*{#1}}
\DeclareFieldFormat{sentencecase}{\MakeSentenceCase{#1}} \renewbibmacro*{title}{%
\ifthenelse{\iffieldundef{title}\AND\iffieldundef{subtitle}}{}
{\ifthenelse{\ifentrytype{article}\OR\ifentrytype{inbook}%
    \OR\ifentrytype{incollection}\OR\ifentrytype{inproceedings}%
    \OR\ifentrytype{inreference}} {\printtext[title]{%
      \printfield[sentencecase]{title}%
      \setunit{\subtitlepunct}%
      \printfield[sentencecase]{subtitle}}}%
  {\printtext[title]{%
      \printfield[titlecase]{title}%
      \setunit{\subtitlepunct}%
      \printfield[titlecase]{subtitle}}}%
  \newunit}%
\printfield{titleaddon}}
\usepackage[f]{esvect}
\usepackage{stmaryrd}

\DeclareFontFamily{U}{BOONDOX-calo}{\skewchar\font=45 }
\DeclareFontShape{U}{BOONDOX-calo}{m}{n}{
  <-> s*[1.05] BOONDOX-r-calo}{}
\DeclareFontShape{U}{BOONDOX-calo}{b}{n}{
  <-> s*[1.05] BOONDOX-b-calo}{}
\DeclareMathAlphabet{\mcb}{U}{BOONDOX-calo}{m}{n}
\SetMathAlphabet{\mcb}{bold}{U}{BOONDOX-calo}{b}{n}

\makeatletter
\newcommand*\bigcdot{\mathpalette\bigcdot@{.5}}
\newcommand*\bigcdot@[2]{\mathbin{\vcenter{\hbox{\scalebox{#2}{$\m@th#1\bullet$}}}}}
\makeatother

\definecolor{darkblue}{rgb}{0.13,0.13,0.39}
\hypersetup{colorlinks=true,urlcolor=darkblue,citecolor=darkblue,linkcolor=darkblue,%
  pdftitle=TASEP and generalizations, pdfauthor={K. Matetski, D. Remenik}}

\mathtoolsset{showmanualtags,showonlyrefs}

\newtheorem{thm}{Theorem}[section] 
\newtheorem{lem}[thm]{Lemma}
 
\newtheorem{prop}[thm]{Proposition} 
\newtheorem{cor}[thm]{Corollary}

\theoremstyle{definition} 
\newtheorem{rem}[thm]{Remark} 
\newtheorem*{rem*}{Remark}

\newtheorem{ex}[thm]{Example}
 
\newtheorem{assumption}[thm]{Assumption} 
\newcounter{assum}

\newcommand{\head}{{\uptext{head}}}

\newcommand{\TASEP}{{\uptext{TASEP}}}
\newcommand{\push}{{\uptext{push}}}

\newcommand\distr{\mathrel{\overset{\makebox[0pt]{\mbox{\normalfont\tiny\uptext{dist}}}}{=}}}
\newcommand{\prll}{\uptext{prll}}
\newcommand{\gen}{\uptext{gen}}
\newcommand{\rB}{\uptext{r-B}}
\newcommand{\lB}{\uptext{l-B}}
\newcommand{\rG}{\uptext{r-G}}
\newcommand{\lG}{\uptext{l-G}}

\newcommand{\UC}{\uptext{UC}}

\newcommand{\set}[1]{\llbracket #1 \rrbracket}
\renewcommand{\ul}[1]{\underline{#1}}
\newcommand{\ut}{\ul{t}}
\newcommand{\un}{\ul{n}}

\renewcommand{\d}{\mathrm{d}}
\newcommand{\per}{\uptext{prd}}

\newcommand{\T}{\mathbb{T}}
\newcommand{\V}[1]{\mathbb{V}_{\hspace{-0.1em}#1}}
\newcommand{\W}[1]{\mathbb{W}_{\hspace{-0.08em}#1}}
\newcommand{\Wd}[1]{\mathbb{W}^\circ_{\hspace{-0.08em}#1}}
\newcommand{\f}{\mathbf{f}}
\newcommand{\N}{\mathbb{T}}
\newcommand{\mQ}{\bar{Q}}
\newcommand{\Q}{\mcb{V}}
\newcommand{\E}{\vartheta}

\newcommand{\GT}{\mathbb{GT}}
\newcommand{\D}{\mathbb{D}}
\newcommand{\G}{G}
\renewcommand{\H}{\mathbf{F}}
\newcommand{\F}{\mathbf{F}}
\newcommand{\A}{\mathbf{A}}
\newcommand{\fh}{\mathfrak{h}}
\newcommand{\fn}{\mathfrak{n}}

\newcommand{\xx}{X}
\newcommand{\x}{\mathsf{x}}
\newcommand{\X}{\mathsf{X}}

\newcommand{\R}{\mathcal{R}}

\newcommand{\CM}{\mathcal{M}}
\newcommand{\CE}{\mathcal{E}}
\newcommand{\cK}{\mathcal{K}}

\newcommand{\Qt}{Q^*}

\newcommand{\CN}{\mathcal{N}}

\newcommand{\cS}{\mathcal{S}}
\newcommand{\cT}{\mathcal{T}}

\newcommand{\cW}{\mathcal{W}}

\newcommand{\vS}{\vv{S}}
\newcommand{\vT}{\vv{T}}

\newcommand{\arrowright}[1]{\parbox{#1}{\tikz{\draw[->](0,0)--(#1,0);}}}
\newcommand{\arrowleft}[1]{\parbox{#1}{\tikz{\draw[<-](0,0)--(#1,0);}}}

\newcommand{\la}{\vphantom{m}\arrowleft{.15cm}}
\newcommand{\ra}{\vphantom{m}\arrowright{.15cm}}

\newcommand{\I}{{\rm i}} 
\newcommand{\pp}{\mathbb{P}}

\newcommand{\ee}{\mathbb{E}} 
\newcommand{\rr}{\mathbb{R}}
\newcommand{\nn}{\mathbb{N}} 
\newcommand{\zz}{\mathbb{Z}}
\newcommand{\bn}{{\bar{n}}}

\newcommand{\M}{\CM}
\newcommand{\K}{K}
\newcommand{\Id}{I}
\newcommand{\cc}{{\mathbb{C}}}
\newcommand{\p}{\partial}
\newcommand{\uno}[1]{\mathbf{1}_{#1}}
\newcommand{\ep}{\varepsilon}

\newcommand{\vs}{\vspace{6pt}}
\newcommand{\wt}{\widetilde}

\newcommand{\m}{\overline{m}}

\newcommand{\qqand}{\qquad\text{and}\qquad}

\DeclareMathOperator{\Ai}{\uptext{Ai}}

\newcommand{\ts}{\hspace{0.1em}}
\newcommand{\tts}{\hspace{0.05em}}

\newcommand{\bioneref}{\hyperlink{it:biorth4}{\normalfont ($\star$)}\xspace}
\newcommand{\bitworef}{\hyperlink{it:poly4}{\normalfont ($\star\star$)}\xspace}

\makeatletter
\newcommand\RedeclareMathOperator{%
  \@ifstar{\def\rmo@s{m}\rmo@redeclare}{\def\rmo@s{o}\rmo@redeclare}%
}
\newcommand\rmo@redeclare[2]{%
  \begingroup \escapechar\m@ne\xdef\@gtempa{{\string#1}}\endgroup
  \expandafter\@ifundefined\@gtempa
     {\@latex@error{\noexpand#1undefined}\@ehc}%
     \relax
  \expandafter\rmo@declmathop\rmo@s{#1}{#2}}
\newcommand\rmo@declmathop[3]{%
  \DeclareRobustCommand{#2}{\qopname\newmcodes@#1{#3}}%
}
\@onlypreamble\RedeclareMathOperator
\makeatother
\newcommand{\uptext}[1]{\text{\upshape{#1}}}

\RedeclareMathOperator{\det}{\mathop{\uptext{det}}}
\RedeclareMathOperator{\ker}{\mathop{\uptext{ker}}}
\RedeclareMathOperator{\exp}{\mathop{\uptext{exp}}}
\RedeclareMathOperator{\log}{\mathop{\uptext{log}}}
\RedeclareMathOperator*{\lim}{\mathop{\uptext{lim}}}
\RedeclareMathOperator*{\sup}{\mathop{\uptext{sup}}}
\RedeclareMathOperator*{\limsup}{\mathop{\uptext{lim\hspace{1pt}sup}}}
\RedeclareMathOperator*{\liminf}{\mathop{\uptext{lim\hspace{1pt}inf}}}
\RedeclareMathOperator*{\max}{\mathop{\uptext{max}}}
\RedeclareMathOperator*{\inf}{\mathop{\uptext{inf}}}
\RedeclareMathOperator*{\min}{\mathop{\uptext{min}}}
\RedeclareMathOperator*{\cos}{\mathop{\uptext{cos}}}
\RedeclareMathOperator*{\sin}{\mathop{\uptext{sin}}}
\RedeclareMathOperator*{\arg}{\mathop{\uptext{arg}}}
\RedeclareMathOperator{\Re}{\uptext{Re}}
\RedeclareMathOperator{\Im}{\uptext{Im}}

\DeclareMathOperator{\spanning}{\uptext{span}}

\newcommand{\twopii}[1]{\ifthenelse{#1=1}{2\pi\I}{(2\pi\I)^{#1}}}

\newcommand{\SLP}{\mathbb{S}}

\newcommand{\SM}{\mathcal{S}}
\newcommand{\SN}{\bar{\mathcal{S}}}

\newcommand{\fT}{\mathbf{S}}

\newcommand{\ft}{\mathbf{t}}
\newcommand{\ff}{\mathfrak{f}}

\newcommand{\fx}{\mathbf{x}}

\newcommand{\fa}{\mathbf{a}}
\newcommand{\fB}{\mathbf{B}}

\newcommand{\fK}{\mathbf{K}}
\newcommand{\fI}{\mathbf{I}}
\newcommand{\ftau}{\bm{\tau}}

\renewcommand{\P}{\chi}
\newcommand{\bP}{\bar{\P}}

\newcommand{\rin}{r}
\newcommand{\rout}{\bar\rin}
\newcommand{\rrin}{\rin}
\newcommand{\rrout}{\rout}
\newcommand{\rhoin}{\rho}
\newcommand{\rhoout}{\bar\rho}

\let\Re\relax
\DeclareMathOperator{\Re}{Re}
\DeclareMathOperator{\hypo}{\uptext{hypo}}
\DeclareMathOperator{\epi}{\uptext{epi}}

\def\dash---{\kern.16667em---\penalty\exhyphenpenalty\hskip.16667em\relax}

\input{catp-ch}

\numberwithin{equation}{section}

\let\oldmarginpar\marginpar
\renewcommand\marginpar[1]{\-\oldmarginpar[\raggedleft\footnotesize #1]%
  {\raggedright{\small\textsf{#1}}}}

\allowdisplaybreaks[2]

\settocdepth{section}

\begin{document}

\title{TASEP and generalizations: Method for exact solution}

\author{Konstantin Matetski} \address[K.~Matetski]{
  Department of Mathematics\\
  Columbia University\\
  2990 Broadway\\
  New York, NY 10027\\
  USA} \email{matetski@math.columbia.edu}

\author{Daniel Remenik} \address[D.~Remenik]{
  Departamento de Ingenier\'ia Matem\'atica and Centro de Modelamiento Matem\'atico (UMI-CNRS 2807)\\
  Universidad de Chile\\
  Av. Beauchef 851, Torre Norte, Piso 5\\
  Santiago\\
  Chile} \email{dremenik@dim.uchile.cl}
  
\begin{abstract}
 The explicit biorthogonalization method, developed in \cite{fixedpt} for continuous time TASEP, is generalized to a broad class of determinantal measures which describe the evolution of several interacting particle systems in the KPZ universality class.
 The method is applied to sequential and parallel update versions of each of the four variants of discrete time TASEP (with Bernoulli and geometric jumps, and with block and push dynamics) which have determinantal transition probabilities; to continuous time PushASEP; and to a version of TASEP with generalized update.
 In all cases, multipoint distribution functions are expressed in terms of a Fredholm determinant with an explicit kernel involving hitting times of certain random walks to a curve defined by the initial data of the system.
 The method is further applied to \emph{systems of interacting caterpillars}, an extension of the discrete time TASEP models which generalizes sequential and parallel updates.
\end{abstract}

\maketitle
  \tableofcontents 
  
\section{Introduction and main result}
\label{sec:intro}

A large class of exactly solvable models in the KPZ universality class can be described as marginals of determinantal measures, corresponding to either determinantal point processes or their generalizations to complex-valued measures.
For instance, some projections of Schur processes with suitably chosen specifications have the distribution of (discrete or continuous time) totally asymmetric simple exclusion processes (TASEP) with blocking and pushing interaction and with very special initial states (typically with an infinite number of particles and with either half-packed or half-stationary initial states) \cite{IntProbLectures}.
This naturally yields formulas for the cumulative distribution functions of the particle positions based on the Fredholm determinant of a kernel given in the form of a double contour integral.
These formulas can be used to show that, in the usual KPZ 1:2:3 scaling limit, the asymptotic fluctuations of the particle positions are described by the Tracy-Widom GUE distribution from random matrix theory \cite{tracyWidom} and, more generally, their joint distributions lead to the Airy$_2$ process \cite{johanssonShape,prahoferSpohn,johansson}.

For general initial condition, a representation for continuous time TASEP as a marginal of a signed determinantal measure was discovered in \cite{sasamoto,borFerPrahSasam}, where a formula in terms of a Fredholm determinant was derived involving a kernel characterized implicitly by a certain biorthogonalization problem.
For the simplest initial condition, half-packed, this biorthogonalization problem becomes trivial and the previously derived kernel is recovered.
In the case of $2$-periodic initial state the biorthogonalization problem was solved explicitly in those papers, and in the 1:2:3 scaling limit the resulting formulas led to the Airy$_1$ process, with one-point marginals now given by the Tracy-Widom GOE distribution \cite{tracyWidom2}.
Later this method was applied to several other models, including discrete time TASEP with sequential and parallel updates \cite{bfp,borodFerSas} and PushASEP \cite{bp-push} with periodic initial conditions.
The method was also applied to compute the distribution of the two-dimensional process of interacting particles introduced in \cite{Anisotropic}, whose projections yield a coupling of TASEPs with pushing and blocking interactions.

For continuous time TASEP with arbitrary (one-sided) initial condition, the biorthogonalization problem was solved in \cite{fixedpt}, leading to a kernel which can be expressed in terms of the hitting time of a certain random walk to a curve defined by the initial data.
This was used to show that, if the initial data converges under diffusive scaling, then in 1:2:3 scaling limit the TASEP height function converges to a Markov process whose fixed time, multipoint distributions are explicitly given by a Fredholm determinant of an analogous kernel, now defined in terms of Brownian hitting times.
This limiting process is known as the \emph{KPZ fixed point}, and is expected to arise as the universal scaling limit of all processes in the KPZ universality class.
The same approach was later used in studying the KPZ fixed point scaling limit of reflected Brownian motions \cite{Mihai}.

The purpose of this article is to extend the biorthogonalization method of \cite{sasamoto,borFerPrahSasam} and the explicit biorthogonalization scheme developed in \cite{fixedpt} to compute correlation kernels for a general class of determinantal measures, whose marginals in particular cases yield several exactly solvable models in the KPZ class.
The class of measures which we will study include, in particular, discrete time variants of TASEP with either blocking or pushing interaction and with Bernoulli or geometric jumps. Formulas will be derived for versions of these processes with either sequential or parallel update, unifying in particular the biorthogonalization schemes presented in \cite{bfp} and \cite{borodFerSas}.

Sequential and parallel update dynamics for TASEP-like systems will be obtained as particular cases of a more general class of \emph{systems of interacting caterpillars} which we introduce. Along the way we will study some natural versions of these interacting particle systems which seem to not have been considered before.

The formulas which we will obtain have the same structure as those obtained for continuous time TASEP, and for measures corresponding to models in the KPZ class they can be used to show convergence to the KPZ fixed point as in \cite{fixedpt}.
This is left for future work.

\subsection{Notation and conventions}\label{sec:notat}

We use $\nn$ for the set of natural numbers $1, 2, \dotsc$, and we denote $\nn_0 = \nn \cup \{0\}$.
For $N \in \nn$ we use the shorthand notation $\set{N} = \{1, \dotsc, N\}$.
$t$ will denote a time variable taking values in a domain $\T$, which can be either $\T = \rr$ or $\T = \zz$.

For $N \geq 2$ the \emph{Weyl chamber} is 
\[\Omega_N = \{(x_1,\dotsc,x_N)\in\zz^N\!:x_1>x_2>\dotsm>x_N\}.\]

$\gamma_{r}$ will denote a circle in the complex plane with radius $r$, centered at $0$.
$A_{r,R}$ will denote the closed annulus on the complex plane centered at $0$ and with radii $0<r<R$.
For a closed subset $U$ of $\cc$ we say that a complex function $f$ is analytic on $U$ if it is analytic on some open domain which contains $U$.

Throughout the paper we will consider many different kernels $K\!:\zz\times\zz\longrightarrow\rr$.
We regard such kernels as integral operators acting on suitable families of functions $f\!: \zz \to \cc$, i.e.,
\begin{equation}
Kf(x) = \sum_{y \in \zz} K(x,y)f(y),
\end{equation}
provided the sum is absolutely convergent.
Two such kernels of this form are composed as $KL(x,y)=\sum_{z\in\zz}K(x,z)L(z,y)$, provided again that the sum is absolutely convergent.
We will usually not need to spell out the precise function spaces on which these operators act; in particular, compositions of these kernels will be well defined by
the absolute convergence of all the sums involved.

By the inverse of a kernel $K$ we mean a kernel $K^{-1}$ such that $K^{-1}(x,y)=K^{-1}K(x,y)=\uno{x=y}$.
For a kernel $K$ we denote its adjoint as $K^*(x_1, x_2) = K(x_2,x_1)$.

\subsection{Main result}
\label{sec:main}

We present next our main result in the context of a class of Markov chains with determinantal transition probabilities.
This is a particular case of the results proved in the paper, which work for a more general class of determinantal measures.
We do this in order to simplify the presentation, and because it is enough to cover most of the applications to examples of interest (exceptions are discrete time TASEP with geometric jumps and sequential update, see Sec.~\ref{sec:rightGeometric}, and the generalized TASEP dynamics from Secs. \ref{sec:Povolotsky-discrete} and \ref{sec:Povolotsky-continuous}).

For $N \geq 1$ we consider a Markov chain $\xx_t$ on $\Omega_N$, where time $t \in \T$ is either continuous or discrete.
We interpret $\xx_t = (\xx_t(1), \dotsc, \xx_t(N))$ as the locations of a system of particles whose evolution preserves the order.
The key property which we ask of $\xx_t$ (first shown to hold for TASEP \cite{MR1468391}) is the following: the transition probabilities of the process from $\vec y \in \Omega_N$ to $\vec x \in \Omega_N$ are given by
\begin{align}\label{eq:G-main}
  \pp(X_t = \vec x | X_0 = \vec y) = \det \bigl[F_{i - j}(x_{N + 1 - i} - y_{N + 1 - j}, t)\bigr]_{i, j \in \set{N}},
\end{align}
where 
\begin{align}\label{eq:F-main}
F_{n}(x, t) = \frac{1}{2\pi\I}\oint_{\gamma_{\rhoout}}\!\!\d w\, \frac{(w - 1)^{-n}}{w^{x - n +1}} \varphi(w)^t
\end{align}
for some $\rhoout>1$ and some complex function $\varphi$ which, besides giving a probability distribution in \eqref{eq:G-main} (this in particular implies $\varphi(1) = 1$), satisfies the following:

\begin{assumption}\label{a:phi}
\leavevmode
\begin{enumerate}[label=\uptext{(\roman*)}]
\item $\varphi\!:U\longrightarrow\cc$, where the domain $U\subseteq\cc$ contains $0$ and $1$, and $\varphi$ has at most a finite number of singularities in $U$.
\item $\varphi$ is analytic on an annulus $A_{\rhoin,\rhoout} \subseteq U$ with radii $0<\rhoin < 1 <\rhoout$.
\item $\varphi(w) \neq 0$ for all $w\in A_{\rhoin,\rhoout}$.
\end{enumerate}
\end{assumption}

Throughout the rest of this section we assume that Assum.~\ref{a:phi} is satisfied and, in particular, we fix the parameter $\rho$ appearing in the assumption.

Note that for $n\leq0$ we can shrink the contour in the integral in \eqref{eq:F-main} to a circle of radius less than $1$, from which we get $\sum_{z<x}F_n(z,t)=\uno{n=0}-F_{n+1}(x,t)$ for such $n$ (where for $n=0$ we have used $\varphi(1)=1$).
Using this and the multilinearity of the determinant in \eqref{eq:G-main} one gets
\begin{equation}\label{eq:chierarchy}
	\textstyle\sum_{x_N < x_{N-1}} \pp(X_t = \vec x | X_0 = \vec y) = \pp(X^{(N-1)}_t = \vec x_{<N} | X^{(N-1)}_0 = \vec y_{<N}),
\end{equation}
where $X^{(N-1)}_t$ denotes the process with $N-1$ particles and the vector $\vec x_{<N}$ is obtained from $\vec x$ by removing the last entry. In other words, we can remove the last particle from the $N$-particle system to obtain the same evolution on $N-1$ particles; equivalently, the first $N-1$ particles do not ``feel'' the presence of the $N^{\text{th}}$ one.

In order to state our result for the joint cumulative distribution function of the particle locations, we need to make some definitions.
To this end we introduce a parameter $\kappa\in\nn_0$, which for now can be thought of as being $0$ ($\kappa>0$ will be used later in Thm.~\ref{thm:main2} to state a more general result).
We also introduce an auxiliary parameter $\theta\in(\rho,1)$, whose role will be clear shortly (in applications to scaling limits, $\theta$ is adjusted according to the density of particles in the system, see Rem. \ref{rem:fp-dens}). 

For $z_1,z_2\in\zz$ we define
\begin{equation}\label{eq:defQintro-pre}
Q(z_1,z_2) = \frac{\alpha}{2\pi\I}\oint_{\gamma_{r}}\d w\,\frac{\theta^{z_1-z_2}}{w^{z_1-z_2}}\frac{\varphi(w)^\kappa}{1-w}
\end{equation}
with $r\in(\rho,1)$ and with
\[\alpha=(1-\theta)\theta^{-1}\varphi(\theta)^{-\kappa}.\]
When $\kappa=0$ one has $Q(z_1,z_2)=(1-\theta)\theta^{z_1-z_2-1}\uno{z_1>z_2}$, i.e., $Q$ is the transition matrix of a random walk on $\zz$ taking Geom$[1-\theta]$ steps (strictly) to the left.
In the case $\kappa\geq1$ it is not very hard to check that $\sum_{z_2\in\zz}Q(z_1,z_2)=1$ (for instance, as in the proof of Lem.~\ref{lem:conv}), and we will impose an additional condition on $\varphi$ (see Assum.~\ref{a:kappa} below) which will ensure positivity, so that $Q$ is still the transition matrix of a random walk. 
In any case, under Assum.~\ref{a:phi}, $Q$ has an inverse, and the $n$-th powers of $Q$ and its inverse can be obtained explicitly by convolution (see Lem. \ref{lem:conv}): for all $n\in\zz$ we have
\begin{equation}\label{eq:defQintro}
Q^n(z_1,z_2) = \frac{\alpha^n}{2\pi\I}\oint_{\gamma_{\rrin}}\d w\,\frac{\theta^{z_1-z_2}}{w^{z_1-z_2-n+1}}\left(\frac{\varphi(w)^\kappa}{1-w}\right)^n.
\end{equation}

Next we set
\begin{align}
\SM_{-t, -n}(z_1,z_2) &= \frac{\alpha^{-n+1}}{2\pi\I}\oint_{\gamma_{\rrin}}\d w\,\frac{\theta^{z_2-z_1}}{w^{z_2-z_1+n+1}}(1-w)^n\varphi(w)^{t-\kappa(n-1)},\label{eq:defSMintro}\\
\SN_{-t, n} (z_1,z_2) &= \frac{\alpha^{n-1}}{2\pi\I}\oint_{\gamma_\delta}\d w\,\frac{(1-w)^{z_2-z_1+n-1}}{\theta^{z_2-z_1}w^n}\varphi(1-w)^{\kappa(n-1)-t},\label{eq:defSNintro}
\end{align}
with $r,\theta\in(\rho,1)$ as above and with $\delta>0$ small enough so that $\varphi(1-w)^{\kappa(n-1)-t}$ is analytic inside $\gamma_\delta$.
If we introduce another family of kernels
\[\R_t(z_1,z_2)=\frac1{2\pi\I}\oint_{\gamma_\rin}\d w\,\frac{\theta^{z_1-z_2}}{w^{z_1-z_2+1}}\varphi(w)^t\]
for $t\in\zz$, then the last two kernels can be written as
\begin{equation}\label{eq:SMSNQ}
\SM_{-t, -n}=Q^{-n}\R_{t+\kappa}\qqand\SN_{-t,n}=\mQ^{(n)}\R_{-t-\kappa}
\end{equation}
with $\mQ^{(n)}=\SN_{0,n}$ (see Sec.~\ref{sec:biorth}, and in particular the comment at the end of the section, for a proof of these formulas).

Now we focus on the case $\kappa=0$, so that $Q$ is the transition matrix of a random walk on $\zz$ with Geom$[1-\theta]$ jumps to the left, which we denote by $B$.
Fix $\vec y\in\Omega_N$ and let
\begin{equation}\label{eq:tauintro}
\tau= \min\{m = 0, \dotsc, N-1 : B_m> y_{m+1}\}
\end{equation}
be the hitting time of the strict epigraph of the ``curve'' $(y_{m+1})_{m=0,\dotsc,n-1}$ by the random walk $(B_m)_{m\geq0}$ (we set $\tau=\infty$ if the walk does not go above the curve by time $N-1$).
Then we set 
\begin{equation}\label{eq:SMepi}
	\SN^{\epi(\vec y)}_{-t, n}(z_1,z_2) = \ee_{B_{0}=z_1}\!\left[\SN_{-t, n-\tau}(B_\tau,z_2) \uno{\tau<n}\right].
\end{equation}
The indicator $\uno{\tau<n}$ can be omitted in the expectation, because $\SN_{-t,m}$ vanishes for $m\leq0$, as can be seen from \eqref{eq:defSNintro}.

Finally, for a fixed vector $a\in\rr^m$ and indices $n_1<\dotsm<n_m$ we let 
\begin{equation}\label{eq:defChis}
\chi_a(n_j,x)=\uno{x>a_j}\qqand\bP_a(n_j,x)=\uno{x\leq a_j},
\end{equation}
which we also regard as multiplication operators acting on the space $\ell^2(\{n_1,\dotsc,n_m\}\times\zz)$. We will also use this notation when the first argument is a pair $(n_j,t)$ with $t\in\T$, with $\chi_a((n_j,t),x)=1-\bP_a((n_j,t),x)=\P_a(n_j,x)$, as well as in the case that $a$ is a scalar, writing $\chi_a(x)=1-\bP_a(x)=\uno{x>a}$.

The following is the simplest version of the main result of this article, and can be applied for example to continuous time TASEP and discrete time TASEP with sequential update.

\begin{thm}\label{thm:main}
Let $\varphi$ satisfy Assum.~\ref{a:phi}. Then for any $t\geq0$, any $1 \leq n_1 < \dotsm <  n_m \leq N$, any $\vec a \in \rr^m$, and any $\vec y \in \Omega_N$, we have
\begin{equation}\label{eq:probability-main-kappa=0}
\pp \bigl(\xx_{t}(n_i) > a_{i},\, i \in \set{m}\, \big|\, X_{0} = \vec y\bigr) = \det \bigl(I-\bP_{a}  K_t \bP_{a} \bigr)_{\ell^2(\{n_1, \dotsc, n_m\} \times \zz)},
\end{equation}
with
\begin{equation}\label{eq:kernel-main-kappa=0}
K_t(n_1, \cdot\,; n_2, \cdot\,) = -Q^{n_2-n_1} \uno{n_1 < n_2} + (\SM_{-t, -n_1})^*\SN^{\epi(\vec y)}_{-t, n_2},
\end{equation}
where the objects on the right hand side are as in \eqref{eq:defQintro}--\eqref{eq:defChis} with $\kappa = 0$.
\end{thm}

An extension of the above setting involves considering particles which start at different times.
This extension, which will correspond to using $\kappa\geq1$ above, will allow us to cover discrete time TASEP with parallel update and the more general systems of interacting caterpillars, see Sec. \ref{sec:caterpillars}.
In this case it is convenient to regard the process $\xx_t$ as starting at negative times.
Then, for an integer $\kappa \geq 0$ 
we define the event 
\begin{equation}\label{eq:E-event}
\CE_\kappa = \bigcap_{i \in \set N} \bigl\{i^{\text{th}} ~\text{particle stays put until time}~ -\kappa (i-1)\bigr\}.
\end{equation}
If the process starts at time $t \leq -\kappa (N-1)$, then conditioning on this event means that the $i^{\uptext{th}}$ particle will only start evolving at time $-\kappa(i-1)$ (see Fig.~\ref{fig:TASEP} for an example of possible trajectories of the particles).
Note that particles with smaller indices start moving later.

We will be interested in models for which the following assumption holds.
In the assumption we take $\kappa\in\nn$ as given together with an initial state $\vec y\in\zz^N$.
We remark that, since we are interested in particles starting at different times, the correct space for initial conditions is not necessarily $\Omega_N$, and will have to be specified in each application (in Sec.~\ref{sec:caterpillars} there will be cases where particles have to start at distance at least $\kappa$ from each other and others where particles can initially be only weakly ordered).

\needspace{20pt}
\begin{assumption}\label{a:kappa}
\leavevmode
\begin{enumerate}[label=\uptext{(\alph*)}]
\item\label{it:one} Fix an integer $\kappa \geq 1$ and an initial state $\vec y\in\zz^N$.
Then for any $\vec x \in \Omega_N$,
\begin{equation}\label{eq:second-property}
\pp(X_0 = \vec x | X_{- \kappa (N-1)} = \vec y, \CE_\kappa) = \det \bigl[F_{i - j}(x_{N + 1 - i} - y_{N + 1 - j}, \kappa (j-1))\bigr]_{i, j \in \set{N}}.
\end{equation}
\item\label{it:phi} The function $\varphi$ from Assum.~\ref{a:phi} has the following additional properties:
\begin{enumerate}[label=\uptext{(\roman*)}]
\item $\varphi$ is analytic on $\{w\in\cc\!:|w|\geq\rhoin\}$, with $\rhoin$ the radius from Assum.~\ref{a:phi}.
\item $\varphi$ is the generating function of a real positive measure on $\{i\in\zz\!: i \leq 1\}$, i.e.,
\begin{equation}\label{eq:varphiappl}
\varphi(w)=\sum_{i\leq 1}b_iw^i
\end{equation}
with $b_i\geq0$ for all $i\leq1$ and not all $b_i$'s are zero.
\item $\sum_{i\leq1}b_i\leq1$.
\end{enumerate}
\end{enumerate}
\end{assumption}

Assum.~\ref{a:kappa}\ref{it:one}, which may look a bit artificial, is essentially just stating that the transition probabilities of the particle system with different starting times have a determinantal form similar to \eqref{eq:G-main} (note however how the time index in $F_{i-j}$ is shifted in \eqref{eq:second-property}).
We will see in examples that this property in general does not hold for any initial condition. Moreover, formula \eqref{eq:second-property} does not hold for $\kappa > 0$ for transition probabilities of the form \eqref{eq:G-main} with any $\varphi$, so we will need to prove it for the models we are interested in.

Assum.~\ref{a:kappa}\ref{it:phi}, on the other hand, encodes some extra restrictions on $\varphi$ which are not needed when $\kappa=0$.
They imply, in particular, that 
\begin{equation}
Q(z_1,z_2)=\alpha\tts\theta^{z_1-z_2}q_{z_1-z_2},\label{eq:Qalphaq-intro}
\end{equation}
where the $q_i$'s are non-negative and are uniquely defined through the conditions $q_i=1$ for $i>\kappa$ and $\varphi(w)^\kappa=\sum_{i\leq\kappa}(q_{i+1}-q_i)w^i$, and that $Q$ is still the transition matrix of a random walk on $\zz$ (see Sec.~\ref{sec:setting} for a proof of this in a more general setting and Sec.~\ref{sec:main-proof} for the application to the present context).
We will keep denoting by $B$ the random walk with transition matrix $Q$, by $\tau$ the associated hitting time \eqref{eq:tauintro}, and by $\SN^{\epi(\vec y)}_{-t, n}$ the kernel defined through \eqref{eq:SMepi} in terms of this new random walk (with $\SN_{-t, n}$ now given by \eqref{eq:defSNintro} with this $\kappa$).

The following result extends Thm.~\ref{thm:main} to the case of different starting times, where each particle is evolved for the same total amount of time $t$.
The basic case of discrete time TASEP with parallel update corresponds to $\kappa=1$, while $\kappa>1$ will yield the generalization to systems of caterpillars (and $\kappa=0$ essentially recovers Thm.~\ref{thm:main}).

\begin{thm}\label{thm:main2}
Assume that $\varphi$ satisfies Assums.~\ref{a:phi} and \ref{a:kappa} and let $\kappa \geq 0$ and $\vec y$ be as in Assum.~\ref{a:kappa}.
Then for any $1 \leq n_1 < \dotsm <  n_m \leq N$, any $t\geq\kappa(n_m-1)$, and any $\vec a \in \rr^m$, we have
\begin{equation}\label{eq:probability-main}
\pp \bigl(\xx_{t - \kappa (n_i-1)}(n_i) > a_{i},\, i \in \set{m}\, \big|\, X_{- \kappa (N-1)} = \vec y, \CE_\kappa\bigr) = \det \bigl(I-\bP_{a}  K_t \bP_{a} \bigr)_{\ell^2(\{n_1, \dotsc, n_m\} \times \zz)},
\end{equation}
with the kernel 
\begin{equation}\label{eq:kernel-main}
K_t(n_1, \cdot\,; n_2, \cdot\,) = -Q^{n_2-n_1} \uno{n_1 < n_2} + (\SM_{-t, -n_1})^*\SN^{\epi(\vec y)}_{-t, n_2},
\end{equation}
where the objects on the right hand side are as in \eqref{eq:defQintro}--\eqref{eq:defChis} with this value of $\kappa$.
Moreover, if $X$ satisfies the additional condition \eqref{eq:backward-in-time-one} given below, then \eqref{eq:probability-main} holds for $t \geq 0$.
\end{thm}

The quantity on the left hand side of \eqref{eq:probability-main} can be thought of as the distribution of the particles with a particular choice of starting and ending times (both regularly spaced by $\kappa$).
In Sec.~\ref{sec:measures} we derive an expression for \eqref{eq:kernel-main} as a biorthogonal kernel in the case of general starting and ending times, see Thm.~\ref{thm:biorth_general}.
The extension of the explicit biorthogonalization of Sec.~\ref{sec:biorth} to that case is left for future work.

The restriction $t\geq\kappa(n_m-1)$ in the theorem means that we are requiring all of the first $n_m$ particles to start moving before any of them stop.
This is used in our method of proof, but in fact it is not clear to us that the restriction can be lifted under the general assumptions of the theorem.
However, our argument can be extended to cover all $t\geq0$ under an additional assumption, which holds at least in the important case of discrete time TASEP with right Bernoulli jumps.
The additional assumption is the following: for any $t \geq 0$, any $\vec y \in \Omega_N$, any $x_1 \in \zz$ and any integer values $z_N < z_{N-1} < \dotsm < z_2$, satisfying $z_2 < y_{1}$, one has
\begin{align}\label{eq:backward-in-time-one}
\sum_{x_N < \dotsm < x_{2}\uptext{ s.t. } x_2< x_1} &\pp(X_t = \vec x | X_0 = \vec y) \det \bigl[F_{i - j}(z_{N + 1 - i} - x_{N + 1 - j}, -t)\bigr]_{i, j \in \set{N-1}} \\
&\hskip1.4in= \pp(X_t(1) = x_1 | X_0(1) = y_1) \prod_{2 \leq i \leq N} \uno{z_i = y_i}.
\end{align}
This amounts essentially to saying that the evolution of the particles can be run ``backwards in time'' by using the function \eqref{eq:G-main} with a negative time (even though the determinant on the right hand side doesn't have a clear physical meaning).

\subsection*{Outline}

The first two sections contain the applications of Thms.~\ref{thm:main} and \ref{thm:main2} (as well as their generalization in Sec.~\ref{sec:biorth}) to several interacting particle systems in discrete and continuous time: in Sec.~\ref{sec:caterpillars} we consider the four discrete time models which were solved in \cite{MR2469339}, as well as their extensions to parallel update dynamics and systems of interacting caterpillars, while Sec.~\ref{sec:continuousTASEP} is devoted to the models in continuous time.
In Sec.~\ref{sec:continuousTASEP} we also review briefly the convergence of TASEP to the KPZ fixed point and derive explicit formulas for these processes started with a special choice of random initial data.

\noindent The proof of Thms.~\ref{thm:main} and \ref{thm:main2} is split into two big steps: in Sec.~\ref{sec:measures} we provide a formula of the form \eqref{eq:probability-main} with a kernel $K_t$ defined implicitly through a biorthogonalization problem, while in Sec.~\ref{sec:biorth} we solve the biorthogonalization problem and obtain an explicit formula for the resulting kernel.
Sec. \ref{sec:biorth} works in a more general setting; the main result is Thm. \ref{thm:kernel-rw}, which is then applied in Sec. \ref{sec:main-proof} to prove Thms.~\ref{thm:main} and \ref{thm:main2}.

\noindent In Appdx.~\ref{app:convolutions} we prove several generalizations of the Cauchy-Binet identity, which yield determinantal formulas for convolutions of determinants which may be of different sizes. Appdx.~\ref{app:biorth} contains the derivation of the biorthogonal ensemble from Sec.~\ref{sec:measures}.
In Appdcs.~\ref{sec:rightBernoulli-assumptions} and \ref{sec:rGeometric-proof} we show that the variants of TASEP with right Bernoulli and geometric jumps satisfy Assum.~\ref{a:kappa}. In Appdx.~\ref{app:DW} we rewrite the formulas from \cite{MR2469339} for discrete time variants of TASEP in the form \eqref{eq:G-main}. 

\section{Discrete time variants of TASEP and their generalizations}
\label{sec:caterpillars}

Determinantal formulas of the form \eqref{eq:G-main} for the transition probabilities of a particle system were first derived by \citet{MR1468391} for TASEP in continuous time using the coordinate Bethe ansatz.
Similar formulas were later derived in a similar way for discrete time TASEP with sequential \cite{Brankov} (see also \cite{Rakos2005}) and parallel update \cite{Povolotsky_2006} and for several other models.
The same type of formulas arise for non-colliding Markov processes \cite{karlinMcGregor}.
In some cases, and for a special choice of initial data, such processes can be coupled with interacting particle systems via the Robinson-Schensted-Knuth (RSK) correspondence (an alternative coupling through a process on triangular arrays exists in essentially the same cases \cite{warren,Anisotropic}, see Sec.~\ref{sec:previous} for a brief discussion).

In \cite{MR2469339} the authors described the RSK coupling for four discrete time particle systems with different transition and interaction rules, using the four known variants of the RSK correspondence: the RSK and Burge algorithms, as well as their dual variants.
Using intertwining of transition kernels, this allowed them to compute Sch\"{u}tz-type formulas for these four models, corresponding to discrete time TASEPs with blocking and pushing dynamics, and with Bernoulli and geometric jumps.
In their formulas the functions appearing inside the determinant in \eqref{eq:G-main} are written in terms of certain sums involving symmetric polynomials, but, as we show in Appdx.~\ref{app:DW}, they have equivalent expressions in terms of contour integral formulas like \eqref{eq:F-main} (alternatively, one could proceed along the lines of Sch\"{u}tz's derivation and prove directly that the resulting determinants solve the Kolmogorov forward equation for each model; in the particular case of Bernoulli jumps, block dynamics are addressed in \cite{Brankov} while push dynamics follow by adapting the continuous time proof for PushASEP given in \cite{bp-push}). 

In this section we introduce an extension of each of these models, which for the first three we call systems of interacting caterpillars, and state explicit Fredholm determinant formulas for their multipoint distributions.
Caterpillars of length $1$ yield the basic models mentioned above.
Caterpillars of length $2$ correspond to the parallel TASEP model studied in \cite{borodFerSas} in the case of Bernoulli jumps with block dynamics, but seem to not have been studied before for the other systems (though they appear implicitly in \cite{Anisotropic}).
For caterpillar lengths larger than $2$ the models appear to be new; the ``heads'' of the caterpillars evolve as Markov chains with memory length larger than $1$.
The fourth model, corresponding to geometric jumps with block dynamics, is different from the other three; in that case our extension is from the basic rule, which is parallel update for this model, to sequential update.
In all cases we consider only the situation when all particles have equal speeds.
See Sec.~\ref{sec:previous} for some connections to earlier work and an additional discussion.

We remark that in \cite{arai}, the author considered TASEP with right Bernoulli and geometric jumps, for which he derived formulas of the type \eqref{eq:probability-main} following the blueprint of \cite{fixedpt} and showed pointwise convergence of the resulting kernels to those appearing in the KPZ fixed point formulas.
He also derived a formula for a mixture of the two dynamics and continuous time TASEP; those formulas can be derived too as an application of our results and the general fact that certain mixtures of TASEP-like particle systems yield again formulas \eqref{eq:G-main} (see Sec.~\ref{sec:KPZfixedpt-random} and Appdx.~\ref{app:convolutions}).

Throughout the section, for $p\in(0,1)$ we always write $q=1-p$.

\subsection{Right Bernoulli jumps}
\label{sec:RB_caterpillars}

Consider the evolution of $N$ particles $X^{\rB}_t$ on $\Omega_N$, where to go from time $t$ to time $t+1$, particles are updated sequentially from right to left as follows \cite{Brankov}: the $k^{\text{th}}$ particle jumps to the right with probability $p \in (0,1)$ and stays put with probability $q = 1 - p$, but if a particle tries to jump on top of an occupied site, the transition is blocked.
Note that a particle trying to jump at time $t$ is blocked by the position of its right neighbor at time $t+1$.
The evolution of the particle system $X^{\rB}_t \in \Omega_N$ can be written as $X^{\rB}_{t+1}(1) = X^{\rB}_{t}(1) + \xi(t+1,1)$ and
\begin{equation}\label{eq:rB-rec}
X^{\rB}_{t+1}(k) = \min \{X^{\rB}_t(k) + \xi(t+1,k), X^{\rB}_{t+1}(k-1) - 1\}, \qquad k = 2, \dotsc, N,
\end{equation}
where $\xi(t,k)$ are independent Ber[$p$] random variables. 
The transition probabilities of $X^{\rB}_t$ are given by (see \cite{Brankov} and Appdx.~\ref{sec:Ber-TASEP})
\begin{equation}\label{eq:rightBernoulliBlock}
 \pp (X^{\rB}_t = \vec x | X^{\rB}_0 = \vec y) = \det \bigl[F^{\rB}_{i - j}(x_{N + 1 - i} - y_{N + 1 - j}, t)\bigr]_{i, j \in \set{N}},
\end{equation}
where $\vec x, \vec y \in \Omega_N$, $t \in \nn_0$, and
\begin{equation}\label{eq:rBerBlockF}
	F^{\rB}_{n}(x, t) = \frac{1}{2\pi\I}\oint_{\gamma}\!\d w\, \frac{(w - 1)^{-n}}{w^{x - n +1}} (q + p w)^t,
\end{equation}
where the contour $\gamma$ includes $0$ and $1$. One can readily see that the model \eqref{eq:rightBernoulliBlock} satisfies Assum.~\ref{a:phi} with the function $\varphi(w) = q + p w$.
In particular, Thm.~\ref{thm:main} can be applied for $\kappa=0$, giving that the distribution function is given by formula \eqref{eq:probability-main-kappa=0} with this choice of $\varphi$.
The result is stated explicitly below in the more general setting of Prop.~\ref{prop:caterpillars}, corresponding to $\kappa \geq 0$.

\subsubsection{Caterpillars}
\label{sec:caterpillars-rB}

Now we describe the extension of discrete time TASEP with Bernoulli jumps to a system of interacting caterpillars.
A \emph{(forward) caterpillar} of length $L\geq 1$ is an element $\xx$ of the space
\begin{equation}\label{eq:catp-sp}
\cK^{\ra}_L=\big\{(\xx^1,\dotsc,\xx^L)\in\zz^L\!: \xx^{i}-\xx^{i+1}\in\{0,1\},\,i \in \set{L-1}\big\}.
\end{equation}
Our system of $N$ interacting caterpillars of length $L$ will take values in the space
\begin{equation}\label{eq:catp-syst-sp}
\Omega^{\catp}_{N,L}=\big\{\xx = (\xx(1), \dotsc, \xx(N))\in(\cK^{\ra}_L)^N\!:\xx^1(i+1) < \xx^L(i),\,i \in \set{N-1}\big\}
\end{equation}
(i.e., configurations in $\Omega^{\catp}_{N,L}$ are such that no two caterpillars overlap).
Fig.~\ref{fig:caterpillars} depicts possible configurations of $N=4$ caterpillars of length $3$. 

For $\xx\in\Omega^{\catp}_{N,L}$ we call $\xx^1(i)$ and $\xx^L(i)$ the \emph{head} and \emph{tail} of the $i^{\text{th}}$ caterpillar respectively, and we define $\xx^{\head} = (\xx^1(i) : i \in \set{N}) \in \Omega_N$ to be the vector of heads, which can be thought of as the location of $N$ particles located at the sites $\xx^1(i)$ for $i \in \set{N}$.
Note that in the case of caterpillars of length $1$, $\Omega^{\catp}_{N, 1}$ becomes just the usual Weyl chamber $\Omega_N$, i.e., $\xx$ becomes just a configuration of $N$ particles on $\zz$ whose locations are strictly decreasing.

We define now the dynamics on caterpillars $\xx_t \in \Omega^{\catp}_{N, L}$, $t\in\nn_0$, associated to discrete time TASEP with right Bernoulli jumps.
The transition from time $t$ to time $t+1$ occurs in the following way, with the positions of the caterpillars being updated consecutively for $i\in\set{N}$ (i.e., from right to left):
\begin{itemize}[leftmargin=1.5em]
\item The head of the $i^\text{th}$ caterpillar makes a unit step to the right with probability $p \in (0,1)$ (i.e., $\xx^1_{t+1}(i) = \xx^1_t(i) + 1$), provided that the destination site is empty.
Otherwise it stays put (i.e., $\xx^1_{t+1}(i) = \xx^1_t(i)$).
\item The remaining sections of the $i^\text{th}$ caterpillar move according to $\xx^j_{t+1}(i) = \xx^{j-1}_t(i)$, $j = 2, \dotsc, L$.
\end{itemize}
In words, the heads jump as in TASEP with right Bernoulli jumps, but are blocked by the whole caterpillar to its right, while each of the remaining sections of each caterpillar follows the movement of the section to its right in the previous time step.
One sees directly that the new configuration $\xx_{t+1}$ is again in $\Omega^{\catp}_{N, L}$ and that this choice of dynamics defines a Markov chain on $\Omega^{\catp}_{N, L}$. 
An example of such an update is provided in Fig.~\ref{fig:caterpillars}.

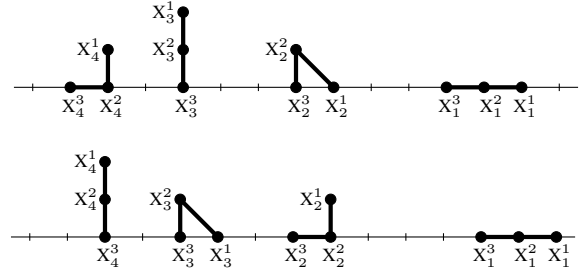
\begin{figure}[t]
\centering
\begin{tikzpicture}
  \draw (0.25, 0) -- (7.75, 0);
  \foreach \x in {1,...,15} {
    \draw (0.5 * \x, -0.05) -- (0.5 * \x, 0.05);
  }
  \draw[fill=black] (1,0) circle (0.07);
  \node at (1.05,-0.25) {\tiny $\text{X}_4^3$};
  \draw[fill=black] (1.5,0) circle (0.07);
  \node at (1.55,-0.25) {\tiny $\text{X}_4^2$};
  \draw[fill=black] (1.5,0.5) circle (0.07);
  \node at (1.25,0.5) {\tiny $\text{X}_4^1$};
  \draw[ultra thick] (1, 0) -- (1.5, 0);
  \draw[ultra thick] (1.5, 0) -- (1.5, 0.5);
  \draw[fill=black] (2.5,0) circle (0.07);
  \node at (2.55,-0.25) {\tiny $\text{X}_3^3$};
  \draw[fill=black] (2.5,0.5) circle (0.07);
  \node at (2.25,0.5) {\tiny $\text{X}_3^2$};
  \draw[fill=black] (2.5,1) circle (0.07);
  \node at (2.25,1) {\tiny $\text{X}_3^1$};
  \draw[ultra thick] (2.5, 0) -- (2.5, 1);
  \draw[fill=black] (4,0) circle (0.07);
  \node at (4.05,-0.25) {\tiny $\text{X}_2^3$};
  \draw[fill=black] (4,0.5) circle (0.07);
  \node at (3.75,0.5) {\tiny $\text{X}_2^2$};
  \draw[fill=black] (4.5,0) circle (0.07);
  \node at (4.55,-0.25) {\tiny $\text{X}_2^1$};
  \draw[ultra thick] (4, 0) -- (4, 0.5);
  \draw[ultra thick] (4, 0.5) -- (4.5, 0);
  \draw[fill=black] (7,0) circle (0.07);
  \node at (7.05,-0.25) {\tiny $\text{X}_1^1$};
  \draw[fill=black] (6.5,0) circle (0.07);
  \node at (6.6,-0.25) {\tiny $\text{X}_1^2$};
  \draw[fill=black] (6,0) circle (0.07);
  \node at (6.05,-0.25) {\tiny $\text{X}_1^3$};
  \draw[ultra thick] (7, 0) -- (6, 0);
\end{tikzpicture}

\medskip
\begin{tikzpicture}
  \draw (0.25, 0) -- (7.75, 0);
  \foreach \x in {1,...,15} {
    \draw (0.5 * \x, -0.05) -- (0.5 * \x, 0.05);
  }
  \draw[fill=black] (1.5,0) circle (0.07);
  \node at (1.55,-0.25) {\tiny $\text{X}_4^3$};
  \draw[fill=black] (1.5,0.5) circle (0.07);
  \node at (1.25,0.5) {\tiny $\text{X}_4^2$};
  \draw[fill=black] (1.5,1) circle (0.07);
  \node at (1.25,1) {\tiny $\text{X}_4^1$};
  \draw[ultra thick] (1.5, 0) -- (1.5, 0.5);
  \draw[ultra thick] (1.5, 0.5) -- (1.5, 1);
  \draw[fill=black] (2.5,0) circle (0.07);
  \node at (2.55,-0.25) {\tiny $\text{X}_3^3$};
  \draw[fill=black] (2.5,0.5) circle (0.07);
  \node at (2.25,0.5) {\tiny $\text{X}_3^2$};
  \draw[fill=black] (3,0) circle (0.07);
  \node at (3.05,-0.25) {\tiny $\text{X}_3^1$};
  \draw[ultra thick] (2.5, 0.5) -- (3, 0);
  \draw[ultra thick] (2.5, 0) -- (2.5, 0.5);
  \draw[fill=black] (4,0) circle (0.07);
  \node at (4.05,-0.25) {\tiny $\text{X}_2^3$};
  \draw[fill=black] (4.5,0) circle (0.07);
  \node at (4.55,-0.25) {\tiny $\text{X}_2^2$};
  \draw[fill=black] (4.5,0.5) circle (0.07);
  \node at (4.25,0.5) {\tiny $\text{X}_2^1$};
  \draw[ultra thick] (4, 0) -- (4.5, 0);
  \draw[ultra thick] (4.5, 0) -- (4.5, 0.5);
  \draw[fill=black] (7.5,0) circle (0.07);
  \node at (7.55,-0.25) {\tiny $\text{X}_1^1$};
  \draw[fill=black] (7,0) circle (0.07);
  \node at (7.1,-0.25) {\tiny $\text{X}_1^2$};
  \draw[fill=black] (6.5,0) circle (0.07);
  \node at (6.55,-0.25) {\tiny $\text{X}_1^3$};
  \draw[ultra thick] (7.5, 0) -- (6.5, 0);
  \end{tikzpicture}
\caption{Two possible configurations of $N = 4$ caterpillars of lengths $L = 3$ (here we write for convenience $X^i_k$ instead of $X^i(k)$, and we draw the parts of the caterpillars, occupying the same site, above each other). Each caterpillar $X_k = (X_k^1, X_k^2, X_k^3)$ is an element of $\cK^{\protect\ra}_3$, so that the configuration of four caterpillars is an element of $\Omega^{\protect\catp}_{4,3}$. The bottom configuration is obtained from the one on the top by the described update rule, where the heads of the $2^\text{nd}$ and $4^\text{st}$ caterpillars stay put, while the $1^\text{st}$ and $3^\text{rd}$ make one step to the right.}
\label{fig:caterpillars}
\end{figure}

We will be interested in the evolution of the vector of heads $\xx^{\head}_t$.
For $L\geq2$, it evolves as a particle system with memory of length $L-1$: a particle trying to jump at time $t$ is blocked by the position of its right neighbor at time $t-L+1$.
As we will explain next, in the cases $L=1$ and $L=2$ the heads evolve as the well known versions of discrete time TASEP with Bernoulli jumps and either sequential or parallel update.

We will say that the system of caterpillars $\xx_t$ has \emph{initial condition} $\vec y \in \Omega_N$ if $\xx_0 \in \Omega^{\catp}_{N, L}$ is given by $\xx^1_0(k) = \dotsm = \xx^L_0(k) = y_k$ for each $k \in\ \set{N}$. 
With a little ambiguity, we will write in this case $\xx_0 = \vec y \in \Omega_N$.
We will only be interested in the case where $\vec y$ is in the set
\begin{equation}\label{eq:Omega-kappa}
\Omega_{N}(\kappa) = \{\vec x \in \Omega_N\!:x_{i - 1} - x_i \geq \kappa\vee1\uptext{ for }i = 2, \dotsc, N\}
\end{equation}
with $\kappa=L-1$.
We are interested in this type of initial data\footnote{For fixed $\vec y\in\Omega_N(L-1)$ there are other choices of initial data so that $\xx^{\head}_0=\vec y$ and which are equivalent to the one above, in the sense that the evolution of the heads (and of the other sections after time $L-1$) is the same, as can be checked directly from the definition of the process (for example one could take $\xx^1_0(k)=y_k$ and $\xx^i_0(k)=y_k-1$ for $i=2,\dotsc,L$).} because it ensures that each caterpillar will not feel the caterpillar to its right until time $L-1$, resolving any ambiguity in the evolution of the heads for small times.

The key to our analysis is the following simple relation between caterpillars and the model $\xx^{\rB}_t$.
In the result we need to consider the process $\xx^{\rB}_t$ with particles starting at different (negative) times.
As in Sec.~\ref{sec:main}, this corresponds to forcing particles to stay put for some time by conditioning on an event like the one in \eqref{eq:E-event}, but for simplicity we will omit this from the notation and simply say that particles start moving at some different prescribed times.

\begin{lem}\label{lem:TASEP_and_caterpillars}
For $\kappa \geq 0$, let the process $\xx^{\rB}_t$ start at initial times $\vT = (-(k-1) \kappa)_{k \in \set{N}}$ and at a configuration $\vec y \in \Omega_N(\kappa)$.
Let $L=\kappa+1$ and define a new process $\xx_t \in \Omega^{\catp}_{N, L}$ as follows: for each $i \in \set{L}$ and $k\in\set{N}$,
\begin{equation}
\xx^i_t(k) = \xx^{\rB}_{t - (k-1)\kappa-i+1}(k).
\end{equation}
Then $\xx$ is distributed as the system of interacting caterpillars of length $L$ described above, with initial condition $\vec y$.
\end{lem}

In words, for $\xx_t$ and $\xx^{\rB}_t$ as in the lemma, the head of the $k$-th caterpillar follows the trajectory of $\xx^{\rB}_{t - (k-1)\kappa}(k)$ while, for $i=2,\dotsc,L$, $\xx^i_t(k)$ equals the location of the same TASEP particle $i-1$ instants in the past.
The result can be readily obtained from the distributions of the two processes. Fig.~\ref{fig:TASEP} shows an example of possible trajectories of caterpillars and their map to trajectories of TASEP particles. As can be appreciated there, the assumption on $\vec y$ guarantees that particle $k$ can collide with particle $k-1$ only at times $t \geq T_{k-1}$.

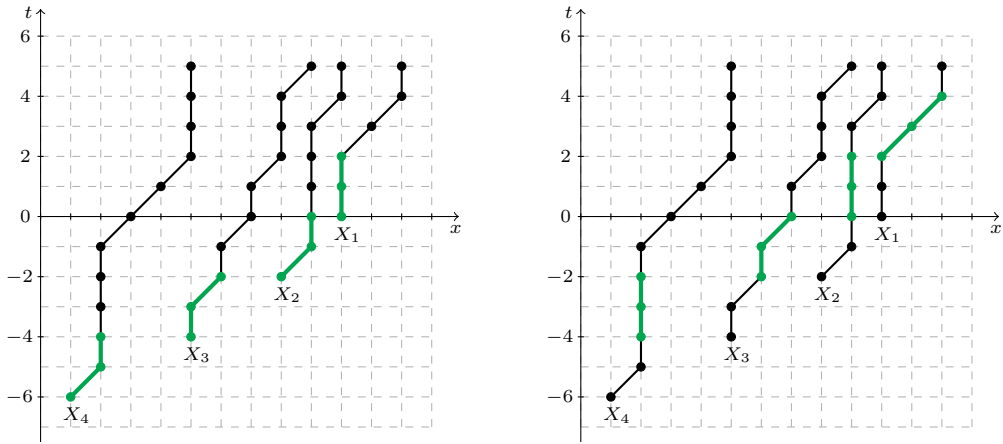
\begin{figure}[h]
\centering
\begin{tikzpicture}[scale=0.8]
	\foreach \x in {4,...,17} {
		\draw[black!30, dashed] (0.5 * \x, -3.5) -- (0.5 * \x, 3);
	}
	
	\foreach \y in {-7,...,6} {
		\draw[black!30, dashed] (2, 0.5 * \y) -- (8.5, 0.5 * \y);
	}

	\draw[->] (2, 0) -- (8.95, 0);
	\node at (8.9,-0.2) {\tiny$x$};
	\foreach \x in {4,...,17} {
		\draw (0.5 * \x, -0.05) -- (0.5 * \x, 0.05);
	}
	
	\draw[->] (2, -3.75) -- (2,3.45);
	\node at (1.8, 3.4) {\tiny$t$};
	\foreach \y in {-6, -4, -2, 0, 2, 4, 6} {
		\draw (1.95, 0.5 * \y) -- (2.05, 0.5 * \y);
		\ifthenelse{\y < 0}{\node at (1.65, 0.5 * \y) {\tiny $\y$};}{\node at (1.75, 0.5 * \y) {\tiny $\y$};};
	}

	\draw[fill=black] (7.5,1.5) circle (0.07);
	\draw[fill=black] (8,2) circle (0.07);
	\draw[fill=black] (8,2.5) circle (0.07);
	\draw[thick] (7, 0) -- (7, 0.5) -- (7, 1) -- (7.5,1.5) -- (8,2) -- (8,2.5);
	\draw[fill=Green, draw=Green] (7,0) circle (0.07);
	\draw[fill=Green, draw=Green] (7,0.5) circle (0.07);
	\draw[fill=Green, draw=Green] (7,1) circle (0.07);
	\draw[ultra thick, draw=Green] (7, 0) -- (7, 0.5) -- (7, 1);
	\node at (7.1,-0.3) {\tiny $X_1$};

	\draw[fill=black] (6.5,0.5) circle (0.07);
	\draw[fill=black] (6.5,1) circle (0.07);
	\draw[fill=black] (6.5,1.5) circle (0.07);
	\draw[fill=black] (7,2) circle (0.07);
	\draw[fill=black] (7,2.5) circle (0.07);
	\draw[thick] (6,-1) -- (6.5,-0.5) -- (6.5,0) -- (6.5,0.5) -- (6.5,1.5) -- (7,2) -- (7,2.5);
	\draw[fill=Green, draw=Green] (6,-1) circle (0.07);
	\draw[fill=Green, draw=Green] (6.5,-0.5) circle (0.07);
	\draw[fill=Green, draw=Green] (6.5,0) circle (0.07);
	\draw[ultra thick, draw=Green] (6,-1) -- (6.5,-0.5) -- (6.5,0);
	\node at (6.1,-1.3) {\tiny $X_2$};

	\draw[fill=black] (5,-0.5) circle (0.07);
	\draw[fill=black] (5.5,0) circle (0.07);
	\draw[fill=black] (5.5,0.5) circle (0.07);
	\draw[fill=black] (6,1) circle (0.07);
	\draw[fill=black] (6,1.5) circle (0.07);
	\draw[fill=black] (6,2) circle (0.07);
	\draw[fill=black] (6.5,2.5) circle (0.07);
	\draw[thick] (4.5,-2) -- (4.5,-1.5) -- (5,-1) -- (5,-0.5) -- (5.5,0) -- (5.5,0.5) -- (6,1) -- (6,1.5) -- (6,2) -- (6.5,2.5);
	\draw[fill=Green, draw=Green] (4.5,-2) circle (0.07);
	\draw[fill=Green, draw=Green] (4.5,-1.5) circle (0.07);
	\draw[fill=Green, draw=Green] (5,-1) circle (0.07);
	\draw[ultra thick, draw=Green] (4.5,-2) -- (4.5,-1.5) -- (5,-1);
	\node at (4.6,-2.3) {\tiny $X_3$};

	\foreach \y in {3,...,5} {
		\draw[fill=black] (3,-3 + 0.5 * \y) circle (0.07);
	}
	\draw[fill=black] (3.5,0) circle (0.07);
	\draw[fill=black] (4,0.5) circle (0.07);
	\foreach \y in {0,...,3} {
		\draw[fill=black] (4.5,1 + 0.5 * \y) circle (0.07);
	}
	\draw[thick] (2.5,-3) -- (3,-2.5) -- (3,-0.5) -- (4.5,1) -- (4.5,2.5);
	\draw[fill=Green, draw=Green] (2.5,-3) circle (0.07);
	\foreach \y in {1,2} {
		\draw[fill=Green, draw=Green] (3,-3 + 0.5 * \y) circle (0.07);
	}
	\draw[ultra thick, draw=Green] (2.5,-3) -- (3,-2.5) -- (3,-2);
	\node at (2.6,-3.3) {\tiny $X_4$};
\end{tikzpicture}
\qquad
\begin{tikzpicture}[scale=0.8]
	\foreach \x in {4,...,17} {
		\draw[black!30, dashed] (0.5 * \x, -3.5) -- (0.5 * \x, 3);
	}
	
	\foreach \y in {-7,...,6} {
		\draw[black!30, dashed] (2, 0.5 * \y) -- (8.5, 0.5 * \y);
	}

	\draw[->] (2, 0) -- (8.95, 0);
	\node at (8.9,-0.2) {\tiny$x$};
	\foreach \x in {4,...,17} {
		\draw (0.5 * \x, -0.05) -- (0.5 * \x, 0.05);
	}
	
	\draw[->] (2, -3.75) -- (2,3.45);
	\node at (1.8, 3.4) {\tiny$t$};
	\foreach \y in {-6, -4, -2, 0, 2, 4, 6} {
		\draw (1.95, 0.5 * \y) -- (2.05, 0.5 * \y);
		\ifthenelse{\y < 0}{\node at (1.65, 0.5 * \y) {\tiny $\y$};}{\node at (1.75, 0.5 * \y) {\tiny $\y$};};
	}

	\draw[fill=black] (7,0) circle (0.07);
	\draw[fill=black] (7,0.5) circle (0.07);
	\draw[fill=black] (8,2.5) circle (0.07);
	\draw[thick] (7, 0) -- (7, 0.5) -- (7, 1) -- (7.5,1.5) -- (8,2) -- (8,2.5);
	\draw[fill=Green, draw=Green] (7,1) circle (0.07);
	\draw[fill=Green, draw=Green] (7.5,1.5) circle (0.07);
	\draw[fill=Green, draw=Green] (8,2) circle (0.07);
	\draw[ultra thick, draw=Green] (7, 1) -- (7.5,1.5) -- (8,2);
	\node at (7.1,-0.3) {\tiny $X_1$};

	\draw[fill=black] (6,-1) circle (0.07);
	\draw[fill=black] (6.5,-0.5) circle (0.07);
	\draw[fill=black] (6.5,1.5) circle (0.07);
	\draw[fill=black] (7,2) circle (0.07);
	\draw[fill=black] (7,2.5) circle (0.07);
	\draw[thick] (6,-1) -- (6.5,-0.5) -- (6.5,0) -- (6.5,0.5) -- (6.5,1.5) -- (7,2) -- (7,2.5);
	\draw[fill=Green, draw=Green] (6.5,0) circle (0.07);
	\draw[fill=Green, draw=Green] (6.5,0.5) circle (0.07);
	\draw[fill=Green, draw=Green] (6.5,1) circle (0.07);
	\draw[ultra thick, draw=Green] (6.5,0) -- (6.5,0.5) -- (6.5,1);
	\node at (6.1,-1.3) {\tiny $X_2$};

	\draw[fill=black] (4.5,-2) circle (0.07);
	\draw[fill=black] (4.5,-1.5) circle (0.07);
	\draw[fill=black] (5.5,0.5) circle (0.07);
	\draw[fill=black] (6,1) circle (0.07);
	\draw[fill=black] (6,1.5) circle (0.07);
	\draw[fill=black] (6,2) circle (0.07);
	\draw[fill=black] (6.5,2.5) circle (0.07);
	\draw[thick] (4.5,-2) -- (4.5,-1.5) -- (5,-1) -- (5,-0.5) -- (5.5,0) -- (5.5,0.5) -- (6,1) -- (6,1.5) -- (6,2) -- (6.5,2.5);
	\draw[fill=Green, draw=Green] (5,-1) circle (0.07);
	\draw[fill=Green, draw=Green] (5,-0.5) circle (0.07);
	\draw[fill=Green, draw=Green] (5.5,0) circle (0.07);
	\draw[ultra thick, draw=Green] (5,-1) -- (5,-0.5) -- (5.5,0);
	\node at (4.6,-2.3) {\tiny $X_3$};
		
	\draw[fill=black] (2.5,-3) circle (0.07);
	\foreach \y in {1,5} {
		\draw[fill=black] (3,-3 + 0.5 * \y) circle (0.07);
	}
	\draw[fill=black] (3.5,0) circle (0.07);
	\draw[fill=black] (4,0.5) circle (0.07);
	\foreach \y in {0,...,3} {
		\draw[fill=black] (4.5,1 + 0.5 * \y) circle (0.07);
	}
	\draw[thick] (2.5,-3) -- (3,-2.5) -- (3,-0.5) -- (4.5,1) -- (4.5,2.5);
	\foreach \y in {2,...,4} {
		\draw[fill=Green, draw=Green] (3,-3 + 0.5 * \y) circle (0.07);
	}
	\draw[ultra thick, draw=Green] (3,-2) -- (3,-1);
	\node at (2.6,-3.3) {\tiny $X_4$};
\end{tikzpicture}
\caption{Possible trajectories of $N = 4$ particles of $\xx^{\rB}$ with starting times $T_k = - 2 (k-1)$. The initial configuration is $\vec y = (10, 8, 5, 1)$ and the final time is $t = 5$. The green segments are the respective locations of the caterpillars of lengths $L = 3$ at time $t = 2$ (on the left) and at time $t = 4$ (on the right).}
\label{fig:TASEP}
\end{figure}

Considering initial conditions $\vec y\in\Omega_{N}(\kappa)$ for $\xx^{\rB}_t$ ensures also that Assum.~\ref{a:kappa} is satisfied; proving \eqref{eq:second-property} is a bit involved, we do it in Lem.~\ref{lem:G-formula}.
Moreover, if $p<1/2$ then in the definition of the function \eqref{eq:F-main} we can take $\rhoout < q/p$ so that the singularity of the integrand at $w=-q/p$ is outside the contour, and under this additional restriction it turns out that \eqref{eq:backward-in-time-one} also holds; we prove this in Lem.~\ref{lem:rB-back-in-time}. 

\subsubsection{Caterpillars of length $L=1$ and $L=2$: sequential and parallel update}
\label{sec:Ber_seq}

In the case $L = 1$ the process $\xx^\head_t \in \Omega_N$ evolves as the usual discrete time TASEP with Bernoulli jumps, blocking, and sequential update described at the beginning of Sec. \ref{sec:RB_caterpillars}.

In the case $L = 2$, let us denote $X^{\prll}_t = \xx^\head_t$.
The definition of the system of caterpillars means that when $\xx^1_t(j)$ tries to jump to the right, it is blocked by $\xx^2_{t+1}(j-1)$.
Moreover, we have $\xx^2_{t+1}(j) = \xx^1_t(j) = X^{\prll}_t(j)$.
Hence, the evolution of $X^{\prll}_t$ is given by $X^{\prll}_{t+1}(1) = X^{\prll}_t(1) + \xi(t+1,1)$ and 
\begin{equation}
X^{\prll}_{t+1}(k) = \min \{X^{\prll}_t(k) + \xi(t+1,k), X^{\prll}_t(k-1) - 1\}, \qquad k = 2, \dotsc, N,
\end{equation}
where $\xi(t,k)$ are independent Ber[$p$] random variables.
Then the evolution of $X^{\prll}_t$ coincides is that of \emph{discrete time TASEP with right Bernoulli jumps, blocking and parallel update} \cite{Povolotsky_2006,borodFerSas}, which evolves in the same way as the model with sequential update corresponding to $L=1$  except that the transition of the particle $X^{\prll}(i)$ from time $t$ to $t+1$ is blocked by the particle $X^{\prll}(i+1)$ at time $t$, rather than $t+1$. Equivalently, all particles attempt to jump at the same time, but get blocked by the current location of the particles to their right, which is why the update is said to occur in parallel.
The representation of parallel TASEP as sequential TASEP with different starting times appears also in \cite{Anisotropic,bf-tilings} in the setting of a Markov chain on Gelfand-Tsetlin patterns (see Sec.~\ref{sec:previous}).

An explicit formula for the transition probabilities for this model can be given \cite[Lem.~10]{borodFerSas}:

\begin{lem}\label{lem:Ber-TASEP_prll}
The transition probabilities of the model with right Ber[$p$] jumps, blocking and parallel update are
\begin{equation}\label{eq:ProbBernoulliBlockY}
 \pp (X^{\prll}_t = \vec x | X^{\prll}_0 = \vec y) = q^{\CN(\vec x)} \det \bigl[F^{\rB}_{i - j}(x_{N + 1 - i} - y_{N + 1 - j}, t + i - j)\bigr]_{i, j \in \set{N}},
\end{equation}
where $\vec x, \vec y \in \Omega_N$, $t \in \nn_0$, $\CN(\vec x) = \# \{ 2 \leq i \leq N : x_{i-1} - x_i = 1\}$ and $F^{\rB}$ is defined in \eqref{eq:rBerBlockF} with a contour $\gamma$ which includes $0$ and $1$ but not the singularity at $-q/p$.
\end{lem}

 In the case $L > 2$ we do not expect to have determinantal formulas of this type describing transition distributions.
Note that when $p\geq1/2$, the function inside the determinant in the lemma cannot be written in the form \eqref{eq:F-main} used in the definition of the Markov chains studied in Sec.~\ref{sec:main}, because, since in this case $q/p\leq1$, the contour $\gamma$ cannot be chosen to be a circle.
We will not use \eqref{eq:ProbBernoulliBlockY} in the sequel, but the same issue will play a role in the next proposition.

\subsubsection{Distribution function for caterpillar heads}
\label{sec:distr_seq}

Finally we use the results from the Introduction to derive a formula for the distribution function of the heads of the caterpillars (of any length $L\geq1$).
For $L=1$ the result is a consequence of \eqref{eq:rightBernoulliBlock} and Thm.~\ref{thm:main}.
In the general case it follows from Lem.~\ref{lem:TASEP_and_caterpillars} and Thm. \ref{thm:main2}, together with an analytic continuation argument in the case $t\leq\kappa(n_m-1)$ and $p\geq1/2$.

\begin{prop}\label{prop:caterpillars}
Consider the system of caterpillars $\xx$ of length $L\geq 1$, and suppose that $\xx(0)=\vec y$ with $\vec y\in\Omega_N(L-1)$.
Then for any $t\geq0$, any $1 \leq n_1 < \dotsm <  n_m \leq N$, and any $\vec a \in \rr^m$, we have
\begin{equation}\label{eq:caterpillars}
\pp \bigl(\xx^{\head}_t(n_i) > a_{i},\, i \in \set{m}\bigr) = \det \bigl(I-\bP_{a}  K^{\rB} \bP_{a} \bigr)_{\ell^2(\{n_1, \dotsc, n_m\} \times \zz)},
\end{equation}
where the kernel $K^{\rB}$ is given by \eqref{eq:kernel-main} defined using $\varphi(w) = q + pw$ and $\kappa = L - 1$ and where, in the definition of the kernels \eqref{eq:defQintro}--\eqref{eq:defSNintro}, $\theta\in(0,1)$ is arbitrary while $r\in(0,1)$ is also arbitrary unless $\kappa\geq1$, $p\geq1/2$ and $t<\kappa(n_m-1)$, in which case $r$ has to be bounded above by $q/p$.
The random walk used to define \eqref{eq:SMepi} in this case has transition matrix $Q(x,y)=\frac{1-\theta}{(q+p\theta)^\kappa}\theta^{x-y-1}q_{x-y}\uno{x>y}$ with $q_i=1-\sum_{j=i}^\kappa\binom{\kappa}{j}p^jq^{\kappa-j}$ ($q_i=1$ for $i>\kappa$). 
\end{prop}

\begin{proof}
By Lem.~\ref{lem:TASEP_and_caterpillars}, $\xx^{\head}_t(n_i)=\xx^1_t(n_i)\distr\xx^{\rB}_{t-(n_i-1)\kappa}(n_i)$, so the probability on the left hand side of \eqref{eq:caterpillars} can be computed using Thm.~\ref{thm:main2}.
The transition probabilities \eqref{eq:rightBernoulliBlock} of the system $\xx^{\rB}_t$ correspond to \eqref{eq:G-main}/\eqref{eq:F-main} with $\varphi(w) = q + pw$, which clearly satisfies Assum.~\ref{a:phi} with any $0<\rhoin<1<\rhoout$, while Assum.~\ref{a:kappa} is also satisfied in the case $\kappa\geq1$ (condition \eqref{eq:second-property} is proved in Appdx.~\ref{sec:kappaone}).
This leads to the desired formula (with arbitrary choices $r,\theta\in(0,1)$) in the case $t\geq\kappa(n_m-1)$.

Next we extend the formula to all $t\geq0$ in the case $\kappa\geq1$.
If $p<1/2$ then, by Lem.~\ref{lem:rB-back-in-time}, \eqref{eq:backward-in-time-one} holds for the model with the function $F_{i-j}$ defined as in the above paragraph, and thus Thm.~\ref{thm:main2} implies again the desired formula with arbitrary parameter choices.
Crucially, in that case we have $q/p>1$, so $\rhoout>1$ can be chosen in \eqref{eq:F-main} so that the singularity of $\varphi(w)$ at $w=-q/p$ is outside the contour $\gamma_{\rhoout}$, as required by Lem.~\ref{lem:rB-back-in-time}.
In the case $p\geq1/2$, however, this is not possible, so the extension to all $t\geq0$ provided in Thm.~\ref{thm:main2} does not apply, and thus we need a different argument.

Take $p=\bar p\geq1/2$, $\kappa\geq1$.
The left hand side of \eqref{eq:caterpillars} defines a function of the parameter $p$ which is real analytic for $p\in(0,1)$ (this can be proved for instance using \eqref{eq:second-property} with the current choice of $F_n$). We claim that for fixed (small) $\ep>0$, the right hand side is also real analytic in $p\in(0,\bar p+\ep)$ as long as the radius $r$ used to define $\SM_{-t,-n}$ satisfies $r<(1-\bar p-\ep)/(\bar p+\ep)$.
To see this, note that the kernels \eqref{eq:defQintro}--\eqref{eq:defSNintro}\noeqref{eq:defSMintro} are all real analytic in $p\in(0,\bar p+\ep)$ because no singularities of the integrand are crossed as $p$ is moved along this interval under the stated condition on $r$.
Moreover, in $(\SM_{-t, -n_i})^*\SN^{\epi(\vec y)}_{-t, n_j}(z_1,z_2)=\sum_{y\in\zz}(\SM_{-t, -n_i})^*(z_1,y)\SN^{\epi(\vec y)}_{-t, n_j}(y,z_2)$ the sum is actually finite: for $y\gg1$ the first factor vanishes because the integrand in \eqref{eq:defSMintro} has no pole at $0$, while for $y\ll1$ the second factor vanishes because the random walk $B_m$ in \eqref{eq:SMepi} takes only negative steps and so it cannot hit $\epi(\vec y)$ by time $n$.
Then the kernel $K^{\rB}(z_1,z_2)$ is analytic in $p\in(0,\bar p+\ep)$, and a standard argument (e.g. using a Hadamard bound to show uniform convergence of the Fredholm series) shows that the Fredholm determinant also is so.
Since the two sides are real analytic in $p\in(0,\bar p+\ep)$ and are equal for $p\in(0,1/2)$, we deduce they are equal also at $p=\bar p$.
This gives the result for this value of $p$, with the restriction on $r<(1-\bar p-\ep)/(\bar p+\ep)$, and since $\ep>0$ is arbitrary we recover the restriction $r<q/p$.
\end{proof}

In the next two examples we will focus mostly, for simplicity, on the one-point kernel 
\[K^{(n)}(z_1,z_2)\coloneqq K^{\rB}(n,z_1;n,z_2).\]

\begin{ex}\label{ex:periodic-one-sided} {\bf (Half-periodic initial data)}
\enspace For a fixed $L \geq 1$, let $\kappa = L - 1$ as before, $d \geq \kappa \vee 1$, and let the initial state be $\vec y \in \Omega_N$ such that $y_i = - d i$ for each $i$. 

\noindent In the case $d = 1$ (which means necessarily $L = 1$ or $L=2$ due to the restriction $d \geq \kappa \vee 1$), if $B_0 \leq -1$, then the random walk never hits the epigraph of $\vec y$.
This means that $\SM^{\epi(\vec y)}_{-t,n}=\P_{-1}\SM_{t,-n}$, and thus the second term in \eqref{eq:kernel-main} can be computed using \eqref{def:sm} and \eqref{def:sn}:
\begin{equation}\label{eq:stepRB}
(\SM_{-{n}_i})^*\P_{-1}\SN_{{n}_j}(x_1, x_2) = \frac{\theta^{x_1 - x_2}}{(2\pi\I)^2} \oint_{\gamma_r}\d w \oint_{\gamma_\delta}\d v\, \frac{(1-w)^{n_i} (1-v)^{n_2 + x_2}}{w^{n_i + x_1 +1} v^{n_j} (1 - w - v)} \frac{(q + p w)^{t - \kappa(n_i - 1)}}{(1 - p v)^{t - \kappa(n_j - 1)}},
\end{equation}
This gives the kernels for the models in the case of \emph{step} (or \emph{packed}) \emph{initial condition} $X_0(i)=-i$, $i\geq1$.

\noindent Let us now consider the case $d \geq 2$, with $n_i=n_j=n$.
When $d=2$ the kernel $\SM^{\epi(\vec y)}_{-t,n}$ can be computed directly as in \cite[Ex. 2.9]{fixedpt}, but in the general case it turns out to be easier to do it using an equivalent description, provided in Sec.~\ref{sec:biorth}, by computing the functions \eqref{eq:h_heat_Q} explicitly and then using \eqref{eq:Kn}.
In the notation of Sec.~\ref{sec:biorth} we have $a(w) = (q + p w)^\kappa$ and $\psi(w) = (q + p w)^t$.
Using \eqref{eq:contourQ} we can compute the logarithm of the moment generating function of $B^*_1$ as \mbox{$\phi^*(\lambda) = \log \bigl(\alpha \theta e^\lambda \frac{(q + p \theta e^{\lambda})^\kappa}{1 - \theta e^{\lambda}}\bigr)$} for $\lambda < - \log \theta$. Then \mbox{$\bigl(e^{\lambda B^*_m - m \phi^*(\lambda)}\bigr)_{m \geq 0}$} is a martingale, which yields \mbox{$\ee_{B^*_{-1}=z}\bigl[e^{\lambda B^*_{\tau^*} - \tau^* \phi^*(\lambda)}\bigr] = e^{\lambda z + \phi^*(\lambda)}$}, where $\tau^*$ is the hitting time of the strict epigraph of $(y_{n-m})_{m = 0, \ldots, n-1}$.
This identity can be used to compute the distribution of $\tau^*$ in the same way as in \cite[Ex.~2.10]{fixedpt}, and leads to the following formula for the function defined in \eqref{eq:pnkrw} in the case $z \leq y_n$:
\begin{equation}
\textstyle p^n_k(0,z) = \frac{\alpha^{k+1} \theta^{d(n-k) - z + 1}}{2 (1 - \theta) \pi\I}\oint_{\gamma_r} \d u\, \frac{(1 - p u)^{\kappa (k+1)-1} (1-u)^{z + d (n-k) + k - 1}}{u^{k+1}} \bigl(1 + (p \kappa - p - d)u + p(d - \kappa) u^2\bigr).
\end{equation}
This expression is analytic in $z$ and we can extend it to all $z \in \zz$. Then using Thm.~\ref{thm:h_heat_Q} and formulas \eqref{eq:defR}, \eqref{eq:contourQ}, and \eqref{eq:Apow}, we can compute $\Phi^n_k(x)=\frac{1 - \theta}{\theta} (\R^*)^{-1}(A^*)^{-1}\bar p^n_k(0,x)$ as
\begin{equation}
\textstyle\Phi^n_k(x) = \frac{\alpha^k \theta^{- d(n-k) - x}}{2 \pi\I}\oint_{\gamma_r} \d u\, \frac{(1-u)^{x + d n - 1}}{u (1 - p u)^{t + 1}} \left( \frac{(1 - pu)^\kappa}{u (1-u)^{d - 1}} \right)^{k} \bigl(1 + (p \kappa - p - d)u + p(d - \kappa) u^2\bigr).
\end{equation}
The function \eqref{eq:defPsink} on the other hand equals $\textstyle\Psi^n_k(x)=\frac{\alpha^{-k}}{2\pi\I}\oint_{\gamma_{\rin'}}\d w\,\frac{\theta^{x + d (n-k)} (q + pw)^t}{w^{x + d n + 1}} \left( \frac{w^{d - 1} (1 - w)}{(q + pw)^\kappa} \right)^{k}$, for $r' > 0$. 
Define now $g(w) = \frac{(q + pw)^\kappa}{w^{d - 1} (1 - w)}$.
Since clearly $\Phi^n_k(x) = 0$ for $k<n$ we may compute the sum in \eqref{eq:Kn} over all $k < n$; we get 
\begin{equation}
\textstyle K^{(n)}(z_1,z_2)=\textstyle \frac{\theta^{z_1 - z_2}}{(2\pi\I)^2} \oint_{\gamma_r} \d u \oint_{\gamma_{r'}}\d w\,\frac{(q + pw)^t}{ w^{z_1 + d n + 1}} \frac{(1-u)^{z_2 + d n - 1}}{u (1 - p u)^{t + 1}}\frac{(1 + (p \kappa - p - d)u + p(d - \kappa) u^2)g(1-u)^n}{g(w)^{n-1}(g(1-u) - g(w))}, \label{eq:kernel-in-example}
\end{equation}
where the contours are such that $|g(w)| < |g(1-u)|$ (for this we need to choose $r$ sufficiently small and $r' > r$ sufficiently large).
In the case $d = 1$ we recover the kernel from the beginning of this example (with $n_i=n_j=n$).
\end{ex}

\begin{rem}\label{rem:kappa-extension}
The restriction $d \geq \kappa \vee 1$ in the above example, coming from Prop. \ref{prop:caterpillars}, leaves out the important case of step initial data for caterpillars of length $L=\kappa+1\geq3$.
We believe (and have checked in a computer algebra system in simple cases), however, that the formula holds for all $\kappa$.
More precisely, we conjecture that the distribution of the heads of caterpillars of any length $L$ with step initial data is determined by \eqref{eq:probability-main} with the kernel $K_t$ in \eqref{eq:kernel-main} computed using \eqref{eq:stepRB} with $\kappa-L-1$.
\end{rem}

\begin{ex}{\bf (Periodic initial data)}\label{ex:TASEP-periodic}
\enspace Now we derive a kernel for caterpillars with infinite periodic initial state $X_0(i) = - d i$ for each $i \in \zz$ with for $d \geq 2$ (in the setting of Ex.~\ref{ex:periodic-one-sided}).
To this end we consider the initial state $\vec y \in \Omega_{2 N}$ given by $y_i = d(N-i)$ for each $i$ and focus on distribution of particles with indices $N+1, N+2, \dotsc, N + M$ for a fixed $M \leq N/2$. Then the respective kernel $K^{(n)}_{\per}$ is obtained from \eqref{eq:kernel-in-example} by $K^{(n)}_{\per}(z_1, z_2) = K^{(N + n)}(z_1 - d N, z_2 - d N)$. Proceeding as in \cite[Ex. 2.10]{fixedpt} one sees that for $N \geq \frac{M + a + 1}{d-1}$ the kernel becomes independent of $N$ and is given by
\begin{align}
\textstyle K^{(n)}_\per(z_1,z_2)&\textstyle=\frac{\theta^{z_1 - z_2}}{2\pi\I}\oint_{\gamma_r}\d u \sum_{i = 1}^{d-1} \frac{(q + p w_i(u))^{t - \kappa (n -1)} (1-w_i(u))^{n}}{w_i(u)^{z_1 + n + 1}} \frac{(1-u)^{z_2 + d + n - 2}}{u^{n} (1 - p u)^{t - \kappa n + 1}}  \\
&\hspace{5cm}\textstyle\times \frac{1 + (p \kappa - p - d)u + p(d - \kappa) u^2}{\partial_w f(u,w) |_{w = w_i(u)}},\label{eq:kernel-in-example2-final}
\end{align}
where $f(u,w) = (1-w) w^{d-1} (1-pu)^\kappa - u (1-u)^{d-1} (q+pw)^\kappa$ and $w_1(u), \ldots, w_{d-1}(u)$ denote the $d-1$ distinct solutions other than $w=1-u$ of $f(u,w)=0$ inside $\gamma_{r'}$ (with $r'$ from the previous example).
This gives the kernel for the $d$-periodic initial condition introduced in this example. For $L = 1$, the formula recovers the kernel derived in \cite[Thm.~2.1]{bfp}. In the case $L = d = 2$, the equation $f(u,w)=0$ has two solutions $w = 1 - u$ and $w = \frac{qu}{1-pu}$. In the formula for the kernel we consider only the latter and get
\[\textstyle K^{(n)}_\per(z_1,z_2) =\frac{\theta^{z_1 - z_2}}{(2\pi\I)^2}\oint_{\gamma_r}\d u\, \frac{q^{t - 2 n - z_1 (1-u)^{z_2 + 2 n}}}{u^{z_1 + 2n + 1} (1-pu)^{2t - 2n - z_1 + 1}}.\]
This is the kernel for parallel TASEP which was obtained in \cite[Thm.~1]{borodFerSas}.
\end{ex}

\subsubsection{TASEP with generalized update}
\label{sec:Povolotsky-discrete}

Finally we consider a version of discrete time TASEP with a more general dynamics, introduced in \cite{general-update}. In this new model $X^{\gen}(t)$, which takes values in $\Omega_N$, particles try to make right Bernoulli jumps independently with probability $p$ with the usual exclusion rule that jumps onto occupied sites are blocked.
As above the update of particles is from right to left. 
If at time $t$ it is the turn of the $i^{\text{th}}$ particle to jump and the configuration is such that $X^{\gen}_{i-1}(t)>X^{\gen}_i(t)+1$ then as usual the particle jumps with probability $p$, but if $X^{\gen}_{i - 1}(t) = X^{\gen}_i(t) + 1$ and $X^{\gen}_{i - 1}(t + 1) = X^{\gen}_{i - 1}(t) + 1$ then $X^{\gen}_i(t)$ jumps to the right with probability $\beta \in [0, 1)$ and stays put with probability $1 - \beta$.

In the case $\beta = p$ this model evolves as TASEP with sequential update, while in the case $\beta = 0$ it is TASEP with parallel update. Another interesting case is $\beta \to 1$, in which when a particle moves it pulls its left neighbor. 
In \cite[Eqn.~3.2]{general-update} it was proved that the transition probabilities for this model are
 \begin{equation}\label{eq:G-Povolotsky}
 	\pp (X^{\gen}_t = \vec x | X^{\gen}_0 = \vec y) = (\tfrac{q}{1-\beta})^{\CN(\vec x)} \det \bigl[F^{\gen}_{i - j}(x_{N + 1 - i} - y_{N + 1 - j}, t)\bigr]_{i, j \in \set{N}},
\end{equation}
where $\vec x, \vec y \in \Omega_N$, $t \in \nn_0$, the function $\CN(\vec x)$ is defined in Lem.~\ref{lem:Ber-TASEP_prll}, and
\begin{equation}\label{eq:Fgen}
F^{\gen}_{n}(x, t) = \frac{1}{2\pi\I}\oint_{\gamma}\frac{\d w}{w^{x - n + 1}} \Bigl( \frac{w-1}{1 - w (\beta - p) / q} \Bigr)^{-n} (q + p w)^t,
\end{equation}
with any contour $\gamma$ enclosing only the poles at $0$ and $1$.
While Thm.~\ref{thm:main} cannot be applied to this model, because the function $F^{\gen}_{n}$ is not quite of the form \eqref{eq:F-main}, the more general biorthogonalization result which we prove below, Thm.~\ref{thm:kernel-rw}, does apply.

In order to use Thm.~\ref{thm:kernel-rw}, one first applies the scheme of Sec. \ref{sec:measures} to show that
\begin{equation}\label{eq:genTASEP}
\pp_{\xx^{\gen}_0} \bigl(\xx^{\gen}_{t}(n_i) > a_{i},\, i \in \set{m}\bigr) = \det \bigl(I-\bP_{a}  K^{\gen} \bP_{a} \bigr)_{\ell^2(\{n_1, \dotsc, n_m\} \times \zz)}
\end{equation}
with $K^{\gen}$ a kernel of the form \eqref{eq:K-schutz} with $Q$, $\Psi^n_k$ and $\Phi^n_k$ defined using $a(w)=1 - w (\beta - p) / q$ and $\psi(w) = (q + pw)^t$ and with the choice $\vec y=X^{\gen}_{0}$ (note that, in the setting of Sec.~\ref{sec:biorth}, we have in this case $\kappa=1$).
In order for this to work, the contour $\gamma$ in \eqref{eq:Fgen} needs to be a circle of radius greater than $1$ (just as in \eqref{eq:F-main}).
Since it cannot include any poles of the integrand other than $0$ and $1$, we need the parameters of the model to satisfy $|\beta - p| < q$.
This condition holds if and only if $\beta\in[(2p-1)\vee0,1)$, so at this point we add this assumption (it could be lifted by an analytic continuation argument as in Prop.~\ref{prop:caterpillars}).
Then Thm.~\ref{thm:kernel-rw} shows that $K^{\gen}$ has the form \eqref{eq:Kn-RW} with the same choices where $\theta\in(0,1)$ is arbitrary while the radius $r$ in \eqref{eq:defQintro}--\eqref{eq:defSNintro} can be taken arbitrarily in $(0,1)$ when $\beta\in[p,1)$, while when $\beta\in[(2p-1)\vee0,p)$ it has to be bounded above by $q/(p - \beta)$.

\subsection{Left Bernoulli jumps}
\label{sec:LB}

In this model, to go from time $t$ to time $t+1$ particles are updated sequentially from right to left as follows: each particle jumps to the left with probability $p \in (0,1)$ and stays put with probability $q = 1 - p$ independently, except that particle $k$ is forced to jump if particle $k-1$ arrives on top of it, so that the configuration of particles stays in $\Omega_N$.
In other words, when a particle jumps on top of another one, it pushes it and its whole cluster of nearest neighbors one step to the left.
The model is often referred to as \emph{discrete time PushTASEP}.
Note that, analogously to the dynamics for right Bernoulli jumps, a given particle at time $t$ is pushed by the location of its right neighbor at time $t+1$.
The evolution of the particle system $X^{\lB}_t \in \Omega_N$ can be written as $ X^{\lB}_{t+1}(1) = X^{\lB}_{t}(1) - \xi(t+1,1)$ and
\begin{equation}
X^{\lB}_{t+1}(k) = \min \{X^{\lB}_{t}(k)  - \xi(t+1,k), X^{\lB}_{t+1}(k-1) - 1\}, \qquad k = 2, \dotsc, N,
\end{equation}
where $\xi(t,k)$ are independent Ber[$p$] random variables. The transition probabilities of the model are (see Appdx.~\ref{sec:Ber-pushTASEP})
 \begin{equation}\label{eq:Ber-pushTASEP}
 	\pp (X^{\lB}_t = \vec x | X^{\lB}_0 = \vec y) = \det \bigl[F^{\lB}_{i - j}(x_{N + 1 - i} - y_{N + 1 - j}, t)\bigr]_{i, j \in \set{N}},
\end{equation}
where $\vec x, \vec y \in \Omega_N$, $t \in \nn_0$ and (here the contour $\gamma$ encloses $0$ and $1$)
\begin{equation}\label{eq:lBerFunction}
	F^{\lB}_{n}(x, t) = \frac{1}{2\pi \I} \oint_{\gamma} \d w\,\frac{(w-1)^{-n}}{w^{x - n +1}} \left(q+ {p \over w}\right)^t.
\end{equation}
This corresponds to the setting of Thm.~\ref{thm:main} with $\varphi(w)=q+ \frac{p}{w}$, and one can check that Assum.~\ref{a:phi} is satisfied as needed.
We will write the distribution function for this model in a more general setting in Prop.~\ref{prop:caterpillars_BerPush}.
To that end we will in fact rely directly on the formulas for $X^{\rB}_t$, as we explain next (alternatively, one can proceed as in Sec.~\ref{sec:RB_caterpillars} and then apply Thms.~\ref{thm:main} and \ref{thm:main2} directly to this model).

There is a simple coupling which relates $X^{\lB}_t$ with the model of discrete time TASEP with blocking introduced in the preceding subsection.
Starting from the configuration of particles $X^{\lB}_t$ at time $t$, decrease each value $X^{\lB}_t(k)$ by $1$ and after that perform one step of Ber[$q$] TASEP with jumps to the right, where the update of particles is from right to left.
This will give us a configuration $X^{\lB}_{t+1}$, which is distributed as the one obtained after one step of Ber[$p$] TASEP with pushing.
In other words, we have 
\begin{equation}\label{eq:FromBlockToPush}
\xx^{\lB}_t \distr \bar{\xx}^{\rB}_t - t,
\end{equation}
where $\bar{\xx}^{\rB}$ is TASEP with right Ber[$q$] jumps, blocking and sequential update.
This identity is easy to prove directly, and provides an alternative proof of \eqref{eq:Ber-pushTASEP} (or alternatively, it follows from \eqref{eq:Ber-pushTASEP}).

\subsubsection{Caterpillars}
\label{sec:LB_distribution}

Caterpillars of length $L = 1$ are given by the model described above. 
We now construct caterpillars of lengths $L \geq 2$ by coupling them to the system of caterpillars from Sec.~\ref{sec:RB_caterpillars}, using \eqref{eq:FromBlockToPush} and Lem.~\ref{lem:TASEP_and_caterpillars}.
For $L \geq 2$ let $\bar \xx_t \in \Omega^{\catp}_{N, L}$ be a copy of the system of caterpillars defined in Sec.~\ref{sec:RB_caterpillars}, with jumps to the right occurring with probability $q$ (instead of $p$). Then, using Lem.~\ref{lem:TASEP_and_caterpillars} and the relation \eqref{eq:FromBlockToPush}, we define caterpillars with left Bernoulli jumps as 
\begin{equation}\label{eq:lBerCaterpillars}
\xx^{i}_{t}(k) = \bar \xx^{i}_{t}(k) - t + (k-1) \kappa + i - 1,
\end{equation}
where $\kappa = L-1$.
From the definition, $\xx_t(k)$ now lives in the space of \emph{backward caterpillars} $\cK^{\la}_L=\big\{(\xx^1,\dotsc,\xx^L)\in\zz^L\!: \xx^{i+1}-\xx^{i}\in\{0,1\},\,i \in \set{L-1}\big\}$, whose heads are to the left of their tails, and the system of caterpillars takes values in the space
\begin{equation}\label{eq:catp-syst-sp-left}
\Omega^{\catpl}_{N,L}=\big\{\xx = (\xx(1), \dotsc, \xx(N))\in(\cK^{\la}_L)^N\!:\xx^1(k+1) < \xx^L(k),\,k \in \set{N-1}\big\}.
\end{equation}
Note that in this new space, a site can be occupied by more than one caterpillar (in fact the head and intermediate sections of the $k^{\text{th}}$ caterpillar, but not its tail, can be on top or to the left of the head of the $(k+1)^{\text{th}}$ one). Moreover, the definition \eqref{eq:lBerCaterpillars} implies that if the initial state $\bar \xx_{0}$ is in $\Omega_N(\kappa)$, then the initial state $\xx_0$ is given by $\xx^1_0(k) = y_k + (k-1) \kappa$ and $\xx^i_0(k) = \xx^1_0(k) + i - 1$ for all $i$ and $k$.
This means that at time $0$ the caterpillars are stretched horizontally while the vector of heads $\xx^\head_0$ lives in
\begin{equation}\label{eq:OmegaBar}
\bar{\Omega}_N = \{(x_1,\dotsc,x_N)\in\zz^N\!:x_1 \geq x_2 \geq \dotsm \geq x_N\}.
\end{equation}
With a little ambiguity, we will write in this case $\xx_0 \in \bar \Omega_N$.

By definition of $X_t$ and \eqref{eq:rB-rec} the heads of the caterpillars evolve according to the equations $\xx^{1}_{{t+1}}(1) = \xx^{1}_{t}(1) - \xi(t+1, 1)$ and 
\begin{equation}
	\xx^{1}_{{t+1}}(k) = \min \{\xx^{1}_{t}(k) - \xi(t + 1, k), \xx^{L}_{{t+1}}({k-1}) - 1\}, \qquad k = 2, \dotsc, N,
\end{equation}
where the $\xi(t,k)$'s are independent Ber[$p$] random variables.
After the head jumps, the other parts are updated according to  $\xx^{i}_{{t+1}}(k) = \xx^{i-1}_{t}(k)$ for $i = 2, \dotsc, L$. More precisely, the transition from time $t$ to time $t+1$ is given as follows, with the positions of the caterpillars being updated consecutively for $k \in\ \set{N}$ (i.e., from right to left):
\begin{itemize}[leftmargin=1.5em]
	\item If $\xx^{1}_{t}({k}) \leq \xx_{t+1}^{{L}}({k-1}) - 1$, then the head of the $k^\text{th}$ caterpillar makes a unit step to the left with probability $p$. 
	\item If $\xx_{t}^{1}({k}) = \xx_{t+1}^{{L}}({k-1})$, then the head of the $k^\text{th}$ caterpillar makes a deterministic unit step to the left. 
\end{itemize}
After that we set $\xx_{t+1}^{i}(k) = \xx_{t}^{{i-1}}(k)$ for $i = 2, \dotsc, L$.
In words, the heads jump as in TASEP with left Bernoulli jumps but are pushed to the left by the tail of the caterpillar to their right, while each of the remaining sections of each caterpillar follows the movement of the section to its right in the previous time step.
We can reformulate Prop.~\ref{prop:caterpillars} using the transformation \eqref{eq:lBerCaterpillars} to get a formula for the distribution function of the heads of the caterpillars in this model:

\begin{prop}\label{prop:caterpillars_BerPush}
Consider the system of caterpillars $\xx$ of length $L\geq 1$, and with the initial state $\xx(0)=\vec y$, such that $\vec y \in \Omega_N$ if $L =1$ and $\vec y\in\bar{\Omega}_N$ if $L \geq 2$.
Then for any $t\geq0$, any $1 \leq n_1 < \dotsm <  n_m \leq N$, and any $\vec a \in \rr^m$, we have
\begin{equation}
\pp \bigl(\xx^{\head}_t(n_i) > a_{i},\, i \in \set{m}\bigr) = \det \bigl(I-\bP_{a}  K^{\lB} \bP_{a} \bigr)_{\ell^2(\{n_1, \dotsc, n_m\} \times \zz)},
\end{equation}
where the kernel $K^{\lB}$ is given by \eqref{eq:kernel-main} defined using $\varphi(w) = q + \frac{p}{w}$ and $\kappa = L - 1$ and where, in the definition of the kernels \eqref{eq:defQintro}--\eqref{eq:defSNintro}, $\theta\in(0,1)$ is arbitrary while $r\in(0,1)$ is also arbitrary unless $\kappa\geq1$, $p\leq1/2$ and $t<\kappa(N-1)$, in which case $r$ has to be bounded above by $p/q$.
The random walk used to define \eqref{eq:SMepi} in this case has transition matrix $Q(x,y)=\frac{1-\theta}{(q+p/\theta)^\kappa}\theta^{x-y-1}q_{x-y}\uno{x>y-\kappa}$ with $q_i=1-\sum_{j=i+\kappa}^\kappa\binom{\kappa}{j}p^jq^{\kappa-j}$ ($q_i=1$ for $i>0$). 
\end{prop}

\subsubsection{Caterpillars of length $L=2$}

Let us write $X^{2\uptext{-}\lB}_t = \xx^\head_t$.
Then we have $\xx^2_t(k) = \xx^{2\uptext{-}\lB}_{t-1}(k)$ and the discussion in Sec.~\ref{sec:LB_distribution} implies that $\xx^{2\uptext{-}\lB}_t \in \bar \Omega_N$.
Moreover, the evolution is given by $X^{2\uptext{-}\lB}_{t+1}(1) = X^{2\uptext{-}\lB}_{t}(1) - \xi(t+1,1)$ and
\[X^{2\uptext{-}\lB}_{t+1}(k) = \min \{X^{2\uptext{-}\lB}_{t}(k)  - \xi(t + 1, k), X^{2\uptext{-}\lB}_{t}(k-1) - 1\},\quad k = N, \dotsc, 2,\]
where $\xi(t + 1, k)$ are independent Ber[$p$] random variables.
In other words, $\xx^{2\uptext{-}\lB}_t$ evolves as follows.
Particles are updated from left to right, and the update of $X^{2\uptext{-}\lB}_t(k)$ is as follows:
\begin{itemize}[leftmargin=1.5em]
	\item If $X^{2\uptext{-}\lB}_t(k) < X^{2\uptext{-}\lB}_t(k-1)$, then $X^{2\uptext{-}\lB}_t(k)$ makes a Ber[$p$] jump to the left,
	\item If $X^{2\uptext{-}\lB}_t(k) = X^{2\uptext{-}\lB}_t(k-1)$, then $X^{2\uptext{-}\lB}_t(k)$ makes a deterministic jump to the left.
\end{itemize}
This is the left Bernoulli analog of TASEP with right Bernoulli jumps and parallel update: when going from time $t$ to time $t+1$, particle $k$ is pushed by the location of particle $k-1$ at time $t$ (instead of time $t+1$ as in the sequential case).
We can write an explicit formula for the transition probabilities for this model (a formula for the multi-point distributions was given already in Prop.~\ref{prop:caterpillars_BerPush}):

\begin{lem}\label{lem:Ber_Left_L2}
The transition probabilities of the process $\xx^{2\uptext{-}\lB}$ introduced above are
\begin{equation}\label{eq:Ber_Left_L2}
 \pp (X^{2\uptext{-}\lB}_t = \vec x | X^{2\uptext{-}\lB}_0 = \vec y) = p^{\bar \CN(\vec x)} \det \bigl[F^{\lB}_{i - j}(x_{N + 1 - i} - y_{N + 1 - j}, t + i - j)\bigr]_{i, j \in \set{N}},
\end{equation}
where $t \in \nn_0$, $\vec y, \vec x \in \bar \Omega_N$, $\bar \CN(\vec x) = \# \{ 2 \leq i \leq N : x_{i-1} = x_i\}$, and the function \eqref{eq:lBerFunction} with a contour $\gamma$ which includes $0$ and $1$ but not the singularity at $-p/q$.
\end{lem}

\begin{proof}
Using the definition \eqref{eq:lBerCaterpillars}, we get the identity $X_t(k) = \bar{X}_t(k) - t + k -1$, where $\bar{X}_t$ is the parallel TASEP with right Ber[$q$] jumps. Then Lem.~\ref{lem:Ber-TASEP_prll} allows to write the probability \eqref{eq:Ber_Left_L2} as 
\begin{align}
 \pp &(\bar X_t(k) = x_k + t - k + 1, k \in \set{N} | \bar X_0(k) = y_k  - k + 1, k \in \set{N}) \\
 &= p^{\CN(x_k - k, k \in \set{N})} \det \bigl[\bar F^{\rB}_{i - j}(x_{N + 1 - i} - y_{N + 1 - j} + t + i - j, t + i - j)\bigr]_{i, j \in \set{N}},
\end{align}
where $\CN$ is defined in Lem.~\ref{lem:Ber-TASEP_prll}, and where the function $\bar F^{\rB}_{n}$ is given by \eqref{eq:rBerBlockF} with $p$ and $q$ swapped.
The last identity is exactly \eqref{eq:Ber_Left_L2}.
\end{proof}

\subsection{Left geometric jumps} 
\label{sec:LG}

In this model, \emph{discrete time TASEP with left geometric jumps and pushing} (or \emph{geometric PushTASEP}), to go from time $t$ to time $t+1$ particles are updated sequentially from right to left as follows: each particle makes a jump to the left with distribution Geom[$p$] pushing to the left all particles on its way, so that the configuration of particles stays in $\Omega_N$.
As in the previous two cases, the update rule is sequential: when going from time $t$ to time $t+1$, a particle is pushed by the location of its right neighbor at time $t+1$.
However, and in contrast to the push dynamics in the left Bernoulli model from Sec. \ref{sec:LB}, if a particle is pushed by its right neighbor, it still gets to make its own geometric jump after that (see also Sec.~\ref{sec:previous}).
The evolution of particles $\xx^{\lG}_t \in \Omega_N$ satisfies
$X^{\lG}_{t+1}(1) = X^{\lG}_{t}(1) - \xi(t+1,1)$ and
\begin{equation}\label{eq:ProbGeomPush_L1}
X^{\lG}_{t+1}(k) = \min \{X^{\lG}_{t}(k), X^{\lG}_{t+1}(k-1) - 1\} - \xi(t+1,k), \quad k = 2, \dotsc, N,
\end{equation}
where  $\xi(t,k)$ are independent Geom[$p$] random variables.
The transition probabilities of the model are (see Appdx.~\ref{sec:ProbGeomPush})
\begin{equation}\label{eq:ProbGeomPush}
 	\pp (X^{\lG}_t = \vec x | X^{\lG}_0 = \vec y) = \det \bigl[F^{\lG}_{i - j}(x_{N + 1 - i} - y_{N + 1 - j}, t)\bigr]_{i, j \in \set{N}},
\end{equation}
where $\vec x, \vec y \in \Omega_N$, $t \in \nn_0$ and
\begin{equation}\label{eq:F_GeomPush_new}
 	F^{\lG}_{n}(x, t) = \frac{1}{2\pi \I} \oint_{\gamma} \d w\,\frac{(w-1)^{-n}}{w^{x - n +1}} \left(\frac{p}{1- q / w}\right)^t,
\end{equation}
where the contour encloses $0$, $q$ and $1$. 
This is the setting of Thm.~\ref{thm:main} with $\varphi(w) = \frac{p}{1- q / w}$, for which Assum.~\ref{a:phi} clearly holds.
We write the corresponding distribution function in greater generality below.

This model can be used to obtain explicit formulas for TASEP, which scale to formulas for the KPZ fixed point, with some special choices of random initial data, see Sec.~\ref{sec:KPZfixedpt-random}.

\subsubsection{Caterpillars}

By analogy with Lem.~\ref{lem:TASEP_and_caterpillars} we define caterpillars in this case through $\xx^i_t(k) = \xx^{\lG}_{t - (k-1)\kappa-i+1}(k)$, with $\kappa = L -1$. Each caterpillar now lives in the space of \emph{stretched backward caterpillars} $\bar\cK^{\la}_L=\big\{(\xx^1,\dotsc,\xx^L)\in\zz^L\!: \xx^{i+1}-\xx^{i}\in\nn_0,\,i \in \set{L-1}\big\}$ (two segments of a given caterpillars can be at any distance from each other), and the whole system takes values in the space
\[\bar\Omega^{\catpl}_{N,L}=\big\{X=(X(1),\dotsc,X(N))\in(\bar\cK^{\la})^N\!:X^1(k+1)<X^L(k),\,k\in\set{N-1}\big\}\]
(as for the left Bernoulli case, caterpillars may overlap in this system but the tail of every caterpillar always has to be to the right of the head of its left neighbor).
The evolution of this system follows from the definition of $\xx^i_t(k)$ (and the dynamics \eqref{eq:ProbGeomPush_L1} of $\xx^{\lG}_t$); the transition from time $t$ to time $t+1$ is given as follows, with the positions of the caterpillars being updated consecutively for $k \in\ \set{N}$ (i.e., from right to left):
\begin{itemize}[leftmargin=1.5em]
 \item If $\xx_t^{1}(k) \geq \xx_{t+1}^{L}(k-1)$, then the head of the $k^\text{th}$ caterpillar moves to $\xx_{t+1}^{L}(k-1) - 1$.
 \item Next the head of the $k^\text{th}$ caterpillar makes a jump to the left with distribution Geom[$p$].
 \item Finally we set $\xx^i_{t+1}(k) = \xx^{i-1}_{t}(k)$ for $i = 2, \dotsc, L$.
\end{itemize}
We can now compute the distribution functions using Thm.~\ref{thm:main2}:

\begin{prop}\label{prop:LG_distribution}
Consider the system of caterpillars $\xx$ of length $L\geq 1$, with an initial state $\xx(0) = \vec y \in \Omega_N$.
Then for any $1 \leq n_1 < \dotsm <  n_m \leq N$, any $t \geq \kappa(n_m-1)$, and any $\vec a \in \rr^m$, we have
\begin{equation}\label{eq:LGheads}
\pp \bigl(\xx^{\head}_t(n_i) > a_{i}, i \in \set{m}\bigr) = \det \bigl(I-\bP_{a}  K^{\lG} \bP_{a} \bigr)_{\ell^2(\{n_1, \dotsc, n_m\} \times \zz)},
\end{equation}
where the kernel $K^{\lG}$ is given by \eqref{eq:kernel-main} with $\varphi(w) = \frac{p}{1- q / w}$, $\kappa = L - 1$, and any $r,\theta\in(q,1)$.
The random walk used to define \eqref{eq:SMepi} in this case has transition matrix $Q(x,y)=(1-\theta)(\frac{p}{1- q / \theta})^{-\kappa}\theta^{x-y-1}q_i$ with $q_i=1$ for $i>0$ and $q_i=1 - p^\kappa \sum_{j = 0}^{-i} {j + \kappa - 1 \choose \kappa - 1} q^j$ otherwise.
\end{prop}

\begin{proof}
The function $\varphi$ as chosen in this case clearly satisfies Assum.~\ref{a:phi} and Assum.~\ref{a:kappa}(b) with any $q<\rhoin<1<\rhoout$.
On the other hand, Assum.~\ref{a:kappa}(b) is satisfied for this model for any $\kappa\geq 1$ if the initial states $\vec y$ take values in $\bar \Omega_N$ (defined in \eqref{eq:OmegaBar}); the proof of this fact is similar to the one for the right Bernoulli case (see Appdx.~\ref{sec:rightBernoulli-assumptions}), so we omit the details.
In view of this, the result follows from Thm.~\ref{thm:main2}.
\end{proof}

For this model, condition \eqref{eq:backward-in-time-one} seems not to hold, and in fact it appears to be the case that \eqref{eq:LGheads} does not hold in general for $t<\kappa(n_m-1)$.

\subsubsection{Caterpillars of length $L=2$} 
In the case $L=2$ we set $X^{2\uptext{-}\lG}_t = \xx^\head_t$. The dynamics of $X^{2\uptext{-}\lG}_t$ can be described using the general caterpillars' dynamics. We can also derive equations similar to \eqref{eq:ProbGeomPush_L1}. We have $\xx^2_t(k) = \xx^{2\uptext{-}\lG}_{t-1}(k)$ and \eqref{eq:ProbGeomPush_L1} gives the evolution $$X^{2\uptext{-}\lG}_{t+1}(k) = \min \{X^{2\uptext{-}\lG}_{t}(k), X^{2\uptext{-}\lG}_{t}(k-1) - 1\} - \xi(t + 1, k)$$ for $k = N, \dotsc, 2$, and $X^{2\uptext{-}\lG}_{t+1}(1) = X^{2\uptext{-}\lG}_{t}(1) - \xi(t+1,1)$, where $\xi(t + 1, k)$ are independent Geom[$p$] random variables.
The process is again Markovian, and gives a parallel update version of geometric PushTASEP.
In words, particles are updated from left to right; when $X^{2\uptext{-}\lG}_{t}(k)$ is updated, first, if it is located on or to the right of $X^{2\uptext{-}\lG}_{t}(k-1)$, then it is pushed to $X^{2\uptext{-}\lG}_{t}(k-1) - 1$; after that $X^{2\uptext{-}\lG}_t(k)$ makes a Geom[$p$] jump to the left.
Note that at a fixed time $t$ the particles may not be ordered.

\subsection{Right geometric jumps}
\label{sec:rightGeometric}

\subsubsection{Parallel update}
\label{sec:rightGeometric-prll}

In this model, to go from time $t$ to time $t+1$ particles are updated sequentially from left to right (i.e., starting with the particle with label $N$) as follows: the $k^{\text{th}}$ particle $X^{\rG}_t(k)$ tries to make a Geom[$p$] jump to the right but if the destination site is bigger than or equal to $X^{\rG}_t(k-1)$, then it arrives at $X^{\rG}_t(k-1) - 1$.
Note that, in contrast with the previous three cases, updates in this model, \emph{TASEP with right geometric jumps and blocking} occur in \emph{parallel}, just as in the model from Sec.~\ref{sec:Ber_seq} (even though, as we will see, this model corresponds to $\kappa=0$, which in the above setting would correspond to caterpillars of length $L=1$), see also Sec.~\ref{sec:previous}.
The evolution of the particles $X^{\rG}_t \in \Omega_N$ is described by the equations $X^{\rG}_{t+1}(1) = X^{\rG}_{t}(1) + \xi(t+1,1)$ and
\begin{equation}\label{eq:RG-evolution}
X^{\rG}_{t+1}(k) = \min \{X^{\rG}_{t}(k) + \xi(t+1,k), X^{\rG}_t(k-1) - 1\},\qquad k = N, \dotsc, 2,
\end{equation}
with $\xi(t,k)$ i.i.d. Geom[$p$] random variables on $\{0,1,\dotsc\}$. The transition probabilities of $X^{\rG}_t$ are (see Appdx.~\ref{sec:GeomPush})
\begin{equation}\label{eq:rightGeometric}
\pp (X^{\rG}_t = \vec x | X^{\rG}_0 = \vec y) = \det \bigl[F^{\rG}_{i - j}(x_{N + 1 - i} - y_{N + 1 - j}, t)\bigr]_{i, j \in \set{N}},
\end{equation}
where $\vec x, \vec y \in \Omega_N$, $t \in \nn_0$, and 
\begin{equation}\label{eq:rightGeometric_F}
	F^{\rG}_{n}(x, t) = \frac{1}{2\pi\I}\oint_{\gamma}\!\d w\, \frac{(w - 1)^{-n}}{w^{x - n +1}} \left(\frac{p}{1-q w}\right)^t,
\end{equation}
where the contour includes $0$ and $1$, but does not include $1/q$. Assum.~\ref{a:phi} is satisfied with $\varphi(w) = \frac{p}{1- q w}$, which allows one to apply Thm.~\ref{thm:main}.
Assum.~\ref{a:kappa}, however, is not satisfied for this model.
The problem is that particles can jump to a location arbitrarily far to the right, so a particle starting at a given time needs to be aware of the particles to its right even if they have not started to move.
This is related to the fact that in this model updates occur from left to right, and means that the caterpillar construction of the above sections does not work in this case.

\subsubsection{Sequential update}

Instead, in order to have a determinantal formula of the type \eqref{eq:second-property} for the transition probability of particles with different starting times, we consider now a situation with particle $i$ starting at time $i - N$, $1 \leq i \leq N$.
Lem.~\ref{lem:RG-assumption} below proves the analog of \eqref{eq:second-property} in this situation, and by analogy with Lem.~\ref{lem:TASEP_and_caterpillars} we define 
\[X^{2\uptext{-}\rG}_t(k) = \xx^{\rG}_{t + k-1}(k).\]
Note that the analogy cannot be pushed any further: for $L=\kappa+1\geq3$ the definition $\xx^i_t(k) = \xx^{\rG}_{t + (k-1)\kappa-i+1}(k)$, $i\in\set{L}$, does not have a physical meaning, because a ``caterpillar'' interacts with its left neighbor in the future. 
From \eqref{eq:RG-evolution} we get the evolution $X^{2\uptext{-}\rG}_{t+1}(1) = X^{2\uptext{-}\rG}_{t}(1) + \xi(t+1,1)$ and
\[X^{2\uptext{-}\rG}_{t+1}(k) = \min \{X^{2\uptext{-}\rG}_{t}(k) +  \xi(t + 1, k), X^{2\uptext{-}\rG}_{t+1}(k-1) - 1\}\]
for $k = N, \dotsc, 2$, with $\xi(t + 1, k)$ i.i.d. Geom[$p$] random variables. 
In other words, $\xx^{2\uptext{-}\rG}_t$ evolves as follows: particles are updated from right to left; $X^{2\uptext{-}\rG}_t(k)$ makes a Geom[$p$] jump to the right, and if the destination site is larger than or equal to $X^{2\uptext{-}\rG}_{t+1}(k-1)$, then the $k^{\text{th}}$ particle arrives at $X^{2\uptext{-}\rG}_{t+1}(k-1)-1$.
Note that, with these dynamics, $X^{2\uptext{-}\rG}_t \in \Omega_N$ for all $t\geq0$.
This is just the \emph{sequential update} version of TASEP with right geometric jumps and blocking: the transformation $\xx^{\rG}_{t}(k)\longmapsto\xx^{\rG}_{t + k-1}(k)$ turns the parallel update of the right geometric model into sequential update, just as $\xx^{\rB}_{t}(k)\longmapsto\xx^{\rB}_{t -k+1}(k)$ turns sequential into parallel update for right Bernoulli jumps.

\subsubsection{Distribution functions}

The following result is proved in Appdx.~\ref{sec:rGeometric-proof}.
For parallel update it follows from a direct application of Thm.~\ref{thm:main}.

\begin{prop}\label{prop:RG_distribution}
Let $X_t$ be any of the particle systems $X^{\rG}_t$ (parallel update) or $X^{2\uptext{-}\rG}_t$ (sequential update).
Suppose $\xx(0) = \vec y \in \Omega_N$. Then for any $1 \leq n_1 < \dotsm <  n_m \leq N$ and any $\vec a \in \rr^m$, for any $t \geq 0$ in the case $X_t = X^{\rG}_t$ and for any $t \geq n_m-1$ in the case $X_t = X^{2\uptext{-}\rG}_t$ we have
\begin{equation}\label{eq:RG_distribution}
	\pp \bigl(X_t(n_i) > a_{i}, i \in \set{m}\bigr) = \det \bigl(I-\bP_{a}  K^{\rG} \bP_{a} \bigr)_{\ell^2(\{n_1, \dotsc, n_m\} \times \zz)},
\end{equation}
where the kernel $K^{\rG}$ is given by \eqref{eq:kernel-main}, defined using $\varphi(w) = \frac{p}{1 -q w}$, the values $\kappa = 0$ for parallel update and $\kappa = -1$ for sequential update, and with any $\theta\in(0,1)$ and $r \in (0,1)$.
The random walk used to define \eqref{eq:SMepi} in this case has transition matrix $Q(x,y)=(1-\theta)\theta^{x-y-1} \uno{x > y}$ in the case $\kappa = 0$ and $Q(x,y)=(1-\theta)\theta^{x-y-1} \frac{1}{1 - q \theta} (\uno{x = y + 1} + p \uno{x \geq y + 2})$ in the case $\kappa = -1$.
\end{prop}

The restriction $t \geq n_m-1$ for sequential update is analogous to the restriction appearing in Prop.~\ref{prop:LG_distribution} in the case of geometric jumps, and it is not clear to us whether it can be lifted. 

\subsection{Relation to previous work}\label{sec:previous}

As we mentioned at the beginning of this section, the transition probabilities of the four basic types of interacting particle systems (with blocking and pushing interactions, and with Bernoulli and geometric jumps) have been previously computed using the coordinate Bethe ansatz (which provides a method to solve the Kolmogorov forward equation for each model) \cite{Brankov,Povolotsky_2006} and via the RSK algorithm \cite{MR2469339}. 
While our generalization of these dynamics to systems of interacting caterpillars appears to be new, at least in the case of right Bernoulli jumps, the idea of computing transition probabilities of particles starting from different times (which plays a key role in our analysis, since it relates the evolutions of particles and caterpillars as explained in Lem.~\ref{lem:TASEP_and_caterpillars}) can be found in the literature in a different setting.

More precisely, in \cite{Brankov} a Sch\"{u}tz-type formula was derived for the transition probabilities of Bernoulli TASEP with sequential update with particles starting and finishing at different times (referred to as a generalized Green function in that paper) under a certain, not entirely explicit, condition on the initial and final configurations.
The method was developed further in \cite{PovolotskyPriezzhevSchutz} and \cite{Poghosyan_2012} in a setting where particles with larger indices start moving at later times, in contrast to our choice (which is exactly the opposite).
In that setting, there is no ambiguity in how the process gets started (since the dynamics of each particle is independent of those to its left), but interpreting what these Green functions compute is not entirely straightforward; those works show that they are essentially the exit probability of particles from certain special space-time regions. 
\cite{Poghosyan_2012} also contains a derivation of the biorthogonal correlation kernel for arbitrary starting and ending times (which we derive in a general setting see Sec.~\ref{sec:measures} for our choice of ordering of starting times).

The four basic dynamics considered in this section can also be described as edge projections of the $2+1$ dynamics on Gelfand-Tsetlin patterns (or triangular arrays of integers whose consecutive levels are interlaced) introduced in \cite{Anisotropic}, which is based on another intertwining construction (different from the one related to RSK; see also \cite{warren} for the Brownian case).
More precisely, the four choices of the function $F_t$ in \cite[Sec.~2.6]{Anisotropic} correspond to (up to a change of the variable and the parameters), and motivated, our four choices of the function $\varphi$ in Props.~\ref{prop:caterpillars}, \ref{prop:caterpillars_BerPush}, \ref{prop:LG_distribution} and \ref{prop:RG_distribution}.

The construction in \cite{Anisotropic} also helps to explain some of the specific features of the four basic particle systems which we have studied.
For example, the choice of right Bernoulli jumps yields Bernoulli TASEP as the left edge projection of the associated Gelfand-Tsetlin dynamics and Bernoulli PushTASEP  as its right edge projection (the top edge projection is related to non-intersecting Bernoulli walks for special choices of initial data).
In particular, the construction provides a coupling of the two processes.
Left Bernoulli jumps is essentially a sort of dual choice, which exchanges the dynamics of the two edges.
The case of right geometric jumps (left geometric jumps are again essentially equivalent) is more subtle: the right edge projection yields geometric PushTASEP, but the left edge projection is not even Markovian.
Very roughly, the reason is that the dynamics on Gelfand-Tsetlin patterns involves a certain conditioning on the particle jumps to respect the interlacing between levels, and while for Bernoulli jumps this can be done at each of the two edge projections by looking only at particles at the same edge, the long-range nature of geometric jumps makes this conditioning on the left edge depend also on other particles inside the triangular array.
This conditioning also explains why in Bernoulli PushTASEP particles which are pushed in a given time step do not get to attempt another jump, while in geometric PushTASEP particles can jump after being pushed: a Bernoulli random variable conditioned on being larger than or equal to $1$ can only take the value $1$, while (by the memoryless property) a geometric random variable conditioned on being larger than or equal to some value $\ell$ is just $\ell$ plus the same geometric random variable.

TASEP with right Bernoulli jumps and parallel update was also described in \cite[Sec.~2.6]{Anisotropic} as a projection of a parallel update version of their $2+1$ dynamics.
In the case of geometric jumps, \cite{WarrenWindridge} studied a suitable version of the same $2+1$ dynamics, which can be thought as enforcing one direction in the interlacing inequalities sequentially and the other one in parallel, for which both edge projections are Markovian, one being the parallel update version of geometric TASEP with block dynamics which we studied in Sec.~\ref{sec:rightGeometric}, the other one being the standard geometric PushTASEP.

Regarding the explicit Fredholm determinant formulas for the multipoint distribution of these systems, they were only available (though seemingly not explicitly written in the literature in all cases) in the case of packed or half-periodic initial data (see Ex.~\ref{ex:periodic-one-sided}) for the four basic models and for right Bernoulli TASEP with parallel update (versions of these formulas could be derived too in the case of half-stationary initial data, see Sec.~\ref{sec:KPZfixedpt-random}).

\section{Continuous time variants of TASEP}
\label{sec:continuousTASEP}

This section is devoted to the application of the general biorthogonalization framework presented in Sec.~\ref{sec:biorth} to continuous time versions of the TASEP models described in Sec. \ref{sec:caterpillars}.
We begin with continuous time TASEP, for which the biorthogonalization problem was originally solved in \cite{fixedpt}.

\subsection{Continuous time TASEP}\label{sec:TASEP}

In the \emph{totally asymmetric simple exclusion process} (simply \emph{TASEP} throughout the rest of this section), the process $X_t$, $t\in[0,\infty)$, takes values in $\Omega_N$ and evolves as follows: each particle independently attempts jumps to its neighboring site to the right at rate $1$, the jump being allowed only if that site is unoccupied. TASEP was first solved by \citet{MR1468391} using the coordinate Bethe ansatz, leading to
\begin{equation}\label{eq:schutzTASEP}
\pp \bigl(\xx_{t} = \vec x \big| \xx_0 = \vec y\bigr) = \det \bigl[F_{i - j}(x_{N + 1 - i} - y_{N + 1 - j}, t)\bigr]_{i, j \in \set{N}}
\end{equation}
with
\begin{equation}
F_{n}(x, t) = \frac{1}{2\pi\I}\oint_{\gamma}\d w\, \frac{(w - 1)^{-n}}{w^{x - n +1}} e^{t (w-1)},
\end{equation}
where the contour contains $0$ and $1$. 
Applying Thm.~\ref{thm:main}, we obtain 
\begin{equation}\label{eq:P-TASEP}
\pp \bigl(\xx_t(n_i) > a_{i},\, i \in \set{m}\bigr) = \det \bigl(I-\bP_{a}  K^{\TASEP} \bP_{a} \bigr)_{\ell^2(\{n_1, \dotsc, n_m\} \times \zz)},
\end{equation}
where $K^{\TASEP}$ is given by \eqref{eq:kernel-main} with $\varphi(w) = e^{w-1}$ and $\kappa = 0$, with any choice of $r,\theta\in(0,1)$.

In the case $\theta=1/2$, \eqref{eq:P-TASEP} is the formula obtained in \cite{fixedpt}, where it was used to show that in the usual KPZ 1:2:3 scaling, the TASEP particle positions converge to the KPZ fixed point (see Sec.~\ref{sec:fp}).

\subsection{Continuous time PushASEP}

Now we have $N$ particles at locations $X_t(1)>X_t(2)>\dotsm$ evolving according to the following continuous time Markovian dynamics.
Particles jump independently to the right with some rate $R\geq 0$ and to the left with some rate $L\geq 0$. 
When the $i^{\text{th}}$ particle jumps to the left, if the destination site is occupied, the occupying particle is pushed to the left (in other words, the whole cluster of nearest neighbor occupied site to the left of $i$ moves); this is the same push mechanism as in Sec. \ref{sec:LB}.
Jumps to the right by the $i^{\text{th}}$ particle, on the other hand, are blocked if the destination site is occupied (TASEP dynamics).
If $L=0$ the model becomes TASEP, while if $R=0$ it is a special case of Toom's model \cite{DLSS91}. 

From \cite[Prop.~2.1]{bp-push}, the distribution function of $N \geq 1$ particles is given by
\begin{equation}
\pp \bigl(\xx_{t} = \vec x \big| \xx_0 = \vec y\bigr) = \det \bigl[F^\push_{i - j}(x_{N + 1 - i} - y_{N + 1 - j}, t)\bigr]_{i, j \in \set{N}},
\end{equation}
where (here the contour again contains $0$ and $1$)
\begin{equation}
F^\push_{n}(x, t) = \frac{1}{2\pi\I}\oint_{\gamma} \d w\, \frac{(w - 1)^{-n}}{w^{x - n +1}} e^{t (Rw + L/w - R - L)}.
\end{equation}
Thm.~\ref{thm:main} yields
\begin{equation}\label{eq:PushTASEP}
\pp \bigl(\xx_t(n_i) > a_{i},\, i \in \set{m}\bigr) = \det \bigl(I-\bP_{a}  K^{\push} \bP_{a} \bigr)_{\ell^2(\{n_1, \dotsc, n_m\} \times \zz)},
\end{equation}
where $K^{\push}$ is given by \eqref{eq:kernel-main} with $\varphi(w) = e^{Rw + L/w - R - L}$ and $\kappa = 0$, with any choice of $r,\theta\in(0,1)$.

In the case $\theta=1/2$, \eqref{eq:PushTASEP} coincides with the formula obtained in \cite{nqr-kolmogorov}.
In that paper the validity of the formulas \eqref{eq:P-TASEP} and \eqref{eq:PushTASEP} for TASEP and PushASEP was proved directly by showing that the right hand side satisfies in each case the corresponding Kolmogorov backward equations.

\subsection{Continuous time TASEP with generalized update}\label{sec:Povolotsky-continuous}

Now we introduce a continuous time version of the model described in Sec.~\ref{sec:Povolotsky-discrete}. 
The state space is $\Omega_N$, and as for TASEP each particle jumps to the right at rate $1$ provided that the target site is empty.
However, if a particle jumps to the right and its neighboring site to the left was occupied before the jump, then its neighbor makes a unit jump to the right with probability $\beta \in [0,1)$ (in other words, a particle which jumps to the right brings along its left neighbor, if it has one, with probability $\beta$).
The transition probabilities for this model are given by
 \begin{equation}
 	\pp (X_t = \vec x | X_0 = \vec y) = (1-\beta)^{-\CN(\vec x)} \det \bigl[F^\gen_{i - j}(x_{N + 1 - i} - y_{N + 1 - j}, t)\bigr]_{i, j \in \set{N}},
\end{equation}
where $\vec x, \vec y \in \Omega_N$, $t \in \nn_0$, the function $\CN(\vec x)$ is defined in Lem.~\ref{lem:Ber-TASEP_prll}, and
\begin{equation}
F^\gen_{n}(x, t) = \frac{1}{2\pi\I}\oint_{\gamma}\frac{\d w}{w^{x - n + 1}} \Bigl( \frac{w-1}{1 - \beta w} \Bigr)^{-n} e^{t(w-1)},
\end{equation}
with $\gamma$ enclosing only the poles at $0$ and $1$. This formula can be obtained from \eqref{eq:G-Povolotsky}, by replacing the time variable $t$ by $t / p$ and taking $p \to 0$.
In the same way we get a formula like \eqref{eq:genTASEP} for the distribution function with $K^{\gen}$ now defined using $a(w)=1 - \beta w$ and $\psi(w) = e^{t(w-1)}$ and with any $r,\theta\in(0,1)$.

\subsection{Convergence to the KPZ fixed point}\label{sec:fp}

In \cite{fixedpt} the formula \eqref{eq:P-TASEP} was used to show that the TASEP particle positions (or the associated TASEP height function) converge, under the usual KPZ 1:2:3 scaling, to a scaling invariant Markov process.
This process, known as the \emph{KPZ fixed point}, is the conjectured universal scaling limit of all models in the KPZ universality class.
The main motivation for the present work is to obtain analogous formulas for other particle systems which can be used to prove convergence to the KPZ fixed point in a similar way.
We will not perform the asymptotics for these models in this paper, but we very briefly sketch the result in the case of TASEP. 

Let $\UC$ be the space of upper semicontinuous functions $\fh\!:\rr\longrightarrow[-\infty,\infty)$ satisfying $\fh(x)\leq A|x|+B$ for some $A,B>0$ and $\fh\not\equiv-\infty$, endowed with the local Hausdorff topology (see \cite[Sec.~3]{fixedpt} for more details).
Consider TASEP initial data $(X^\ep_0(i))_{i\geq1}$ such that for some $\fh_0\in\UC$ satisfying $\fh_0(\fx)=-\infty$ for $\fx>0$,
\begin{equation}\label{eq:TASEPini}
\ep^{1/2}\big(X^\ep_0(\ep^{-1}\fx)+2\ep^{-1}\fx-1\big)\longrightarrow-\fh_0(-\fx)
\end{equation}
in $\UC$.
Now consider \eqref{eq:P-TASEP} with $\theta=1/2$ and the scaling
\begin{equation}\label{eq:fp-scaling}
t=2\ep^{-3/2}\ft,\quad n_i=\tfrac12\ep^{-3/2}\ft-\ep^{-1}\fx_i-\tfrac12\ep^{-1/2}\fa_i+1,\quad a_i=2\ep^{-1}\fx_i-2.
\end{equation}
We also change variables in the kernel $K^{\TASEP}(n_i,z_i;n_j,z_j)$ in the Fredholm determinant through $z_i=2\ep^{-1}\fx_i+\ep^{-1/2}(u_i+\fa_i)-2$.
Note that this turns the projections $\bP_{a_i}$ into $\bP_{-\fa_i}$ and multiplies the kernel by $\ep^{-1/2}$.
Introducing the kernels
\begin{equation}\label{eq:fT}
\fT_{\ft,\fx}(u,v)=\ft^{-1/3}e^{\frac{2\fx^3}{3\ft^2}-\frac{(u-v)\fx}{\ft}}\Ai(\ft^{-1/3}(v-u)+\ft^{-4/3}\fx^2)
\end{equation}
for $\ft\neq0$, where $\Ai$ is the Airy function, it is proved in \cite[Lem. 3.5]{fixedpt} that for $y=\ep^{-1/2}v$ one has $\ep^{-1/2}\SM_{-t,-n_i}(y,z_i)\longrightarrow\fT_{-\ft,\fx_i}(v,u_i)$ and $\ep^{-1/2}\SN_{-t,n_j}(y,z_j)\longrightarrow\fT_{-\ft,-\fx_j}(v,u_j)$.
Note also that since $Q^n$ is the $n$-step transition probability of a Geom$[1/2]$ random walk jumping strictly to the left, under our scaling we have by the central limit theorem $\ep^{-1/2}Q^{n_j-n_i}(z_i,z_j)\longrightarrow e^{(\fx_i-\fx_j)\p^2}(u_i,u_j)$ as $\ep\to0$ for $\fx_i>\fx_j$ (which implies $n_j>n_i$), where $e^{\fx\p^2}$, $\fx\geq0$, denotes the heat kernel. Similarly, under this scaling the random walk $B$ inside the expectation defining $\SM^{\epi(X_0^\ep)}_{-t,n_j}$ in \eqref{eq:SMepi} becomes $\ep^{1/2}(B_{\ep^{-1}x}+2\ep^{-1}\fx-1)$, which converges to a Brownian motion $\fB(\fx)$ with diffusivity $2$, while the hitting time $\tau$ of the walk $B$ to the epigraph of $X^\ep_0$ in \eqref{eq:SMepi} becomes the hitting time of $\fB$ to the epigraph of the curve $-\fh^-_0(\fx)\coloneqq-\fh_0(-\fx)$ since, by \eqref{eq:TASEPini}, the initial data $X^\ep_0$ rescales to $-\fh^-_0$.
Putting these facts together leads in \cite[Sec. 3.3]{fixedpt} (after some calculations) to
\begin{multline}\label{eq:fixedptdet}
\lim_{\ep\to0}\pp\big(X_{2\ep^{-3/2}\ft}(\tfrac12\ep^{-3/2}\ft-\ep^{-1}\fx_i-\tfrac12\ep^{-1/2}\fa_i+1)>2\ep^{-1}\fx_i-2,\,i\in\set{m}\big)\\
=\det\!\left(\fI-\P_{\fa}\fK^{\hypo(\fh_0)}_{\ft,\uptext{ext}}\P_{\fa}\right)_{L^2(\{\fx_1,\dotsc,\fx_m\}\times\rr)}\eqqcolon\pp\big(\fh(\ft,\fx_i)\leq\fa_i,\,i\in\set{m}\big),
\end{multline}
where
$\fK^{\hypo(\fh_0)}_{\ft,\uptext{ext}}(\fx_i,\cdot;\fx_j,\cdot)=-e^{(\fx_j-\fx_i)\partial^2}\uno{\fx_i<\fx_j}+(\fT^{\hypo(\fh_0^-)}_{\ft,-\fx_i})^*\fT_{\ft,\fx_j}$,
with
\begin{equation}\label{eq:defShypo}
\fT^{\hypo(\fh)}_{\ft,\fx}(v,u)=\ee_{\fB(0)=v}\big[\fT_{\ft,\fx-\ftau}(\fB(\ftau),u)\uno{\ftau<\infty}\big]
\end{equation}
and where $\ftau$ is the hitting time by $\fB$ of the hypograph of $\fh$.
Justifying that the convergence of the kernels which we indicated above holds in trace class so that it  implies convergence of the Fredholm determinants requires considerable effort, we refer to \cite{fixedpt} for the details.

The second line of \eqref{eq:fixedptdet} defines the finite dimensional distributions of the KPZ fixed point $\fh(\ft,\fx)$, which in \cite{fixedpt} is shown to be a $\UC$-valued Markov process.
In fact, the formula only defines the KPZ fixed point for one-sided initial data $\fh_0$ (meaning $\fh_0(\fx)=-\infty$ for $\fx>0$); the generalization to all $\fh_0\in\UC$ can be done through a limiting procedure by shifting, see \cite[Sec. 3.4]{fixedpt}.

\begin{rem}\label{rem:fp-dens}
Note that \eqref{eq:TASEPini} means in particular that we are taking TASEP initial data which has average particle density $1/2$.
This is why in the derivation sketched above one takes $\theta=1/2$. One could instead assume that $\sqrt{2/(1-\rho)}\rho\tts\ep^{1/2}\big(X^\ep_0(\ep^{-1}\fx)+\rho^{-1}\ep^{-1}\fx-1\big)\longrightarrow-\fh_0(-\fx)$, corresponding to average particle density $\rho\in(0,1)$. By suitably modifying the above choice of scaling one would get again convergence to the KPZ fixed point.
To this end one needs to use $\theta=1-\rho$ to ensure as above the convergence of the random walk $B$ to a Brownian motion $\fB$.
We omit the details.
\end{rem}

\subsection{Formulas for TASEP and the KPZ fixed point with random initial data}\label{sec:KPZfixedpt-random}

The TASEP and KPZ fixed point formulas \eqref{eq:P-TASEP}/\eqref{eq:fixedptdet} have been derived for deterministic initial data.
Random initial data can be handled by averaging, but since the determinant is nonlinear this leads to non-explicit formulas.
However, for some special choices of random initial data one can write explicit formulas by composing the dynamics of two different particle systems.
A prominent example is TASEP with half-stationary initial data (and its KPZ fixed point limit), which can be obtained by composing TASEP with geometric PushTASEP: in fact, applying one step of geometric PushTASEP with parameter $p=1/2$ to the step initial condition $X_0^{\lG}(i)=-i$, $i\geq1$, leads to a configuration with particles on the negative integer line with independent Geom$[1/2]$ gaps, i.e., a product measure with density $1/2$.
This is known in the field and relatively simple, but we have not found it explicitly stated in this form in the literature (an exception is \cite{lnr}, which is partly based on a draft of this article), so we include it here, although we work in greater generality.
For simplicity we focus only on the composition of TASEP with geometric PushTASEP, although it will be clear that the same argument can be used for other combinations.

Let $X_t$ denote continuous time TASEP and recall the geometric PushTASEP particle system $X^{\lG}_t$ introduced in Sec.~\ref{sec:LG}.
We want to start with some given initial condition $(X^\lG_0(i))_{i\geq1}$, apply $\ell$ discrete time PushTASEP steps, and use the resulting configuration $X^\lG_\ell$ as the initial condition $X_0$ for the TASEP dynamics.
Since it does not introduce any difficulties, we also allow the parameter $p=p_k$ in the $k^{\text{th}}$ PushTASEP step to depend on $k$. 
 It is shown in \cite{lnr} that the resulting initial condition $X^\lG_\ell$ is essentially the top path of a system of reflected geometric random walks with a wall at $X^\lG_0$: the first walk $X^\lG_1$ is reflected off $X^\lG_0$, the second walk $X^\lG_2$ is reflected off $X^\lG_1$, and so on.
More precisely, these reflections take place through a discrete version of the Skorokhod reflection mapping (with a slight time shift), see \cite[Sec. 5.2]{lnr} for more details.

Using Cor.~\ref{cor:MCconv} we get for TASEP with initial data prescribed as above that
\[\pp\big(X_t=\vec{x} \ts\big|\ts X_0=X^{\lG}_\ell,\,X^{\lG}_0 = \vec y\big)=\det\!\big[\bar F_{i-j}(x_{N+i-1}-y_{N+j-1})\big]_{i,j\in\set{N}}\]
with $\bar F_{n}(x)=\frac{1}{2\pi\I}\oint_{\gamma}\d w\ts\frac{(w - 1)^{-n}}{w^{x - n +1}} e^{t (w-1)}\prod_{k=1}^\ell\frac{p_k}{1-q_kw^{-1}}$ (here $\gamma$ encloses $0$, $1$ and all the $q_k$'s, with $q_k=1-p_k$). Thm.~\ref{thm:main} now gives, for this choice of initial data, 
\begin{equation}\label{eq:circT}
\pp\bigl(X_t(n_i) > a_{i},\, i \in \set{m}\bigr) = \det \bigl(I-\bP_{a}K^\circ\bP_{a} \bigr)_{\ell^2(\{n_1, \dotsc, n_m\} \times \zz)}
\end{equation}
with $K^\circ$ given by \eqref{eq:kernel-main} with $\varphi(w) = e^{t (w-1)}\prod_{k=1}^\ell\frac{p_k}{1-q_kw^{-1}}$ and $\kappa = 0$, with any $\theta\in(\max_kq_k,1)$.

Consider now the scaling \eqref{eq:fp-scaling} introduced in Sec.~\ref{sec:fp} and assume that the initial PushTASEP configuration $X^{\lG}_0$ satisfies $\ep^{1/2}\big(X^\ep_0(\ep^{-1}\fx)+2\ep^{-1}\fx-1\big)\longrightarrow\ff(\fx)$, $\fx\geq0$, in $\UC$, for some $\ff\in\UC$ defined on $[0,\infty)$ (c.f. \eqref{eq:TASEPini}).
Choose also $p_k=\frac12(1-\ep^{1/2}b_k)$.
In view of the above description of our choice of initial data $X^{\lG}_\ell$, it is natural to expect that it will converge to an appropriate system of reflecting Brownian motions (RBMs).
It is shown indeed in \cite[Prop.~3]{lnr} that, under this scaling, $X^{\lG}_\ell$ converges in distribution, uniformly on compact sets, to a system of RBMs with drift $(Y^\ff(k))_{k\in\set{\ell}}$ with a wall at $\ff$, defined as follows: $Y^\ff_t(1)$ is a Brownian motion with drift $2b_1$ reflected off $\ff$ and, recursively, $Y^\ff_t(k)$ is a Brownian motion with drift $2b_k$ reflected off $Y^\ff_t(k-1)$ (all Brownian motions have diffusivity $2$; note also that we have changed the sign of the $b_k$'s compared with \cite{lnr}).
This pins down the limiting initial data for the TASEP dynamics, and then from \cite[Prop. 3.6]{fixedpt} we get, under the scaling \eqref{eq:fp-scaling}, that
\[\lim_{\ep\to0}\pp\bigl(X_t(n_i) > a_{i},\, i \in \set{m}\bigr)=\pp_{\mathfrak{H}^\ff_\ell}\big(\fh(\ft,\fx_i)\leq\fa_i,\,i\in\set{m}\big),\]
where the KPZ fixed point initial data is built out of the RBMs through (here we take $0\cdot\infty=0$)
\begin{equation}\label{eq:RMB-h}
\mathfrak{H}^{(\vec b)}_\ff(\fx)=Y^\ff_\ell(-\fx)\uno{\fx\leq0}-\infty\cdot\uno{\fx>0}.
\end{equation}

In order to compute the limit of the right hand side of \eqref{eq:circT} we need to repeat the arguments sketched in Sec. \ref{sec:fp} for the kernel $K^\circ$ instead of $K^{\TASEP}$.
The only difference in the kernels is that the current choice of $\varphi(w)$ has an extra factor $\prod_{k=1}^m\frac{p_k}{1-q_kw^{-1}}$.
Let us assume all the $b_k$'s are negative so that, given our choice $p_k=\frac12(1-\ep^{1/2}b_k)$, we can take $\theta=1/2$ in the definition of $\bar K$; if some $b_k$ is non-negative the argument can be repeated by adjusting $\theta$ with $\ep$.
To compute the limit \cite{fixedpt} uses the change of variables $w\longmapsto\frac12(1-\ep^{1/2}\tilde w)$. After this change of variables, the pointwise limit of the integrands in \eqref{def:sm} and \eqref{def:sn} are the same as in \cite{fixedpt} except for the additional factors coming from the rational perturbation in $\varphi(w)$.
Moreover, it can be checked that the steepest descent arguments used in Appdx. B of that paper to upgrade this to trace class convergence of the whole operator are not affected by these additional factors; the argument is lengthy but the adaptation is straightforward, so we omit it (the crucial points being, first, that the additional poles at $\tilde w=-b_k$ stay away from the contours of integration and, second, that the required estimates depend on terms of order $\ep^{-3/2}$ in the exponent after writing the integrands as $e^{F_\ep(w)}$, whereas the rational perturbations are of order $1$).
The upshot is that we just need to compute the limit of the rational perturbations in $\psi_t(w)$ and $1/\psi_t(1-w)$ after scaling.
For this we multiply $(\SM_{-t,-n})^*$ by $(-2)^\ell\ep^{\ell/2}$ and ${\SN}_{-t,n}$ by $(-2)^{-\ell}\ep^{-\ell/2}$ and note that, as $\ep\to0$,
\[\textstyle-2\ep^{1/2}\frac{p_k}{1-q_k\tts w^{-1}}\longrightarrow\frac1{b_k+\tilde w}\qqand-\frac{\ep^{-1/2}}{2}\frac{1-q_k\tts(1-w)^{-1}}{p_k}\longrightarrow b_k-\tilde w.\]
In view of this and \cite[Lem. 3.5]{fixedpt} we define the operators
\[\textstyle\fT_{\ft,\fx}^{\vec b,\pm}(u,v)=\frac1{2\pi\I}\int_{\langle}\ts dw\,e^{\frac{t}3 w^3+xw^2+(u-v)w}\prod_{k=1}^m(b_k\mp w)^{\pm1}\]
for $\ft>0$, where the contour crosses the real axis to the left of all $-b_k$'s and goes off in rays at angles $\pm\pi/3$.
Define also
$\fT^{\hypo(\ff),\vec b,+}_{\ft,\fx}$ as in \eqref{eq:defShypo} with $\fT_{\ft,\fx}$ replaced by $\fT_{\ft,\fx}^{\vec b,+}$.
The argument we just sketched leads to:

\begin{thm}
For the KPZ fixed point started with initial data $\fh(0,\cdot)=\mathfrak{H}^{(\vec b)}_{\ff}$, i.e., built out of RBMs with a wall at $\ff$ as in \eqref{eq:RMB-h}, we have, for any $\ft>0$, 
\begin{equation}
\pp\big(\fh(\ft,\fx_i)\leq r_i,\,i\in\set{m}\big)=\det\!\left(\fI-\P_{\fa}\fK^{\hypo(\ff),\vec b}_{\ft,\uptext{ext}}\P_{\fa}\right)_{L^2(\{\fx_1,\dotsc,\fx_m\}\times\rr)}\label{eq:fpff0}
\end{equation}
with $\fK^{\hypo(\ff),\vec b}_{\ft,\uptext{ext}}(\fx_i,\cdot;\fx_j,\cdot)=-e^{(\fx_j-\fx_i)\partial^2}\uno{\fx_i<\fx_j}+(\fT^{\vec b,-}_{\ft,\fx_j})^*\fT^{\hypo(\ff),\vec b,+}_{\ft,-\fx_i}$.
\end{thm}

Taking $\ell=1$, $b_1=0$ and $\ff$ to be $0$ at the origin and $-\infty$ everywhere else, the initial data $\mathfrak{H}^{(0)}_{\ff}$ becomes simply a one-sided Brownian motion on the left and $-\infty$ on the right of the origin, while $\fT^{\hypo(\ff),0,+}_{\ft,\fx}$ becomes simply $\bP_0\fT^{0,+}_{\ft,\fx}$.
This is corresponds then to the KPZ fixed point with half-stationary initial data and, for $\ft=1$, it recovers the formula for the Airy$_{2\to\uptext{BM}}$ process \cite{imamSasam1,corwinFerrariPeche}.

When the $b_i$'s are negative and $\ff\equiv0$, $Y^0_t(\ell)$ has a stationary measure and one can define a double-sided stationary version $\mathfrak{H}^{(\vec b)}_{\uptext{eq}}$ of the initial data $\mathfrak{H}^{(\vec b)}_{0}$.
For $\ft=1$, the KPZ fixed point with initial data $\mathfrak{H}^{(\vec b)}_{\uptext{eq}}$ defines an $\ell$-parameter deformation of the Airy$_1$ process which corresponds to initial data identically $0$ (and, at the level of one-point marginals, of the Tracy-Widom GOE distribution), for which formulas can be obtained as a limit of \eqref{eq:fpff0} with $\ff\equiv0$; see \cite[Sec. 5]{lnr} for more details.
  
\section{Biorthogonalization of a general determinantal measure}
\label{sec:measures}

In this section we study a general class of (possibly signed) determinantal measures and prove a Fredholm determinant formula for certain marginals of them, in terms of kernels given implicitly in a biorthogonal form.
An explicit formula for these kernels will be derived in Sec.~\ref{sec:biorth}. In the setting of Sec.~\ref{sec:main}, the measures which we will study correspond to \eqref{eq:G-main} and \eqref{eq:second-property} and the marginals to the left hand sides of \eqref{eq:probability-main-kappa=0} and \eqref{eq:probability-main}, but our framework is a bit more general.

In particular, we will study measures on particle configurations which depend on some auxiliary parameters $v_1, \dotsc, v_N > 0$ which, in the setting of \eqref{eq:G-main}, can be thought of as different \emph{speeds} for each of the $N$ particles (for example, for right Bernoulli TASEP these speeds would encode different jump probabilities).
Introducing different speeds is helpful to overcome some technical difficulties; it is in fact a standard approach in the framework of Schur processes and TASEP-like particle systems to prove formulas in terms of a Fredholm determinant (see e.g. \cite{Petrov} and \cite{bp-push}).
This generalization is also meaningful from a physical point of view, but we will not pursue it any further in this work: in fact, after obtaining our Fredholm determinant formula in Thm.~\ref{thm:biorth_general}, we will go back to equal speeds by taking $v_i \longrightarrow 1$ for all $i$.

Throughout the section, $t$ denotes a time variable taking values in $\T$ (which, we recall, can be either $\rr$ or $\zz$).
We also fix $N\in\nn$ and a vector $\vec v = (v_i)_{i \in \set{N}}$ such that $v_i > 0$ for each $i$. Define the kernel
\begin{equation}\label{eq:Q_def}
\Q_{i}(x_1, x_2) = \frac{1}{2\pi\I}\oint_{\gamma_{\rhoout}}\!\d w\, \frac{(w - v_i)^{-1}}{w^{x_2 - x_1}} = v_i^{x_1 - x_2} \uno{x_1 \geq x_2}
\end{equation}
for $i\in\set{N}$ and $x_1, x_2 \in \zz$, where $\rhoout>\max_iv_i$.
The inverse of $\Q_{i}$ is
\begin{equation}\label{eq:Q_inverse_def}
\Q^{-1}_{i}(x_1, x_2) = \frac{1}{2\pi\I}\oint_{\gamma_{\rhoin}}\d w\, \frac{w - v_i}{w^{x_2 - x_1 + 2}} = \uno{x_1 = x_2} -  v_i \uno{x_1 = x_2 + 1},
\end{equation}
where $\rhoin>0$. 
For $k \in \set{N}$ we set
\begin{equation}\label{eq:kernelsQ}
\Q^{[k]} = \Q_{1} \Q_{2} \dotsm \Q_{k}, \qquad \Q^{[-k]} = \Q_{k}^{-1} \dotsm \Q_{2}^{-1} \Q_{1}^{-1},
\end{equation}
with the convention $\Q^{[0]} = I$. 
The kernels of these operators can be written explicitly as 
\begin{equation}\label{eq:kernelsQ_formulas}
\Q^{[k]}(x_1, x_2) = \frac{1}{2\pi\I}\oint_{\gamma_{\rhoout}}\!\d w\, \frac{\prod_{i = 1}^k(w - v_i)^{-1}}{w^{x_2 - x_1 - k + 1}}, \qquad \Q^{[-k]}(x_1, x_2) = \frac{1}{2\pi\I}\oint_{\gamma_\rhoin}\!\d w\, \frac{\prod_{i = 1}^k(w - v_i)}{w^{x_2 - x_1 + k + 1}}.
\end{equation}
We also introduce the kernels (for $i\in\set{N}$)
\begin{equation}\label{eq:operatorsE}
\E_{i} (x_1, x_2) = v_i^{-x_1} \uno{x_1 = x_2}, \qquad \E_{-i} (x_1, x_2) = v_i^{x_2} \uno{x_1 = x_2}
\end{equation}
and a further kernel (depending on a given complex function $\varphi$)
\begin{equation}\label{eq:R_def}
\R_t(x_1,x_2) = \frac{1}{2\pi\I}\oint_{\gamma_\rhoin}\d w\, \frac{\varphi(w)^t}{w^{x_2 - x_1 +1}}.
\end{equation}

We make the following assumption throughout the whole section:

\begin{assumption}\label{a:speeds}
The function $\varphi$ satisfies Assum.~\ref{a:phi} with the annulus $A_{\rhoin,\rhoout}$ defined for some radii $\rhoin\in(0,\min_iv_i)$ and $\rhoout>\max_iv_i$.
\end{assumption}

Note that the conditions on $\rhoin$ and $\rhoout$ in Assum.~\ref{a:phi} correspond to the ones in this assumption in the case $v_i = 1$ for all $i$.
Note also that the choice of $\rhoin$ and $\rhoout$ in the assumption is compatible with the choices in \eqref{eq:Q_def}--\eqref{eq:kernelsQ_formulas}, and that it is such that the singularities of $\varphi$ whose modulus is smaller than $\min_i v_i$ are contained inside both $\gamma_\rhoin$ and $\gamma_{\rhoout}$.
This choice of radii will remain fixed throughout the rest of the section.
One could consider slightly more general assumptions (in particular the annulus $A_{\rhoin,\rhoout}$ could be replaced by a more general domain under additional conditions), but this choice is more than enough for all the applications we have in mind.

 For $k, \ell \in \set{N}$ and $t \in \T$ we define the function
\begin{equation}\label{eq:F_def}
F_{k, \ell}(x_1,x_2; t) = \bigl(\E_k \Q^{[k]} \R_t \Q^{[-\ell]} \E_{-\ell}\bigr) (x_1,x_2).
\end{equation}
From \eqref{eq:kernelsQ_formulas} and the properties of $\varphi$ it follows that the compositions of the kernels in this formula are absolutely convergent.
To see this, note first that the kernel $\Q^{[-\ell]}$ has finite range. On the other hand from \eqref{eq:kernelsQ_formulas} we get $|\Q^{[k]}(x_1, x_2)| \leq  \uno{x_1 \geq x_2} \prod_{i = 1}^k (\rhoout - v_i)^{-1} / \rhoout^{x_2 - x_1 - k}$, while in \eqref{eq:R_def} we may move the contour to $\gamma_{\rhoout'}$ for some fixed $\rhoout'>\rhoout$ so that, since $|\varphi(w)^t|$ is bounded on $\gamma_{\rhoout'}$, $|\R_t(x_1,x_2)| \leq C (\rhoout')^{x_1 - x_2}$ for some constant $C \geq 0$.
These bounds yield
\begin{equation}
\textstyle\sum_{y \in \zz} |\Q^{[k]}(x_1, y)| |\R_t(y,x_2)| \leq C \prod_{i = 1}^k (\rhoout - v_i)^{-1} \sum_{y \leq x_1} (\rhoout')^{y - x_2} / \rhoout^{y - x_1 - k} < \infty.
\end{equation}
This allows us also to compute the function $F_{k, \ell}$ explicitly as (see also Lem.~\ref{lem:conv})
\begin{align}\label{eq:F_formula}
F_{k, \ell}(x_1,x_2; t) = \frac{1}{2\pi\I}\oint_{\gamma_{\rhoout}}\!\!\d w\, \frac{(w/v_k)^{x_1}}{(w/v_\ell)^{x_2}} \frac{\prod_{i = 1}^{\ell} (w - v_i)}{\prod_{i = 1}^{k} (w - v_i)} \frac{\varphi(w)^t}{w^{\ell - k +1}}.
\end{align}
Finally, for $\vec y, \vec x \in \Omega_N$ and $s, t \in \T$ with $s \leq t$, we define
\begin{align}\label{eq:G}
\G_{s, t} (\vec y, \vec x) = \left( \prod_{i = 1}^N \varphi(v_i)^{s-t} \right) \det \bigl[F_{k, \ell}(y_{k}, x_{\ell}; t-s)\bigr]_{k, \ell \in \set{N}}.
\end{align}
One can readily check that as $v_i \longrightarrow 1$ for all $i$, the function $\G_{0, t} (\vec y, \vec x)$ converges to the right hand side of \eqref{eq:G-main} if $\varphi(1)=1$.
Although this function integrates to $1$ over $\vec x \in \Omega_N$ (see Lem.~\ref{lem:G_properties}), in general we do not require it to be positive, so it does not define in general a probability measure.
However, $\G_{0, t}$ satisfies the semigroup property.

\begin{lem}\label{lem:G_properties}
For any $\vec x, \vec y \in \Omega_N$ and $s, t \in \T$ one has $\G_{0, 0} (\vec y, \vec x) = \uno{\vec y = \vec x}$ and 
\begin{equation}
\sum_{\substack{x_N \in \zz \\ x_N < x_{N-1}}} \G_{0, t}(\vec y, \vec x) = \G_{0, t}(\vec{y}_{< N}, \vec {x}_{<N}), \qquad \sum_{\vec z \in \Omega_N} \G_{0, t}(\vec y, \vec z) \G_{0, s}(\vec z, \vec x) = \G_{0, t + s}(\vec y, \vec x),\label{eq:G_property1}
\end{equation}
where the vector $\vec{y}_{< N} \in \Omega_{N-1}$ is obtained from $\vec y$ by removing the $N^{\text{th}}$ entry. 
In particular, we have that $\sum_{\vec x \in \Omega_N} \G_{0, t}(\vec y, \vec x)=1$ for all $\vec y\in\Omega_N$.
\end{lem}

\begin{proof}
After the change of variable $z\longmapsto 1/ w$ in \eqref{eq:F_formula}, the function $\G_{0, 0} (\vec y, \vec x)$ coincides with the one in \cite[Eqn.~2.2]{bp-push}, and then the proof of $\G_{0, 0} (\vec y, \vec x) = \uno{\vec y = \vec x}$ is contained in the proof of \cite[Prop.~2.1]{bp-push}.

The second identity in \eqref{eq:G_property1} follows from Prop.~\ref{prop:Cauchy-Binet_general}.
To prove the first one we start by noting that for $k\leq N$ the integrand in $F_{k, N}$ in \eqref{eq:F_formula} does not have singularities inside the annulus $A_{\rhoin,\rhoout}$, so we may shrink the contour to $\gamma_\rhoin$ and compute
\begin{align}
\textstyle\sum_{x_N < x_{N-1}}  F_{k, N}(y_k, x_N; t) &=\textstyle \frac{1}{2\pi\I}\oint_{\gamma_\rhoin}\d w\, \frac{(w/v_k)^{y_k} \prod_{i = k+1}^{N} (w - v_i)  \varphi(w)^t}{w^{N - k +1}} \sum_{x_N < x_{N-1}} (w/v_N)^{-x_N}\\
&\textstyle= \frac{1}{2\pi\I}\oint_{\gamma_\rhoin}\d w\, \frac{(w/v_k)^{y_k}}{(w/v_N)^{x_{N-1}}} \frac{\prod_{i = k+1}^{N} (w - v_i)  \varphi(w)^t}{w^{N - k} (v_N - w)}.\label{eq:sum_of_x_N}
\end{align}
Expanding the contour to $\gamma_{\rhoout}$ we only cross a pole at $w=v_N$, and computing the residue we get 
\begin{align}\label{eq:sum_of_x_N2}
	\textstyle\frac{1}{2\pi\I}\oint_{\gamma_{\rhoout}}\d w\, \frac{(w/v_k)^{y_k}}{(w/v_N)^{x_{N-1}}} \frac{\prod_{i = k+1}^{N} (w - v_i)  \varphi(w)^t}{w^{N - k} (v_N - w)} + \uno{k = N} \varphi(v_N)^t.
\end{align}
From this and multilinearity of determinant, $\sum_{x_N < x_{N-1}}\!\G_{0, t}(\vec y, \vec x)$ equals $\prod_{i=1}^N\varphi(v_i)^{-t}$ times the determinant of an $N \times N$ matrix, whose first $N-1$ columns are the same as before while the $N^{\text{th}}$ one has entries $\uno{k = N} \varphi(v_N)^t$ for $1 \leq k \leq N$; the first term in \eqref{eq:sum_of_x_N2} can be removed because it gets canceled by addition of $\left(v_{N} / v_{N-1}\right)^{x_{N-1}}$ times the $(N-1)^{\text{st}}$ column. A cofactor expansion of this determinant along the $N^{\text{th}}$ column gives the first identity in \eqref{eq:G_property1}.

To see that $\G_{0, t}$ integrates to $1$ we apply the first identity in \eqref{eq:G_property1} $N-1$ times to get $\sum_{\vec x \in \Omega_N}\!\G_{0, t}(\vec y, \vec x)$ 
$=\sum_{x_1\in\zz}\G_{0, t}(y_1,x_1) = \varphi(v_1)^{-t} \sum_{x_1\in\zz} F_{1,1}(y_1,x_1;t)$. 
To compute the sum over $x_1 < y_1$, we shrink the radius of the contour to $\rhoin<\min_i v_i$ as before to get $\varphi(v_1)^{-t} \frac{1}{2\pi\I}\oint_{\ts\gamma_\rhoin}\d w\, \frac{\varphi(w)^t}{v_1 - w}$.
For the sum over $x_1 \geq y_1$ we keep the original contour and get $\varphi(v_1)^{-t} \frac{1}{2\pi\I}\oint_{\ts\gamma_{\rhoout}}\d w\, \frac{\varphi(w)^t}{w - v_1}$. 
Summing the two we are left with the last integral on a small contour around $v_1$, and computing the residue we get $1$ as desired.
\end{proof}

The function \eqref{eq:G} defines a measure on particle configurations in a space-time domain. We are interested in its projections to special sets known as \emph{space-like paths},\footnote{The ``space-like paths'' terminology is related to the intepretation of particle systems related to TASEP as growth models, see the explanation in the introduction and Sec.~2.2 of \cite{bp-push}, where it was introduced.} which we introduce now. For $(n_1, t_1), (n_2, t_2) \in \set{N} \times \T$ we write $(n_1, t_1) \prec (n_2, t_2)$ if $n_1 \leq n_2$, $t_1 \geq t_2$ and $(n_1, t_1) \neq (n_2, t_2)$. We write $\fn=(n,t)$ to denote elements of $\set{N} \times \T$.
Then we define the set of \emph{space-like paths} as
\begin{equation}\label{eq:space-like-paths}
\SLP_N = \bigcup_{m \geq 1} \bigl\{(\fn_i)_{i \in \set{m}}\!: \fn_i \in \set{N} \times \T, \fn_i \prec \fn_{i+1}\bigr\}.
\end{equation}
For a space-like path $\cS = \{(n_1, t_1), \dotsc, (n_m, t_m)\} \in \SLP_N$ and for $\vec y \in \Omega_N$ and $\vec x \in \Omega_m$, we set\footnote{\label{ft:timepoints}Here and later we use $\vec x(t_i)$ to parametrize vectors by time points. In particular, we postulate that $\vec x(t_i)$ and $\vec x(t_{i+1})$ are different vectors even if $t_{i} = t_{i+1}$. 
This slight abuse of notation, which makes clear the correspondence between vectors and the associated time points, will simplify the presentation later on.}
\begin{equation}\label{eq:G+}
G^{+}_{\cS} (\vec y, \vec x) = \sum_{\substack{\vec x(t_i) \in \Omega_{n_i} : \\ x_{n_i}(t_i) = x_i, i \in \set{m}}} \G_{0, t_m}(\vec y_{\leq n_m}, \vec x(t_m)) \prod_{i=1}^{m-1} \G_{t_{i + 1}, t_{i}}(\vec x_{\leq n_{i}}(t_{i+1}), \vec x(t_{i})).
\end{equation}
Furthermore, for $T_N \leq \dotsm \leq T_1$ and for $\vec x \in \Omega_N$ and $\vec y \in \zz^N$, we set  
\begin{equation}\label{eq:G-}
G^{-}_{\vT} (\vec y, \vec x)  = \left( \prod_{i = 1}^N \varphi(v_i)^{T_i} \right)  \det \bigl[F_{k, \ell}(y_k, x_\ell; - T_{k})\bigr]_{k, \ell \in \set{N}}.
\end{equation}
In the setting of the Markov chain $X_t$ considered in Sec.~\ref{sec:main}, if we take $v_i \longrightarrow 1$ and $t_i \geq 0$ for all $i$, then $G^{+}_{\cS} (\vec y, \vec x)$ becomes $\pp (\xx_{t_i}(n_i) = x_i, i \in \set{m} | \xx_{0} = \vec y)$, while for $\kappa\geq1$ if we take $T_i = -\kappa (i-1)$ as in Assum.~\ref{a:kappa}, then the function $G^{-}_{\vT} (\vec y, \vec x)$ becomes \eqref{eq:second-property}. 

Convolving \eqref{eq:G+} and \eqref{eq:G-} in the case $T_1 \leq t_m$, we define
\begin{equation}\label{eq:G_TS}
G_{\vT, \cS}(\vec y, \vec x) = \sum_{\vec z \in \Omega_N} \G^{-}_{\vT} (\vec y, \vec z) \G^{+}_{\cS} (\vec z, \vec x).
\end{equation}
Our goal is to obtain a formula for the following integrated version of $\G_{\vT,\cS}$: for $\vec y\in\zz^N$, $\vec a\in\zz^m$,
\begin{equation}\label{eq:mu_fixed}
\CM_{\vT,\cS}(\vec y, \vec a) = \sum_{\substack{\vec x \in \Omega_m : \\ x_i > a_i,  i \in \set{m}}} \G_{\vT,\cS}(\vec y, \vec x).
\end{equation}
If $v_i \longrightarrow 1$, $T_i = -\kappa (i-1)$ and $t_i = t - \kappa (n_i-1) \geq 0$ for each $i$, then $\CM_{\vT,\cS}(\vec y, \vec a)$ becomes the probability \eqref{eq:probability-main}, which follows from the Markov property. 

\subsection{Biorthogonalization}\label{sec:biorth_sub}

The main result of this section provides a Fredholm determinant formula for $\M_{\vT,\cS}(\vec y, \vec a)$ in terms of a kernel constructed out of the solution of a certain biorthogonalization problem.
This type of result was first obtained for continuous time TASEP in \cite{sasamoto,borFerPrahSasam}, and was later extended to other processes in several papers.
Our result is essentially an extension of those (in particular \cite{bp-push,borodFerSas}, where distributions along space-like paths were first studied) to the case of different starting times.
We indicate the differences with these results more precisely below Thm.~\ref{thm:biorth_general}.

Before stating the result we need to introduce a space of functions $\V{n}(\vec v)$.
For fixed $n \in \set{N}$ and given a vector $\vec v$ as above, let $u_1 < u_2 < \dotsm < u_\nu$ denote the distinct values among the first $n$ entries $v_1,\dotsc,v_n$ of $\vec v$ and let $\beta_k$ be the multiplicity of $u_k$ among these entries.
Then we let
\begin{equation}\label{eq:space}
\V{n}(\vec v) = \spanning \bigl\{x \in \zz \longmapsto x^\ell u_k^{x} : 1 \leq k \leq \nu,\; 0 \leq \ell < \beta_k\bigr\}.
\end{equation}
Furthermore, we extend the multiplication operators \eqref{eq:defChis} to $\ell^2(\cS \times \zz)$, for a space-like path $\cS$, as
\begin{equation}
\chi_a((n_j, t_j),x)=\chi_a(n_j,x)\qqand\bP_a((n_j, t_j),x)=\bP_a(n_j,x).
\end{equation}

\begin{thm}\label{thm:biorth_general}
Let the function $\varphi$ and the values $v_i$ satisfy Assum.~\ref{a:speeds}, and fix $T_N \leq \dotsm \leq T_1$ and a space-like path $\cS$, the time points of which are all greater than $T_1$.
Then the function \eqref{eq:mu_fixed} can be written as
\begin{equation}\label{eq:M_formula}
\CM_{\vT,\cS}(\vec y, \vec a) = \det \bigl(\Id- \bP_{a}  \K \bP_{a}  \bigr)_{\ell^2(\cS \times \zz)},
\end{equation}
where $\det$ is the Fredholm determinant, $\Id$ is the identity operator, $\bP_{a} $ is defined in \eqref{eq:defChis}, and:
\begin{enumerate}[label=\uptext{(\arabic*)}]
\item The kernel $\K\!: (\cS \times \zz)^2\longrightarrow\rr$ depends on $\vv T$ and $\vec y$, and is given by
\begin{equation}\label{eq:KernelK}
\K(\fn_i, x_i; \fn_j, x_j)=-\phi^{(\fn_i, \fn_j)}(x_i,x_j) \uno{\fn_i \prec \fn_j} + \sum_{k = 1}^{n_j}\Psi^{\fn_i}_{n_i - k}(x_i)\Phi^{\fn_j}_{n_j - k}(x_j),
\end{equation}
for $\fn_i = (n_i, t_i)$ and $\fn_j = (n_j, t_j)$ in $\cS$.
\item For $\fn_i$ and $\fn_j$ as before, such that $\fn_i \prec \fn_j$, the function $\phi^{(\fn_i, \fn_j)}$ is defined as
\begin{equation}\label{eq:phi}
\phi^{(\fn_i, \fn_j)}(x_i, x_j) = \frac{1}{2\pi\I}\oint_{\gamma_\rhoin}\d w\, \frac{\varphi(w)^{t_i - t_j}}{w^{x_i - x_j - n_j + n_i + 1}} \prod_{k = n_i + 1}^{n_j} (v_k - w)^{-1}.
\end{equation}
\item For $\fn = (n, t) \in\cS$ and $k \in \set n$, the function $\Psi^{\fn}_{n - k}$ is given by
\begin{equation}\label{eq:Psi}
\Psi^{\fn}_{n - k}(x) = \frac{1}{2\pi\I}\oint_{\gamma_\rhoin}\d w\,\frac{\varphi(w)^{t - T_k}}{w^{x - y_k + n - k + 1}} \prod_{i = k+1}^n (v_i - w).
\end{equation}
\item The functions $\Phi^{\fn}_{n - k}$, for $ k \in \set n$ and $\fn = (n,t)$, are uniquely characterized by:
\begin{enumerate}[label=\uptext{(\alph*)}]
\item The biorthogonality relation $\sum_{x\in\zz}\Psi_\ell^{\fn}(x)\Phi_k^{\fn}(x)=\uno{k=\ell}$, for each $k,\ell=0,\dotsc,n-1$.\label{it:biorth}
\item $\spanning\{x \in \zz \longmapsto \Phi^{\fn}_{k}(x) : 0 \leq k < n \} = \V{n}(\vec v)$.\label{it:poly}
\end{enumerate}
\end{enumerate}
\end{thm}

\begin{rem}\label{rem:biorth_general}
Assum.~\ref{a:speeds} guarantees that $\sum_{x\in\zz}\Psi_\ell^{\fn}(x) x^k v_i^x$ is absolutely convergent for any $k \geq 0$, which makes the statement of the biorthogonality relation valid. More precisely, the contour of integration in \eqref{eq:Psi} yields $|w| = \rhoin < v_i$ and $\sum_{x \leq 0} |\Psi_\ell^{\fn}(x) x^k v_i^x| \leq C_1 \sum_{x \leq 0} |x|^k (v_i / \rhoin)^{x} < \infty$. On the other hand, we can extend the integration contour in \eqref{eq:Psi} to $\gamma_{\rhoout}$, which yields $|w| = \rhoout > v_i$ and $\sum_{x > 0} |\Psi_\ell^{\fn}(x) x^k v_i^x | \leq C_2 \sum_{x > 0} x^k (v_i / \rhoout)^{x} < \infty$.
\end{rem}

The proof of Thm.~\ref{thm:biorth_general} is provided in Appdx.~\ref{app:biorth}.
Although it follows closely the proof of \cite[Prop.~3.1]{bp-push}, there are two important differences in our case which require us to provide a complete proof (beyond the fact that we work with a general choice of $\varphi$): 
1. Only the measure corresponding to \eqref{eq:G+} was biorthogonalized in \cite{bp-push}; the extension to different starting times \eqref{eq:G_TS} which we consider introduces an additional factor $\varphi(w)^{- T_{k}}$ in our formula \eqref{eq:Psi}, and makes the argument a bit more complicated.
2. Our choice of $\phi^{(\fn_i,\fn_j)}$ in \eqref{eq:phi} is slightly different from the one in \cite{bp-push}.
More precisely, taking $v_1 = \dotsm = v_N = 1$, $\phi^{((n,t),(n+1,t))}$ simplifies to $\uno{x_1 > x_2}$, while the respective function in \cite[Prop.~3.1]{bp-push} is given by $\uno{x_1 \leq x_2}$.
Our choice of these functions will be more convenient for the explicit biorthogonalization which we will provide in Sec.~\ref{sec:biorth}.

\begin{rem}\label{rem:conj}
Without changing the value of the function in \eqref{eq:G}, we can conjugate the matrix by $c^{x}$ for any $c \neq 0$, and consider the function $c^{y - x} F_{k, \ell}(y, x; t)$. Then the statement of Thm.~\ref{thm:biorth_general} holds in the same form, with all functions conjugated by $c^x$.
\end{rem}

If we take the limit $v_i \longrightarrow 1$ for all $i$, this theorem can be applied to compute the probability \eqref{eq:probability-main}.

\begin{cor}\label{cor:biorth_caterpillars}
In the setting of Thms.~\ref{thm:main} and \ref{thm:main2}, suppose that Assum.~\ref{a:phi} holds, let $\rhoin$ and $\rhoout$ be as in that assumption, and suppose that either $\kappa=0$ or, for some choice of $\kappa\geq1$ and some given initial state $\vec y$, Assum.~\ref{a:kappa} also holds.
Then for any $t \geq \kappa(n_m-1)$, $\theta\in(\rhoin,\rhoout)$ and $\vec a \in\zz^m$ we have
\begin{equation}\label{eq:M_formula_caterpillars}
\pp \bigl(\xx_{t - \kappa (n_i-1)}(n_i) > a_{i},\, i \in \set{m}\, \big|\, X_{- \kappa (N-1)} = \vec y, \CE_\kappa\bigr) = \det \bigl(\Id- \bP_{a}  \K \bP_{a}  \bigr)_{\ell^2(\{n_1,\dotsc,n_m\} \times \zz)}
\end{equation}
with
\begin{align}\label{eq:K_equal_speeds}
\K(n_i, x_i; n_j, x_j) &=-Q^{n_j - n_i}(x_i,x_j) \uno{n_i < n_j} + \sum_{k = 1}^{n_j}\Psi^{n_i}_{n_i - k}(x_i)\Phi^{n_j}_{n_j - k}(x_j),\\
Q^{n}(x_i, x_j) &= \frac{1}{2\pi\I}\oint_{\gamma_{\rhoin}} \d w\, \frac{\theta^{x_i - x_j}}{w^{x_i - x_j + 1}} \left( \frac{w \varphi(w)^{\kappa}}{1 - w} \right)^n,\label{eq:Q_equal_speeds}\\
\Psi^{n}_{n-k}(x) &= \frac{1}{2\pi\I}\oint_{\gamma_{\rhoin}}\d w\,\frac{\theta^{x - y_k} \varphi(w)^{t}}{w^{x - y_k + 1}} \left( \frac{1 - w}{w \varphi(w)^{\kappa}} \right)^{n-k},\label{eq:Psi_equal_speeds}
\end{align}
where $\varphi$ is from Assum.~\ref{a:phi}, and the functions $\Phi^{n}_{n-k}$, for $k \in \set n$, are uniquely characterized by:
\begin{enumerate}[label=\uptext{(\arabic*)}]
\item\label{it:biorth-cor} The biorthogonality relation $\sum_{x\in\zz}\Psi_\ell^{n}(x)\Phi_k^{n}(x)=\uno{k=\ell}$, for each $k,\ell=0,\dotsc,n-1$;
\item $\spanning\{x \in \zz \longmapsto \Phi^{n}_{k}(x) : 0 \leq k < n\} = \spanning\{x \in \zz \longmapsto x^k \theta^{x} : 0 \leq k < n\}$.
\end{enumerate} 

\noindent In particular, in the setting of Thm.~\ref{thm:main}, the right hand side of \eqref{eq:M_formula} with $K$ defined by \eqref{eq:K_equal_speeds} with $\kappa = 0$ gives a formula for $\pp \bigl(\xx_{t}(n_i) > a_{i},\, i \in \set{m}\, \big|\, X_{0} = \vec y\bigr)$.
Moreover, if the additional condition \eqref{eq:backward-in-time-one} holds, then \eqref{eq:M_formula_caterpillars} holds for $t \geq 0$.
\end{cor}

In the corollary we have introduced an additional conjugation by $\theta^x$ in the kernel coming from \eqref{eq:KernelK} (see Rem.~\ref{rem:conj}).
This will be convenient in Sec.~\ref{sec:biorth}. As in Rem.~\ref{rem:biorth_general} we see that for $\rin<\theta< \rout$ the sum in \ref{it:biorth-cor} is convergent absolutely. 

\begin{proof}
Fix $t \geq \kappa (n_m-1)$.
Applying Thm.~\ref{thm:biorth_general} with starting times $T_i = - \kappa(i-1)$ and speeds $v_1 = \dotsm = v_N = 1$ to the space-like path $\cS = \{(n_i, t - \kappa (i-1)) : 1 \leq i \leq m \}$ with $t \geq \kappa (n_m-1)$, yields \eqref{eq:M_formula_caterpillars} with the kernel $\K$ replaced by the kernel $\wt\K(n_i, x_i; n_j, x_j) =-\phi^{(n_i, n_j)}(x_i,x_j) \uno{n_i < n_j} + \sum_{k = 1}^{n_j}\wt\Psi^{n_i}_{n_i - k}(x_i)\wt\Phi^{n_j}_{n_j - k}(x_j)$ with
\begin{align}
\textstyle \wt\phi^{(n_i, n_j)}(x_i, x_j) &\textstyle = \frac{1}{2\pi\I}\oint_{\gamma_\rhoin}\d w\, \frac{\varphi(w)^{\kappa(n_j-n_i)}}{w^{x_i - x_j - n_j + n_i + 1}}(1 - w)^{n_i-n_j},\label{eq:phi-pf-cat}\\
\textstyle \wt\Psi^{n}_{n-k}(x) &\textstyle = \frac{1}{2\pi\I}\oint_{\gamma_\rhoin}\d w\,\frac{\varphi(w)^{t + \kappa (k - 1)}}{w^{x -y_k + n - k + 1}}(1 - w)^{n-k},\label{eq:Psi-pf-cat}
\end{align}
and where the functions $\wt\Phi^{n}_{n-k}$, $ k \in \set n$, are characterized by $\sum_{x\in\zz} \wt\Psi_\ell^{n}(x) \wt\Phi_k^{n}(x)=\uno{k=\ell}$ for $k,\ell=0,\dotsc,n-1$, together with $\spanning\{x \in \zz \longmapsto \wt\Phi^{n}_{k}(x) : 0 \leq k < n \} = \V{n}(\vec v)$; note that $\V{n}$, defined in \eqref{eq:space}, here equals $\spanning\{x \in \zz \longmapsto x^k\!: 0 \leq k < n\}$.
Multiplying \eqref{eq:phi-pf-cat} by $\theta^{x_i - x_j}$ yields $Q^{n_j-n_i}(x_i,x_j)$. Multiplying by $\theta^{x - y_k}$ in \eqref{eq:Psi-pf-cat} we get \eqref{eq:Psi_equal_speeds}.
Therefore, multiplying the kernel $\wt\K$ by $\theta^{x_i - x_j}$ we obtain the kernel $\K$ in \eqref{eq:K_equal_speeds}, where the functions $\Phi^{n_j}_{n_j - k}(x_j)$ are equal to $\theta^{y_k - x_j}$ times the functions $\wt\Phi^{n_j}_{n_j - k}(x_j)$ defined above, which implies the listed properties.

Next we explain how \eqref{eq:M_formula_caterpillars} can be extended to the case $t < \kappa (n_m-1)$ (with $\kappa\geq1$) under the additional assumption \eqref{eq:backward-in-time-one}.
It is more convenient to do this for the distribution of all particles, i.e., in the case $m = N$ and $n_i = i$ for $i \in \set{N}$; the unnecessary particles can then be eliminated by taking respective values of $a_i$ to be $-\infty$.
For $\cS$ defined as before with $m = N$, instead of \eqref{eq:G+} we define the function 
\begin{equation}
G^{+}_{\cS} (\vec y, \vec x) = \sum_{\substack{\vec x(t_i) \in \Omega_{N - i + 1} : \\ x_{1}(t_{i}) = x_i, i \in \set{N}}} \G_{0, t_1}(\vec y, \vec x(t_1)) \prod_{i=1}^{N-1} \G_{t_i, t_{i + 1}}(\vec x_{> i}(t_{i}), \vec x(t_{i+1})),
\end{equation}
where $\vec y, \vec x \in \Omega_N$ are fixed. Using this function we define $G_{\vT, \cS}$ by \eqref{eq:G_TS}, where in this case we assume $T_1 \leq t_1$.
Applying \eqref{eq:backward-in-time-one} recursively one sees that $G_{\vT, \cS}(\vec y, \vec x)=\pp \bigl(\xx_{t - \kappa (i-1)}(i) = x_{i},\, i \in \set{N}\, \big|\, X_{- \kappa (N-1)} = \vec y, \CE_\kappa\bigr)$ by first evolving all particles up to time $t_1$ and then moving the particles one-by-one ``back in time''.
The biorthogonalization of this function can be proved by analogy with \eqref{eq:M_formula}.
\end{proof}

\subsection{Orthogonal polynomials perspective}

Thm.~\ref{thm:biorth_general} allows us to compute the marginals $\M_{\vT, \cS}(\vec y, \vec a)$ of $\G_{\vT,\cS}$ in terms of the \emph{biorthogonal ensemble}\footnote{This is the (signed) determinantal point process on a certain space of Gelfand-Tsetlin patterns having $K$ as its correlation kernel; see Appdx. \ref{app:biorth} for more details.} associated to $K$.
In applications to the classical particle systems considered in Secs.~\ref{sec:caterpillars} and \ref{sec:continuousTASEP}, the functions making up the kernel are related to classical families of orthogonal polynomials, and it is instructive to spell out in some detail what the biorthogonalization problem means in those cases.
We do this next, focusing on models with block dynamics (for push dynamics one sees the same polynomials but supported on the negative integers).

\subsubsection{Charlier polynomials}

As explained in Sec.~\ref{sec:TASEP}, continuous time TASEP corresponds to the model in Cor.~\ref{cor:biorth_caterpillars} with $\kappa=0$ and $\varphi(w)=e^{w-1}$. In this case the functions \eqref{eq:Psi_equal_speeds} can be written in terms of \emph{Charlier polynomials} (see \cite[Eqn.~7.4]{borFerPrahSasam} or \cite[Eqn.~9.14.9]{hypergeomOrthPolyn}), which are the family of discrete orthogonal polynomials $C_k(x, t)$ with respect to the Poisson weight $w_{t}(x) = e^{-t} \frac{t^x}{x!} \uno{x \geq 0}$ (with the usual normalization $C_k(x, t)=(-1/t)^kx^k+\dotsm$). From \cite[Eqn.~9.14.1]{hypergeomOrthPolyn} we readily conclude $C_k(x, t) = C_x(k, t)$. Then the contour integral formula for Charlier polynomials $C_k(x, t) = \frac{x!}{2\pi\I t^x}\oint_{\gamma_\rin}\d w\,\frac{(1 - w)^k}{w^{x + 1}} e^{t w}$, which follows from \cite[Eqn.~9.14.11]{hypergeomOrthPolyn}, leads to
\begin{equation}
\Psi^n_k(x)= \theta^{y_{n-k}-x} f_k(x-y_{n-k})\qquad\uptext{with}\quad  f_k(x)=C_k(x+k, t)w_{t}(x+k).\label{eq:TASEPpsi}
\end{equation}
For TASEP we can thus rephrase the biorthogonalization problem of Cor.~\ref{cor:biorth_caterpillars} as follows: 
\begin{quote}
Given a family of \emph{shifted Charlier functions} $\Psi^n_k(x)=f_k(x-y_{n-k})$, $k=0,\dotsc,n-1$, with $f_k$ as in \eqref{eq:TASEPpsi}, find a family of functions $\{\Phi^n_k\}_{k=0,\dotsc,n-1}$ on $\zz$, such that $\theta^{x}\Phi^n_k(x)$ is a polynomial of degree $k$, and $\{\Phi^n_k\}_{k=0,\dotsc,n-1}$ are biorthogonal to $\{\Psi^n_k\}_{k=0,\dotsc,n-1}$ (in the sense of \ref{it:biorth-cor} in Cor.~\ref{cor:biorth_caterpillars}).
\end{quote}
The solution to this biorthogonalization problem depends, of course, on the initial positions of the TASEP particles $(y_1,\dotsc,y_N)$.
The simplest choice in this setting is the \emph{packed} (also referred to as \emph{step}) initial condition $y_i=-i$.
In this case we get $\Psi^n_k(x)=  \theta^{k - n-x} C_k(x+n, t)w_t(x+n)$, and hence by definition the biorthogonalization problem is solved by the Charlier polynomials themselves: $\Phi^n_k(x) = \theta^{x - k + n} C_k(x+n, t)$. 

The packed initial condition had actually been solved earlier using different arguments (see e.g., \cite{johanssonShape}).
The goal of the authors in \cite{sasamoto,borFerPrahSasam} was to solve the \emph{periodic} initial condition $y_i=-2i$, and the biorthogonalization method introduced in those papers allowed the authors to achieve this by solving for the biorthogonal functions explicitly (see \cite[Appx.~2]{borFerPrahSasam}).
The solution for general (one-sided) initial data was discovered in \cite{fixedpt},  based on some additional properties of TASEP (its time reversal invariance together with the existence of the so-called \emph{path-integral} version of the Fredholm determinant formula, see \cite[Appx. D]{fixedpt} and \cite{bcr}) to produce an ansatz for the $\Phi^n_k$'s.
The goal of Sec.~\ref{sec:biorth} is to derive this solution in a much more general framework; for details about where the ansatz comes from we refer to \cite[Sec.~2.1]{fixedpt}.

\subsubsection{Krawtchouk polynomials}

In the case of discrete time TASEP with sequential update and Bernoulli jumps considered in Prop.~\ref{prop:caterpillars} we have $\kappa=0$ and $\varphi(w) = q + p w$.
Then the rephrasing of the biorthogonalization problem stated in the previous section holds for this model with Charlier polynomials replaced by \emph{Krawtchouk polynomials} \cite[Eqn.~9.11.1]{hypergeomOrthPolyn}, which are now orthogonal with respect to the binomial weight $w_{t}(x) = \binom{t}{x}p^x q^{t-x}$. From \cite[Eqn.~9.11.1]{hypergeomOrthPolyn} we have $K_n(x, p, T) = K_x(n, p, T)$. From this identity and \cite[Eqn.~9.11.11]{hypergeomOrthPolyn} we get the formula $ {T \choose x} K_n(x, p, T) = \frac{q^{n + x - N}}{p^n} \frac{1}{2\pi\I}\oint_{\gamma_\rin}\d w\,\frac{(1 - w)^n}{w^{x + 1}} (q + p w)^{T - n}$, which yields 
\begin{equation}
\Psi^n_k(x)= q^{-k} \theta^{y_{n-k}-x} f_k(x-y_{n-k})\qquad\uptext{with}\quad  f_k(x)=K_k(x+k, p, t+k) w_{t+k}(x+k).
\end{equation}
In the case of the periodic initial condition $y_i = -d (i-1)$ for $d \geq 2$ the functions $\Phi^n_k$ were computed in \cite{bfp}.

\subsubsection{Meixner polynomials}

TASEP with right geometric jumps, considered in Prop.~\ref{prop:RG_distribution}, is related to \emph{Meixner polynomials}. In this case we have $\kappa=0$ and $\varphi(w) = p / (1 - q w)$, while the weight is given by $w_{t}(x) = \binom{t+x-1}{x} q^{x}$. 
\cite[Eqns.~9.10.1,9.10.11]{hypergeomOrthPolyn} yield $w_{t}(x) M_n(x, t, q) = \frac{1}{2\pi\I}\oint_{\gamma_\rin}\d w\,\frac{(1 - w)^n}{w^{x + 1}} (1 - q w)^{-n-t}$, where the integration contour $\gamma_\rin$ does not include the pole at $w=1/q$. Then we have 
\begin{equation}
\Psi^n_k(x)= p^t \theta^{y_{n-k}-x} f_k(x-y_{n-k})\qquad\uptext{with}\quad  f_k(x)=M_k(x+k, t-k, q) w_{t-k}(x+k).
\end{equation}

\subsubsection{Hermite polynomials}

Systems of one-sided reflected Brownian motions also fall into the described framework, although, the state space of particles in this model is $\rr$ rather than $\zz$ (see \cite{Mihai} for more details). 
As explained in \cite[Rem.~5.2]{Mihai}, the functions $\Psi^n_k$ are equal in this model to shifted Hermite polynomials.

\section{An explicit biorthogonalization scheme}
\label{sec:biorth}

In this section we turn to the main goal of this paper, which is to develop a general scheme to, first, solve explicitly a version of the biorthogonalization problem defining the kernels in Sec.~\ref{sec:biorth_sub} and, second, rewrite the resulting kernel in a form which is in principle suitable for asymptotics, as was done in the particular case of continuous time TASEP in \cite{fixedpt}.

We will do this in a setting which is slightly different from the general one in Sec.~\ref{sec:measures}.
In fact, throughout the section we will focus only on kernels with a certain structure, and not on the general measures from which they arise in that section (in particular, the results here will be independent of those in Sec.~\ref{sec:measures}).
The kernels which we choose to work with will allow us to handle the setting of Cor.~\ref{cor:biorth_caterpillars} when $Q$ is of a specific form (satisfied by all the particle systems which we  consider), and will allow us to to prove Thms.~\ref{thm:main} and \ref{thm:main2} (the application to that section is presented in Sec.~\ref{sec:main-proof}).
But they will be presented and studied in a more general form, which will in particular also allow us to cover some situations---such as TASEP with right geometric jumps with sequential update (Sec.~\ref{sec:rightGeometric}) or with generalized update (Secs.~\ref{sec:Povolotsky-discrete} and \ref{sec:Povolotsky-continuous})---which are not covered by the setting of Sec.~\ref{sec:main}.
The extension to kernels corresponding to particles with different speeds and more general starting and ending times is left for future work.

\subsection{Setting}\label{sec:setting}

The general family of kernels which we will be interested in is made out of two main ingredients. The first one is a (strictly) positive measure on $\zz$, which we denote by $(q_i)_{i\in\zz}$, and which satisfies:

\begin{assumption}\label{assum:q}
There is a $\kappa\in\nn_0$ and a $\theta\in(0,1)$ such that:
\begin{enumerate}[label=\uptext{(\roman*)}]
\item $q_i=1$ for all $i>\kappa$,
\item $\sum_{i\in\zz}q_i\theta^i<\infty$.
\end{enumerate}
\end{assumption}

The geometric sequence $\theta^i$ will be used to normalize the measure defined by the $q_i$'s.
In applications to scaling limits, $\theta$ is related to the density of particles in the initial conditions under consideration, see Rem. \ref{rem:fp-dens}.
Using the $q_i$'s we introduce the following Laurent series:
\begin{equation}
a(w)=\sum_{i=-\infty}^\kappa(q_{i+1}-q_i)w^i.\label{eq:def-a}
\end{equation}

The second ingredient is a complex function $\psi$. We will make the following assumption on $a$ and $\psi$:

\begin{assumption}\label{assum:apsi}
There are radii $\rin$ and $\rout$ satisfying and $0<\rrin<\theta<1<\rout$ (with $\theta$ given in Assum.~\ref{assum:q}) such that $a(w)$ is analytic on $\{w\in\cc\!:|w|\geq\rrin\}$ while $1/a(w)$, $\psi(w)$ and $1/\psi(w)$ are analytic and non-zero on the annulus $A_{\rin,\rout}$.
\end{assumption}

Recall that, by the convention introduced in Sec.~\ref{sec:notat}, the assumption implies that $a(w)$ is actually analytic on an open domain $\{w\in\cc\!:|w|>\rrin-\ep\}$ for some $\ep>0$, and similarly that each of the last three functions are analytic on an open annulus $\{w\in\cc\!:\rin-\ep<|w|<\rout+\ep\}$ for some $\ep>0$.

Throughout the rest of this section we will assume that the two preceding assumptions are satisfied.

\begin{rem}
In the setting of Sec.~\ref{sec:main}, the $q_i$'s from Assum.~\ref{assum:q} are those appearing in \eqref{eq:Qalphaq-intro}, while $\theta$ plays the same role as in that section.
The complex function $\psi(w)$, on the other hand, plays the role of $\varphi(w)^{t}$ in \eqref{eq:F-main}.
Hence in that context, both the $q_i$'s and $\psi$ are determined by the function $\varphi$.
The setting of this section extends that of Sec.~\ref{sec:main} by decoupling that dependence.
\end{rem}

\begin{rem}\label{rem:rrin-new} 
The arguments of this section can be extended, with no essential difference, to the case where only $a$ is asked to be analytic in an annulus including $\theta$ (i.e., allowing $\psi$ to be analytic and non-zero in an annulus $A_{\rhoin,\rout}$ for some $\rho<1$ which is not necessarily smaller than $\theta$).
This extension may be useful in the application to some models with $\kappa\geq1$ if one wants the kernels of the form \eqref{eq:kernel-main} to be defined with as broad a range of parameters $r$ and $\theta$ as possible.
But it yields no improvement in any of the cases we are interested in, and in any case similar extensions can be achieved by deriving the kernels under our assumptions and then extending the validity of the final answer directly to a broader range of parameters (see e.g. the proof of Thm.~\ref{thm:kernel-rw} and Rem.~\ref{rem:assumext}).
Therefore, and since they lead to a cleaner presentation, we have opted to work with these slightly more restricted assumptions.
\end{rem}

Out of the two ingredients we just introduced we will construct the kernels which will show up in the general result of this section.
We begin with the kernel associated to the measure $(q_i)_{i\in\zz}$.
Note that $a(\theta)=\theta^{-1}(1-\theta)\sum_{i\in\zz}\theta^iq_i>0$, so in particular we may define
\begin{equation}\label{eq:alpha}
\alpha=\frac{1-\theta}{a(\theta)\theta}=\frac{1}{\sum_{i\in\zz}\theta^iq_i}.
\end{equation}
Note also that $a(1)=1$.
We introduce a Markov transition matrix $Q$ on $\zz$ built out of $\theta$ and the $q_i$'s as follows:
\begin{equation}\label{eq:defQ}
Q(x,y)=\alpha\tts\theta^{x-y}q_{x-y}.
\end{equation}
Our assumption on the $q_i$'s means that 
\begin{equation}
 Q(x,y)=\alpha\tts\theta^{x-y}\quad\text{for}~~x-y>\kappa,\label{eq:Qgeom}
\end{equation} that is, jumps to the left of size larger than $\kappa$ are geometrically distributed, with parameter $1-\theta$.

An important special case is $q_i=\uno{i\geq1}$, for which the above definition means that $a(w)\equiv1$, $\alpha=(1-\theta)\theta^{-1}$, and $Q=\tilde Q_0$ with $\tilde Q_0(x,y)=(1-\theta)\theta^{x-y-1}\uno{x>y}$, which is the transition matrix of a random walk which takes Geom$[1-\theta]$ steps to the left.
$Q$ can be thought of as a version of the transition matrix $\tilde Q_0$ where the transition probabilities for steps of size greater than or equal to $-\kappa$ are modified arbitrarily (with the only restriction, from Assum.~\ref{assum:q}, that $\sum_{\ell\geq-\kappa} Q(0,\ell)$ be finite).

A useful way to think of $Q$ is as follows:
\begin{equation}
\label{eq:AQ0}
Q=AQ_0=Q_0A
\end{equation}
with
\begin{equation}
A(x,y)=\alpha\tts\theta^{x-y}(q_{x-y+1}-q_{x-y})\label{eq:defA}
\end{equation}
and
\begin{equation}\label{eq:Q0}
Q_0(x,y)=\theta^{x-y}\uno{x>y}
\end{equation}
(this follows directly from a telescopic sum, using that $q_\ell\longrightarrow0$ as $\ell\to-\infty$ by Assum.~\ref{assum:q}).
$Q_0$ is an unnormalized version of the transition kernel of the pure geometric random walk $\tilde Q_0$ introduced in the last paragraph.
On the other hand, if the $q_i$'s are non-decreasing then, modulo normalization, $A$ is also the transition kernel of a random walk, in which case $Q$ can be thought of as the transition kernel of the random walk obtained by convolving the other two kernels.

$Q$ and $A$ can also be expressed through the following contour integral formulas:
\begin{equation}\label{eq:contourQA}
Q(x,y) = \frac{\alpha}{2\pi\I}\oint_{\gamma_\rrin}\d w\,\frac{\theta^{x-y}}{w^{x-y}}\frac{a(w)}{1-w},\qquad
A(x,y) = \frac{\alpha}{2\pi\I}\oint_{\gamma_\rrin}\d w\,\frac{\theta^{x-y}}{w^{x-y+1}}a(w)
\end{equation}
(we prove these identities in Prop.~\ref{prop:AQRinv} below).
Note that the first kernel coincides with $Q$ from Cor.~\ref{cor:biorth_caterpillars} if $a(w)=\varphi(w)^\kappa$.

Similarly, using now the function $\psi$ we define the kernel
\begin{equation}
\R(x,y)=\frac1{2\pi\I}\oint_{\gamma_\rin}\d w\,\frac{\theta^{x-y}}{w^{x-y+1}}\psi(w).\label{eq:defR}
\end{equation}
Recall that we also regard these kernels as operators acting on a suitable space of functions defined on $\zz$, which in this case can be taken to be $\ell^1(\zz)$:

\begin{prop}\label{prop:AQRinv}
$Q$, $A$ and $\R$ are continuous as operators mapping $\ell^1(\zz)$ to itself, and are invertible there.
Moreover the three operators and their inverses all commute, and they have kernels given by
\begin{align}
Q^k(x,y) &= \frac{\alpha^k}{2\pi\I}\oint_{\gamma_\rrin}\d w\,\frac{\theta^{x-y}}{w^{x-y-k+1}}\left(\frac{a(w)}{1-w}\right)^k,\label{eq:contourQ}\\
A^k(x,y) &= \frac{\alpha^k}{2\pi\I}\oint_{\gamma_\rrin}\d w\,\frac{\theta^{x-y}}{w^{x-y+1}}a(w)^k\label{eq:Apow}
 \end{align}
for any $k\in\zz$, as well as
\begin{equation}
\R^{-1}(x,y)=\frac1{2\pi\I}\oint_{\gamma_{\rin}}\d w\,\frac{\theta^{x-y}}{w^{x-y+1}}\frac1{\psi(w)}.\label{eq:defRinv}
\end{equation}
The analogous statements hold for $Q^*$, $A^*$ and $\R^*$.
\end{prop}

The above formulas for inverses and powers of $Q$, $A$ and $\R$ follow from the following simple result, which we will use repeatedly:

\begin{lem}\label{lem:conv}
Consider two kernels $S_1$ and $S_2$ given by
\begin{equation}
S_i(x,y)=\frac1{2\pi\I}\oint_{\gamma}\d w\,\frac{\theta^{x-y}}{w^{x-y+1}}\phi_i(w),\label{eq:Si}
\end{equation}
where $\phi_1,\phi_2$ are complex functions which are both analytic on an annulus $A_{r_1,r_2}$ for some $r_1<r_2$ and $\gamma$ is any simple closed contour contained in $A_{r_1,r_2}$.
Then the sum defining the product $S_1S_2$ is absolutely convergent and
\begin{equation}\label{eq:conv}
S_1S_2(x,y)=\frac1{2\pi\I}\oint_{\gamma}\d w\,\frac{\theta^{x-y}}{w^{x-y+1}}\phi_1(w)\phi_2(w).
\end{equation}
\end{lem}

\begin{proof}
We need to compute
\begin{equation}
\sum_{z\in\zz}\frac{1}{(2\pi\I)^2}\oint_{\gamma}\d w\oint_{\gamma}\d u\,\frac{\theta^{x-y}}{w^{x-z+1}u^{z-y+1}}\phi_1(w)\phi_2(u).
\end{equation}
For the sum over $z<y$ we deform the $u$ contour to $\gamma_{r_1}$ and the $w$ contour to $\gamma_{r_2}$ so that $|w/u|>1$.
The summand can be bounded in absolute value by a constant times $|w|^{z-x-1}/|u|^{z-y+1}$ so this part of the sum is absolutely convergent and after computing the geometric sum we get 
\begin{equation}\label{eq:zlesssum}
\frac1{(2\pi\I)^2}\oint_{\gamma_{r_2}}\d w\oint_{\gamma_{r_1}}\d u\,\frac{\theta^{x-y}}{w^{x-y+1}}\frac{1}{w-u}\phi_1(w)\phi_2(u).
\end{equation}
For the sum over $z\geq y$ we proceed similarly, now deforming the $u$ contour to $\gamma_{r_2}$ and the $w$ contour to $\gamma_{r_1}$; the sum is again absolutely convergent and computing the geometric sum we get now  $-\frac1{(2\pi\I)^2}\oint_{\gamma_{r_1}}\d w\oint_{\gamma_{r_2}}\d u\,\frac{\theta^{x-y}}{w^{x-y+1}}\frac{1}{w-u}\phi_1(w)\phi_2(u)$.
Now we deform $u$ contour to $\gamma_{r_1}$ and the $w$ contour to $\gamma_{r_2}$, and as we do this we pick up a pole at $w=u$, and computing the residue yields $\frac{1}{2\pi\I}\oint_{\gamma_{r_1}}\d u\,\frac{\theta^{x-y}}{u^{x-y+1}}\psi_1(u)\psi_2(u)$, which is exactly what we want.
There remains the double integral after having flipped the contours, but it cancels the double integral which we got in \eqref{eq:zlesssum}.
\end{proof}

\begin{proof}[Proof of Prop.~\ref{prop:AQRinv}]
We begin by proving \eqref{eq:contourQA}.
The right hand side of the second identity may be written as $\frac{\alpha}{2\pi\I}\oint_{\gamma_{\rrout}}\d w\frac{\theta^{x-y}}{w^{x-y+1}}\sum_{i\leq\kappa}(q_{i+1}-q_i)w^i$, where we have expanded the contour to $\gamma_{\rrout}$ (which we may do since $a(w)$ is analytic for $|w|\geq r$).
Then we can interchange the sum and the integral, and we get \eqref{eq:defA} from a straightforward application of Cauchy's integral formula.
For the first one we first note that $Q_0(x,y)=\frac{1}{2\pi\I}\oint_{\gamma_\rrin}\d w\frac{\theta^{x-y}}{w^{x-y} (1-w)}$, which is again Cauchy's integral formula. 
Then \eqref{eq:AQ0} and Lem.~\ref{lem:conv} yield the formula \eqref{eq:contourQA}.

That $Q$ acts continuously on $\ell^1(\zz)$ is straightforward, being a Markov kernel acting by convolution (i.e., $Q$ is Toeplitz), and essentially the same argument works for $A$.
In any case, the general argument works for $Q$, $A$, $\R$ and its inverses (where the inverse kernels are those defined in \eqref{eq:contourQ}, \eqref{eq:Apow} and \eqref{eq:defRinv}).
Observe that they all can be expressed as a kernel of the form \eqref{eq:Si}
\begin{equation}\label{eq:Sgen}
S(x,y)=\frac1{2\pi\I}\oint_{\gamma}\d w\,\frac{\theta^{x-y}}{w^{x-y+1}}\phi(w)
\end{equation}
for some complex function $\phi$ which is analytic in an annulus $A_{r_1,r_2}$ for some $r_1<\theta<r_2$, with $\gamma$ a simple closed contour contained in this annulus.
By choosing $\gamma$ to be a circle of radius either larger or smaller than $\theta$, depending on whether $x>y$ or $x\leq y$, we get that there are constants $c,C>0$ such that $|S(x,y)|\leq C\tts e^{-c|x-y|}$, and from this we get easily that for $f\in\ell^1(\zz)$, $\|Sf\|_1\leq\tilde C\|f\|_1$ for some other constant $\tilde C>0$, as desired.

The contour integral formula \eqref{eq:contourQ} for $k>0$ follows directly from \eqref{eq:contourQA}, Lem.~\ref{lem:conv}, and Assum.~\ref{assum:apsi}.
The same argument shows that, with the definition in that formula, $Q^{-1}Q(x,y)=QQ^{-1}(x,y)=\frac1{2\pi\I}\oint_{\gamma_\rin}\d w\,\frac{\theta^{x-y}}{w^{x-y+1}}=\uno{x=y}$, as desired.
\eqref{eq:contourQ} for $k\leq0$ now follows in the same way, and the same argument gives \eqref{eq:Apow} and \eqref{eq:defRinv}.
That the operators commute follows from the facts that they act by convolution and that the sums involved in their compositions can be interchanged by the estimate used in the last paragraph.
\end{proof}

In Prop.~\ref{prop:AQRinv} we have defined $Q^{-1}$, $A^{-1}$ and $\R^{-1}$ as integral operators with explicit kernels and identified them as the inverses in $\ell^1(\zz)$ of $Q$, $A$ and $\R$.
In the next section we will see that crucially, and except for $Q$, these integral operators also act on functions of the form $\theta^xp(x)$ for $p$ a polynomial.

\begin{rem}
In the case of continuous time TASEP considered in \cite{fixedpt} (see Sec.~\ref{sec:TASEP}), the above kernels have $\psi(w)=e^{t(w-1)}$ and $a(w)\equiv1$, which in particular means $\kappa=0$.
This entails several simplifications in the arguments of this whole section.
One example is that in that case $\psi$ and $a$ are analytic and non-zero in all of $\cc$, simplifying many computations.

\noindent The key difference is that the kernel $Q$ in that case is just the geometric random walk kernel $Q_0$ in \eqref{eq:Q0} (after normalization).
The lack of memory of the geometric distribution is used crucially on \cite{fixedpt}.
The form of $Q_0$ implies also that, whereas $Q_0^{-1}$ is simply a discrete difference operator (see \eqref{eq:Q0inv}), our $Q^{-1}$ in general has infinite range.
In particular, while the boundary value problem which will appear in \eqref{bhe} below can be solved, in the case of $Q_0$, by writing down a solution ``below the curve'' in terms of random walk hitting times and then simply extending it analytically to all of $\zz$, in our case we will need to construct the solution explicitly.
Moreover, $Q_0(x,y)$ itself is, as a function of $x-y$, the truncation of an analytic function, and can thus be extended analytically to all $x,y\in\zz$, but this does not hold for $Q$.
Throughout the argument we will have to account for the difference between our $Q$ and $Q_0$, which is where \eqref{eq:AQ0} will be useful.
\end{rem}

\subsection{The biorthogonalization problem}\label{sec:biorthproblem}

Let $\bn\in\nn$, which will remain fixed throughout the rest of this section.
In applications to settings such as those of Sec.~\ref{sec:main}, $\bn$ can be taken to be the number of particles in the system (i.e., $\bn=N$) or, more precisely, the label of the leftmost particle which one is interested in.
We also fix a vector $\vec y\in\zz^\bn$, which plays the role of the initial data $X_0$ in Sec.~\ref{sec:main}.\footnote{We could also take $\bn=\infty$ and consider instead a sequence $(y_i)_{i\geq1}$ as initial data, but this does not make any difference, since in applications we are always interested in the evolution of a finite number of particles (see in particular the comment after Assum.~\ref{a:phi}).}

In terms of these ingredients we define, for $n\in\set{\bn}$ and $n-\bn\leq k \leq n-1$,
\begin{equation}
\Psi^n_k(x)=\R Q^{-k}(x,y_{n-k})=\frac{\alpha^{-k}}{2\pi\I}\oint_{\gamma_{\rin}}\d w\,\frac{\theta^{x-y_{n-k}}}{w^{x-y_{n-k}+k+1}}\left(\frac{1-w}{a(w)}\right)^{\!k}\!\psi(w)\label{eq:defPsink}
\end{equation}
(the contour integral formula follows from Lem.~\ref{lem:conv}, using \eqref{eq:defR} and \eqref{eq:contourQ}) and then consider a family of functions $\{\Phi^n_k\}_{k=0,\dotsc,n-1}$ characterized by:
\begin{itemize}
\item[($\star$)] \hypertarget{it:biorth4}{The} biorthogonality relation $\sum_{x\in\zz}\Psi_\ell^{n}(x)\Phi_k^{n}(x)=\uno{k=\ell}$ for each $k,\ell=0,\dotsc,n-1$.
\item[($\star\star$)] \hypertarget{it:poly4}{$\theta^x\Phi^n_k(x)$} is a polynomial of degree $k$ in $x$.
\end{itemize}

Our first task is to show that $\{\Phi^n_k\}_{k=0,\dotsc,n-1}$ is uniquely defined.
For a given $n\in\nn$ we introduce the vector spaces
\begin{align}\label{eq:space2}
\W{n} &= \spanning \bigl\{x\in\zz\longmapsto\theta^xx^k\!:0\leq k<n\bigr\},\\
\Wd{n} &= \spanning \bigl\{x\in\zz\longmapsto\theta^{-x}x^k\!:0\leq k<n\bigr\}
\end{align}
($\W{n}$ coincides with $\V{n}((\theta,\dotsc,\theta))$ from \eqref{eq:space}).
We will say that a function $g\in\W{n}$ (respectively $g\in\Wd{n}$) has \emph{degree $k$} if it equals $\theta^x$ (respectively $\theta^{-x}$) times a polynomial of degree $k$.
We begin with the following result:

\begin{prop}\label{prop:AQprops}
\leavevmode
\begin{enumerate}[label=\uptext{(\roman*)}]
\item The operators $Q^{-1}$, $A$, $A^{-1}$, $\R$ and $\R^{-1}$ map $\W{n}$ to itself.
Moreover, they commute as operators acting on this space and the formulas in Prop.~\ref{prop:AQRinv} for the kernels of their powers hold.
\item Let $g\in\W{n}$ have degree $k$.
Then $\R g$, $\R^{-1}g$ and $A^{-1}g$ have degree $k$, while for any $\ell\in\set{k}$, $Q^{-\ell}g$ has degree $k-\ell$.
\item Let $g\in\W{n}$.
Then for $\ell\in\{0,\dotsc,n-1\}$, $Q^{-\ell}g=0$ if and only if $\ell$ is larger than the degree of $g$.
\end{enumerate}
The analogous statements also hold for $\Qt$, $A^*$, $\R^*$ and their inverses, with $\W{n}$ replaced by $\Wd{n}$.
\end{prop}

\begin{proof}
The five operators in (i) have kernels of the form $S(x,y)$ as in \eqref{eq:Sgen}, with $\phi$ analytic on $A_{\rin,\rout}$ thanks to Assum.~\ref{assum:apsi} (note that this fails for $Q$ due to the factor $1-w$ in the denominator).
For such a kernel and $g\in\W{n}$, we have
\begin{equation}\label{eq:RAkg}
Sg(x)=\sum_{\eta\in\zz}\frac{1}{2\pi\I}\oint_{\gamma_{\rin}}\d w\,\frac{\theta^{x-\eta}\tts g(\eta)}{w^{x-\eta+1}}\phi(w)=\theta^{x}\sum_{\eta\in\zz}\frac{1}{2\pi\I}\oint_{\gamma_{\rin}}\d w\,\frac{p(x-\eta)}{w^{\eta+1}}\phi(w)
\end{equation}
where $p(x)=\theta^{-x}g(x)$, which is thus a polynomial of degree strictly less than $n$.
The sum over $\eta<0$ is absolutely convergent, because $\rin<1$, while for the sum over $\eta\geq0$ we may enlarge the contour to $\gamma_{\rout}$ and get again that it is absolutely convergent.
From this it follows also that $Sg(x)=\theta^x\tilde p(x)$ for some other polynomial $\tilde p$ of degree strictly less than $n$.
The rest of (i) follows in a similar way.

To prove (ii) let first $S$ be any of the operators $\R$, $\R^{-1}$ or $A^{-1}$.
From \eqref{eq:RAkg}, if the $x^k$ coefficient of $g(x)$ is $b_k\neq0$ then the $x^k$ coefficient of $Sg(x)$ equals $b_k$ times
$\sum_{\eta\in\zz}\frac{1}{2\pi\I}\oint_{\gamma_{\rin}}\d w\,\frac{1}{w^{\eta+1}}\phi(w)$.
We separate the sum again between $\eta<0$ and $\eta\geq0$.
The first part yields $-\frac{1}{2\pi\I}\oint_{\gamma_{\rin}}\d w\,\frac{\phi(w)}{w-1}$, while the second part (after enlarging the contour to $\gamma_{\rout}$) yields $\frac{1}{2\pi\I}\oint_{\gamma_{\rout}}\d w\,\frac{\phi(w)}{w-1}$ which, after shrinking the contour back to $\gamma_{\rin}$ cancels the first part and leaves us with the residue at $w=1$, i.e., $\phi(1)$.
This is non-zero in each of the cases under consideration, because $a(1)=1$ while $\psi(1)\neq0$ by Assum.~\ref{assum:apsi} (since $1/\psi(w)$ is analytic at $w=1$).
This gives the first part of (ii).
For the other part we note that
\begin{equation}\label{eq:Q0inv}
Q_0^{-1}(x,y)=\theta^{-1}\cdot\uno{x=y-1}-\uno{x=y},
\end{equation}
so by \eqref{eq:AQ0} we have $Q^{-1}g(x)=Q_0^{-1}A^{-1}g(x)=Q_0^{-1}\tilde g(x)$ where $\tilde g\in\W{n}$ has degree $k$.
Writing $\tilde g(x)=\theta^x\tilde p(x)$ we then have $Q^{-1}g(x)=\theta^{x}(\tilde p(x+1)-\tilde p(x))$, which has degree $k-1$.
The statement for $Q^{-\ell}g$, $\ell\in\set{k}$, follows by repeating the argument inductively.
This yields (ii).

For (iii), let $k$ be the degree of $g$ and note first that, from (ii), we know that $Q^{-\ell}g\not\equiv0$ if $\ell<k$.
The same argument shows that the same statement holds for $\ell=k$.
To handle the case $k<\ell$ it is enough to show that $Q^{-\ell}g\equiv0$ for $g(x)=\theta^{x}x^k$.
But from (ii) we know that $A^{-\ell}g(x)$ has degree $k$, and proceeding as above we get inductively that for, $\ell'\leq k$, $Q_0^{-\ell'}A^{-\ell}g(x)$ has degree $k-\ell'$.
In particular $Q_0^{-k}A^{-\ell}g(x)=c\ts\theta^x$ for some constant $c$, and then \eqref{eq:Q0inv} again shows that $Q^{-\ell}g=Q_0^{-(\ell-k)}(Q_0^{-k} A^{-\ell})g=0$.
\end{proof}

\begin{cor}\label{cor:biorth-unique}
The biorthogonalization problem \bioneref--\bitworef has a unique solution.
\end{cor}

\begin{proof}
We may solve the system by finding $\Phi^n_k$ separately for each $k\in\{0,\dotsc,n-1\}$.
Fixing such a $k$, we need to show that there is a unique polynomial $p(x)=b_0+b_1x+\dotsm+b_kx^k$ such that if we let $\Phi^n_k(x)=\theta^{-x}p(x)$, then $\Phi^n_k(x)$ satisfies \bioneref.
Setting $g(x)=\theta^xp(-x)$ we have by \eqref{eq:defPsink} that
\[\textstyle\sum_{x\in\zz}\Psi^n_\ell(x)\Phi^n_k(x)=\R Q^{-\ell}g(-y_{n-\ell}).\]
In particular, Prop.~\ref{prop:AQprops}(iii) implies that $\R Q^{-\ell}g\equiv0$ for $\ell>k$ as needed.
Furthermore, for $\ell\leq k$ the same fact implies that $\R Q^{-\ell}g(-y_{n-\ell})$ only depends on the coefficients $b_{\ell},\dotsm,b_k$, and the arguments in the proof of Prop.~\ref{prop:AQprops} show that in fact $\R Q^{-\ell}g(-y_{n-\ell})$ is a linear combination of these coefficients.
In other words, there is an upper triangular matrix $\Lambda$ of size $(k+1)\times(k+1)$ such that $\Phi^n_k(x)$ satisfies \bioneref for $0 \leq \ell \leq k$ if and only if $\Lambda b=e^{(k)}$ with $e^{(k)}_\ell=\uno{\ell=k}$.
Existence and uniqueness then follows from the invertibility of $\Lambda$, which holds because it is upper triangular and $\Lambda_{\ell,\ell}\neq0$ for each $0 \leq \ell \leq k$; in fact, $\Lambda_{\ell,\ell}=\R Q^{-\ell}g_\ell(-y_{n-\ell})$ for $g_\ell(x)=\theta^x(-x)^\ell$, and by Prop.~\ref{prop:AQprops} again this application yields a non-zero constant.
\end{proof}

Finally we define the (extended) kernel
\begin{equation}\label{eq:K-schutz}
K(n_i,x_i;n_j,x_j)=-Q^{n_j-n_i}(x_i,x_j)\uno{n_i<n_j}+\sum_{k=1}^{n_j}\Psi^{n_i}_{n_i-k}(x_i)\Phi^{n_j}_{n_j-k}(x_j)
\end{equation}
for $n_i,n_j\in\set{\bn}$ and $x_i,x_j\in\zz$, which is our main object of interest.

\subsection{The boundary value problem}\label{sec:bdvpb}

Our goal now is to find an explicit solution of the biorthogonalization problem defined by \bioneref and \bitworef above.
The main idea is to consider the following initial--boundary value problem for the backwards discrete heat equation: for fixed $0\leq k<n\leq\bn$,
\mathtoolsset{showonlyrefs=false}
\begin{subnumcases}{\label{bhe}}
(\Qt)^{-1}h^n_k(\ell,z)=h^n_k(\ell+1,z) &  $\ell<k,\,z \in \zz$;\label{bhe1}\\ 
h^n_k(k,z)=\theta^{y_{n-k}-z}& $z \in \zz$;\label{bhe2}\\ 
h^n_k(\ell,y_{n-\ell})= 0 & $\ell<k$;\label{bhe3}
\end{subnumcases} 
(here $h^n_k(\ell,z)$ is defined for $0 \leq \ell \leq k$ and $z\in\zz$).

We remark that, for the solution of the above system, the identity $\Qt h^n_k(\ell+1,z)=h^n_k(\ell,z)$ does not hold in general.
In fact, using the terminal condition \eqref{bhe2} one can see directly that in general $\Qt h^n_k(k,z)$ diverges.
In particular, this means that the solution of \eqref{bhe} cannot be obtained by simply applying $\Qt$ repeatedly to $h^n_k(k,\cdot)$.
Well-posedness in our setting will be a consequence of Prop.~\ref{prop:AQprops}, as it will be enough to restrict a priori to solutions in $\Wd{n}$.
\mathtoolsset{showonlyrefs=true}

\begin{prop}\label{prop:pnkhnk}
The system \eqref{bhe} has a unique solution $\big\{h^n_k(\ell,z),\,\ell=0,\dotsc,k,\,z\in\zz\big\}$ in $\Wd{n}$ (i.e., such that for each $\ell$, $h^n_k(\ell,\cdot)$ is in $\Wd{n}$).
Moreover this solution is such that $h^n_k(\ell,\cdot)$ has degree $k-\ell$ for each $\ell$.
\end{prop}

\begin{proof}
The solution at time $\ell=k$ is prescribed by \eqref{bhe2}, so it is in $\Wd{n}$ and has degree $0$ as needed.
Now we proceed by induction, backwards in $\ell$.
Suppose that we have constructed the solution uniquely in $\Wd{n}$ at times $\ell+1,\dotsc,k$ for some $\ell<k$, and that this solution has the desired degrees.
We need to show that there is a unique $g\in\Wd{n}$ such that $(\Qt)^{-1}g=h^n_k(\ell+1,\cdot)$ and $g(y_{n-\ell})=0$ (so that if we set $h^n_k(\ell,\cdot)=g$ then this choice satisfies \eqref{bhe1} and \eqref{bhe3}), and that this $g$ has degree $k-\ell$ .
We do this next.

Let $g\in\Wd{n}$.
Since we want $(\Qt)^{-1}g(z)$ to equal $h^n_k(\ell+1,z)$, which has degree $k-\ell-1$ by the inductive hypothesis, Prop.~\ref{prop:AQprops} implies that $g$ has degree $k-\ell$, so we may write it as $g(z)=\theta^{-z}(b_0+b_1z+\dotsm+b_{k-\ell} z^{k-\ell})$.
Now, by Prop.~\ref{prop:AQprops} again, we have $(\Qt)^{-1}g(z)=\theta^{-z}p(z)$ for a polynomial $p$ of degree $k-\ell-1$ which does not depend on $b_0$.
Moreover, the arguments which we used to prove Cor.~\ref{cor:biorth-unique} show that $b_1,\dotsc,b_{k-\ell}$ can be chosen in such a way that $\theta^{-z}p(z)=h^n_k(\ell+1,z)$.
Having made this choice we have $(\Qt)^{-1}g(z)=h^n_k(\ell+1,z)$ as desired, and now we may adjust the free parameter $b_0$ so that $g(y_{n-\ell})=0$ as well.
To see that this choice of $g\in\Wd{n}$ is unique, suppose $\tilde{g}\in\Wd{n}$ also satisfies the necessary conditions.
We have $(\Qt)^{-1}(g-\tilde g)=0$, so Prop.~\ref{prop:AQprops}(iii) implies that $g-\tilde g$ has degree $0$.
But $g(y_{n-\ell})=\tilde g(y_{n-\ell})$, so in fact $g-\tilde g\equiv0$ as desired.
\end{proof}

\begin{thm}\label{thm:h_heat_Q}
The solution of the biorthogonalization problem \bioneref--\bitworef with respect to $\big(\Psi^n_k\big)_{k=0,\dotsc,n-1}$ stated in Sec.~\ref{sec:setting} is given by $\big(\Phi^n_k\big)_{k=0,\dotsc,n-1}$ with 
\begin{equation}\label{eq:h_heat_Q}
\Phi^n_k(x)=(\R^*)^{-1}h^n_k(0,x),
\end{equation}
where $h^n_k(\ell, z)$, $0\leq\ell\leq k$, $z\in\zz$, is the unique solution of \eqref{bhe} prescribed in Prop.~\ref{prop:pnkhnk}.
\end{thm}

\begin{proof}
The argument is essentially the same as the proof of biorthogonality in \cite[Thm.~2.2]{fixedpt}.
The polynomial condition \bitworef holds by construction and Prop.~\ref{prop:AQprops}. To prove \bioneref we write
\begin{align}\label{eq:psiphi}
  \textstyle\sum_{x\in\zz}\Psi^n_\ell(x)\Phi^n_k(x)&\textstyle=\sum_{x\in\zz}\R Q^{-\ell}(x,y_{n-\ell})(\R^*)^{-1}h^{n}_{k}(0,x)=\R^*(\Qt)^{-\ell}(\R^*)^{-1}h^n_k(0,y_{n-\ell})\\
  &\textstyle=(\Qt)^{-\ell}h^n_k(0,y_{n-\ell}),
\end{align}
where we have used Prop.~\ref{prop:AQprops} again.
For $\ell\leq k$ we use \eqref{bhe1} to write the right hand side as $h^n_k(\ell,y_{n-\ell})$.
If $\ell<k$ the boundary condition \eqref{bhe3} now shows that this is $0$ as needed, while, for $\ell=k$ we get, using \eqref{bhe2}, that the right hand side equals $h^n_k(k,y_{n-k})=1$.
Finally for $\ell>k$ we use \eqref{bhe1} to write
$(\Qt)^{-\ell}h^n_k(0,y_{n-\ell})=(\Qt)^{-(\ell-k-1)}(\Qt)^{-1}h^{n}_{k}(k,y_{n-\ell})$, 
which vanishes thanks to \eqref{bhe2} and Prop.~\ref{prop:AQprops}(iii).
\end{proof}

\subsection{Representation in terms of random walk hitting times}\label{sec:rw}

Thm.~\ref{thm:h_heat_Q} provides us with a characterization of the functions $(\Phi^n_k)_{k=0,\dotsc,n-1}$ which appear in the construction \eqref{eq:K-schutz} of the kernel $K$.
In this section we will provide a probabilistic representation for $K$.
Instead of working with the whole \emph{extended kernel} $K$, we will work with the \emph{one-point kernel}
\begin{equation}\label{eq:Kn}
K^{(n)}(z_1,z_2)=K(n,z_1;n,z_2)=\sum_{k=0}^{n-1}\Psi^n_k(z_1)\Phi^n_k(z_2),
\end{equation}
defined for any $n\in\set{\bn}$.
There is no loss of generality in this simplification because, using \eqref{eq:defPsink},
\begin{equation}\label{eq:KextKn}
K(n_i,\cdot;n_j;\cdot) = -Q^{n_j-n_i}\uno{n_i<n_j}+Q^{n_j-n_i}K^{(n_j)}.
\end{equation}

Let
\begin{equation}
G_{0,n}(z_1,z_2)=\sum_{k=0}^{n-1}Q^{n-k}(z_1,y_{n-k})h^n_k(0,z_2)\label{eq:defG0n}
\end{equation}
so that, from \eqref{eq:defPsink} and Thm.~\ref{thm:h_heat_Q},
\begin{equation}
K^{(n)}=\R Q^{-n}G_{0,n}\R^{-1}.\label{eq:KnG0n}
\end{equation}
We will use now the decomposition \eqref{eq:AQ0} to define a certain extension of $Q^\ell$ for $\ell\in\nn$.
It is based on an extension of $Q_0^\ell$ employed in \cite{fixedpt}, which is defined as follows:
\begin{equation}\label{eq:contourmQ0}
\mQ_0^{(\ell)}(z_1,z_2)= \frac{1}{2\pi \I} \oint_{\gamma_\delta}\d v\,\frac{\theta^{z_1-z_2}(1+v)^{z_1 - z_2 -1}}{v^\ell}= \theta^{z_1 - z_2}\frac{(z_1 - z_2 - 1)_{\ell-1}}{(\ell-1)!},
\end{equation}
with $\delta\in(0,1)$ (so that the contour does not include $-1$), where $(x)_\ell = x (x-1) \dotsm (x-\ell+1)$ for $\ell > 0$ and $(x)_0=1$ is the \emph{Pochhammer symbol}.
The point is that for every fixed $z_1$, $\mQ_0^{(\ell)}(z_1,z_2)$ is in $\Wd{\ell-1}$ as a function of $z_2$, and that (as can be seen from \eqref{eq:contourQ} with $a(w)\equiv1$, see also \cite[Eqn.~2.23]{fixedpt})
\begin{equation}
\mQ_0^{(\ell)}(z_1,z_2)=Q_0^\ell(z_1,z_2)\quad\uptext{ for $z_1,z_2\in\zz$,}\quad z_1-z_2\geq1;\label{eq:mQ0ellext}
\end{equation}
we regard thus $\mQ_0^{(\ell)}$ as a polynomial extension of $Q_0^\ell$ from $z_2\leq z_1-1$ to all $z_2\in\zz$.
Using $\mQ_0^{(\ell)}$, and in view of \eqref{eq:AQ0}, we define the extension of $Q^\ell$ as follows:
\begin{equation}\label{eq:QExt}
\begin{aligned}
\mQ^{(\ell)}(z_1,z_2)&=A^\ell\mQ_0^{(\ell)}(z_1,z_2)\\
&=\frac{\alpha^\ell}{2\pi\I}\oint_{\gamma_{\delta}}\d v\,\frac{\theta^{z_1-z_2} (1+v)^{z_1 - z_2 -1}}{v^\ell}a((1+v)^{-1})^\ell,
\end{aligned}
\end{equation}
where the contour integral formula in the second line, which holds for $\delta>0$ small enough so that $(1+v)^{-1}$ lives in a small circle around $1$ where $a$ is analytic by Assum.~\ref{assum:apsi}, follows from \eqref{eq:contourQA} and \eqref{eq:contourmQ0} and the same argument as in the proof of Lem.~\ref{lem:conv}.
Note that the condition on $\delta$ implies in particular that for each fixed $z_1$, the function $z_2\longmapsto\mQ^{(\ell)}(z_1,z_2)$ is in $\Wd{\ell-1}$.
Moreover, from \eqref{eq:mQ0ellext} and since $A^\ell(x,y)=0$ if $x>y+\ell\kappa$ (which follows directly from \eqref{eq:defA} and Assum.~\ref{assum:q}(i)), we have
\begin{equation}\label{eq:QQext}
\mQ^{(\ell)}(z_1,z_2)=Q^\ell(z_1,z_2)\qquad\text{for}\quad\,z_1-z_2\geq1+\ell\kappa.
\end{equation}

Now let $B_m$ be a random walk with transition matrix $Q$ and define the stopping time
\begin{equation}
\tau= \min\{m=0,\dotsc,n-1 : B_m> y_{m+1}\},\label{eq:deftau}
\end{equation}
i.e., $\tau$ is the hitting time of the strict epigraph of the ``curve'' $(y_{m+1})_{m=0,\dotsc,n-1}$ by the random walk $(B_m)_{m\geq0}$ (we set $\tau=\infty$ if the walk does not go above the curve by time $n-1$).
Define
\begin{equation}\label{eq:barG0n}
\bar G_{0,n}(z_1,z_2)=\ee_{B_0=z_1}\!\left[\mQ^{(n-\tau)}(B_\tau,z_2)\uno{\tau<n}\right].
\end{equation}
The following result provides the crucial connection between the kernel $K$ and the random walk $B_m$:

\begin{prop}\label{prop:G0n-formula}
Assume $y_{j}-y_{j+1}\geq\kappa$ for each $j$.
Then for each $n\in\set{\bn}$ we have
\begin{equation}
G_{0,n}=A\bar G_{0,n}A^{-1}.\label{eq:G0n-formula}
\end{equation}
\end{prop}

To prove this proposition we need a preliminary result, which will allow us to express the solution of \eqref{bhe} in terms of the random walk $B_m$ (or, more precisely, a version of $B_m$ running backwards in time).
Define
\begin{equation}
\textstyle p^n_k(k,z)=\sum_{\eta_k>y_{n-k}}\Qt(z,\eta_k)
\end{equation}
and, for $0\leq\ell<k<n$,
\begin{equation}\label{eq:pnk-expl}
\textstyle p^n_k(\ell,z)=\sum_{\eta_\ell\leq y_{n-\ell}}\Qt(z,\eta_\ell)\,\,\dotsm\sum_{\eta_{k-1}\leq y_{n-k+1}}\Qt(\eta_{k-2},\eta_{k-1})\sum_{\eta_k>y_{n-k}}\Qt(\eta_{k-1},\eta_k).
\end{equation}
Note that if $B^*$ denotes the random walk with transition matrix $\Qt$ and we define the stopping times
\begin{equation}
\tau^{\ell,n}=\min\{m = \ell,\dotsc,n-1 : B^*_m> y_{{n-m}}\},
\end{equation}
then
\begin{equation}
p^n_k(\ell,z)=\pp_{B^*_{\ell-1}=z}(\tau^{\ell,n}=k).\label{eq:pnkrw}
\end{equation}
Note also that $p^n_k(\ell,z)$ satisfies \eqref{bhe1} for $z\leq y_{n-\ell}$.
Our goal in the two results that follow will be construct a certain extension of this $p^n_k(\ell,z)$ to a function $\bar p^n_k(\ell,z)$ defined for all $z\in\zz$ which satisfies \eqref{bhe1} everywhere, and to show that the solution $h^n_k$ of \eqref{bhe} can be expressed explicitly in terms of $\bar p^n_k(\ell,z)$.
Since this last function comes from an extension of \eqref{eq:pnkrw}, this will allow us to establish the connection between the kernel $K$ and the random walk $B_m$ stated in Prop.~\ref{prop:G0n-formula}.

More precisely, in the next result we show that, if $\vec y\in\Omega_n(\kappa)$, then $p^n_k(\ell,z)$ equals $\theta^{-z}$ times a polynomial for $z\leq y_{n-\ell}-\kappa$, so it can be extended to a function $\bar p^n_k(\ell,z)$ which equals $\theta^{-z}$ times a polynomial for all $z$, which for brevity we call the \emph{analytic extension} of $p^n_k(\ell,z)$ (note that this extension is such that $\bar p^n_k(\ell,z)=p^n_k(\ell,z)$ for all $z\leq y_{n-\ell}-\kappa$ and not necessarily for all $z\leq y_{n-\ell}$ as one could have hoped in view of the discussion in the last paragraph).
Furthermore, we will derive an explicit formula for $p^n_k(\ell,z)$, which will allow us to show in Cor.~\ref{cor:barpnk} that, as needed, its extension satisfies \eqref{bhe1} everywhere.

\begin{lem}\label{lem:pnk}
Fix $k\in\{0,\dotsc,n-1\}$, assume $y_{j}-y_{j+1}\geq\kappa$ for $j\in\{n-k,\dotsc,n-1\}$ and write $\bar y_\ell=y_{n-\ell}$.
Then for each $\ell=0,\dotsc,k$ and $z\leq\bar y_{\ell}-\kappa$ we have
\begin{align}\label{eq:pnk-ms}
&\textstyle p^n_k(\ell,z)=\sum_{m=\ell}^{k-2}\frac{\alpha^{k-\ell+1}\theta^{\bar y_k-z+1}}{(1-\theta)(2\pi\I)^{k-m}}\oint_{\gamma_\delta}\d v\oint_{\gamma_{\rrin-(k-m-1)\ep}\times\dotsm\times\gamma_{\rrin-\ep}}\d w_{m+1}\dotsm\d w_{k-1}\\
&\hspace{0.4in}\textstyle\times\frac{a((1+v)^{-1})^{m-\ell+1}a(w_{k-1})(1+v)^{\bar y_m-z}}{v^{m-\ell+1}(1-w_{k-1})^2((1+v)w_{m+1}-1)w_{m+1}^{\bar y_{m+1}-\bar y_m-1}}\prod_{j=m+2}^{k-1}\frac{1}{(w_{j}-w_{j-1})w_j^{\bar y_j-\bar y_{j-1}-1}}\prod_{j=m+1}^{k-2}\frac{a(w_j)}{1-w_j}\hskip-14pt\\
&\hspace{2.7in}\textstyle+\frac{\alpha^{k-\ell+1}\theta^{\bar y_k-z+1}}{(1-\theta)2\pi\I}\oint_{\gamma_{\delta}}\d v\,\frac{a((1+v)^{-1})^{k-\ell}(1+v)^{\bar y_{k-1}-z}}{v^{k-\ell+1}},
\end{align}
where $\ep>0$ is small enough so that $\gamma_{\rrin-k\ep}$ is contained inside the domain of analyticity of $a$ and $\delta\in(0,\rrin^{-1}-1)$ (here sums and products over empty index ranges are taken to be equal to $0$ and $1$ respectively) and, when $k=0$, $\bar y_{-1}=y_{n+1}$ is a dummy parameter (and one has $p^n_0(0,z)=\alpha(1-\theta)^{-1}\theta^{y_n-z+1}$).
Moreover, for such $z$ we have that $\theta^zp^n_k(\ell,z)$ is a polynomial of degree at most $k-\ell$.
\end{lem}

\begin{proof}
Using \eqref{eq:pnk-expl} and in view of \eqref{eq:Qgeom} we have, for $z\leq \bar y_k-\kappa$, $p^n_k(k,z)=\sum_{\eta>\bar y_k}\alpha\tts\theta^{\eta-z}=\alpha (1-\theta)^{-1}\theta^{\bar y_k-z+1}$.
This gives the result for $\ell=k$, since in this case only the last term in the right hand side of \eqref{eq:pnk-ms} survives, and the residue of that integral at $v=0$ is clearly $1$.

Assume now that $\ell<k$.
Throughout the proof we will write $\d w^j_{i_1,\dotsm,i_j}=\d w_{i_1}\dotsm\d w_{i_j}$, $\gamma^j_{r_1,\dotsc,r_j}=\gamma_{r_1}\times\dotsm\times\gamma_{r_j}$, $\eta^j_{i_1,\dotsc,i_j}=(\eta_{i_1},\dotsc,\eta_{i_j})$ and $I^j_{y_1,\dotsc,y_j}=\{\eta^j_{i_1,\dotsc,i_j}\in\zz^j\!:\eta_{i_1}\leq y_1,\dotsc,\eta_{i_{j}}\leq y_{j}\}$.
By assumption we have $\eta_k-\eta_{k-1}>\kappa$ in the last sum in \eqref{eq:pnk-expl}, so the sum yields $\alpha (1-\theta)^{-1}\theta^{\bar y_k-\eta_{k-1}+1}$ and then from \eqref{eq:contourQA} and \eqref{eq:pnk-expl} we get that $p^n_k(\ell,z)$ equals
\begin{align}
&\textstyle\sum_{\eta^{k-\ell}_{\ell,\dotsc,k-1}\in I^{k-\ell}_{\bar y_{\ell},\dotsc,\bar y_{k-1}}}\frac{\alpha^{k-\ell+1}\theta^{\bar y_k-z+1}}{(1-\theta)(2\pi\I)^{k-\ell}}\oint_{\gamma^{k-\ell}_{\rrin,\dotsc,\rrin}}\d w^{k-\ell}_{\ell,\dotsc,k-1}\,\frac{1}{w_\ell^{\eta_\ell-z}w_{\ell+1}^{\eta_{\ell+1}-\eta_\ell}\dotsm w_{k-1}^{\eta_{k-1}-\eta_{k-2}}}\frac{a(w_\ell)\dotsm a(w_{{k-1}})}{(1-w_\ell)\dotsm(1-w_{k-1})}\\
&\qquad=\textstyle\sum_{\eta^{k-\ell-1}_{\ell+1,\dotsc,k-1}\in I^{k-\ell-1}_{\bar y_{\ell+1},\dotsc,\bar y_{k-1}}}\frac{\alpha^{k-\ell+1}\theta^{\bar y_k-z+1}}{(1-\theta)(2\pi\I)^{k-\ell}}\oint_{\gamma^{k-\ell}_{\rrin-\ep,\rrin,\dotsc,\rrin}}\d w^{k-\ell}_{\ell,\dotsc,k-1}\\
&\textstyle\hspace{1.5in}\times\frac{1}{w_\ell^{\bar y_\ell-z}w_{\ell+1}^{\eta_{\ell+1}-\bar y_\ell-1}w_{\ell+2}^{\eta_{\ell+2}-\eta_{\ell+1}}\dotsm w_{k-1}^{\eta_{k-1}-\eta_{k-2}}}\frac{a(w_\ell)\dotsm a(w_{{k-1}})}{(w_{\ell+1}-w_\ell)(1-w_\ell)\dotsm(1-w_{k-1})},
\end{align}
where in computing the geometric sum over $\eta_\ell\leq\bar y_\ell$ we have shrunk the $w_{\ell}$ contour to a circle of radius $\rrin-\ep$. 
Shrinking now the $w_\ell$ contour further to $\gamma_{\rrin-2\ep}$ and the $w_{\ell+1}$ contour to $\gamma_{\rrin-\ep}$ we may compute the sum over $\eta_{\ell+1}\leq\bar y_{\ell+1}$ to get
\begin{multline}
\textstyle\sum_{\eta^{k-\ell-2}_{\ell+2,\dotsc,k-1}\in I^{k-\ell-2}_{\bar y_{\ell+2},\dotsc,\bar y_{k-1}}}\frac{\alpha^{k-\ell+1}\theta^{\bar y_k-z+1}}{(1-\theta)(2\pi\I)^{k-\ell}}\oint_{\gamma^{k-\ell}_{\rrin-2\ep,\rrin-\ep,\rrin,\dotsc,\rrin}}\d w^{k-\ell}_{\ell,\dotsc,k-1}\\
\textstyle\times\frac{1}{w_\ell^{\bar y_\ell-z}w_{\ell+1}^{\bar y_{\ell+1}-\bar y_\ell-1}w_{\ell+2}^{\eta_{\ell+2}-\bar y_{\ell+1}-1}w_{\ell+3}^{\eta_{\ell+3}-\eta_{\ell+2}}\dotsm w_{k-1}^{\eta_{k-1}-\eta_{k-2}}}\frac{a(w_\ell)\dotsm a(w_{{k-1}})}{(w_{\ell+2}-w_{\ell+1})(w_{\ell+1}-w_\ell)(1-w_\ell)\dotsm(1-w_{k-1})},
\end{multline}
and then proceeding inductively and computing the sums up to the variable $\eta_{k-1}$ we arrive at
\[\textstyle\frac{\alpha^{k-\ell+1}\theta^{\bar y_k-z+1}}{(1-\theta)(2\pi\I)^{k-\ell}}\oint_{\gamma^{k-\ell}_{\rrin-(k-\ell)\ep,\dotsc,\rrin-\ep}}\d w^{k-\ell}_{\ell,\dotsc,k-1}\,\frac{a(w_{k-1})}{w_\ell^{\bar y_\ell-z}(1-w_{k-1})^2}\prod_{j=\ell+1}^{k-1}\frac{1}{(w_{j}-w_{j-1})w_j^{\bar y_j-\bar y_{j-1}-1}}\prod_{j=\ell}^{k-2}\frac{a(w_j)}{1-w_j}.\]
Next we want to make the first integration variable lie on a contour larger than all the other ones, so we write the above as $\alpha^{k-\ell+1}(1-\theta)^{-1}\theta^{\bar y_k-z+1}$ times
\begin{align}
&\textstyle\frac1{(2\pi\I)^{k-\ell}}\oint_{\gamma^{k-\ell}_{\rrin,\rrin-(k-\ell-1)\ep,\dotsc,\rrin-\ep}}\!\!\!\!\d w^{k-\ell}_{\ell,\dotsc,k-1}\,\frac{a(w_\ell)a(w_{k-1})}{w_\ell^{\bar y_\ell-z}(1-w_\ell)(1-w_{k-1})^2}\prod_{j=\ell+1}^{k-1}\frac{1}{(w_{j}-w_{j-1})w_j^{\bar y_j-\bar y_{j-1}-1}}\prod_{j=\ell+1}^{k-2}\frac{a(w_j)}{1-w_j}\hskip-16pt\\
&\qquad+\textstyle\frac1{(2\pi\I)^{k-\ell-1}}\oint_{\gamma^{k-\ell-1}_{\rrin-(k-\ell-1)\ep,\dotsc,\rrin-\ep}}\d w^{k-\ell-1}_{\ell+1,\dotsc,k-1}\,\frac{a(w_{\ell+1})^2a(w_{k-1})}{w_{\ell+1}^{\bar y_{\ell+1}-z-1}(1-w_{\ell+1})^2(1-w_{k-1})^2}\\
&\textstyle\hspace{3in}\times\prod_{j=\ell+2}^{k-1}\frac{1}{(w_{j}-w_{j-1})w_j^{\bar y_j-\bar y_{j-1}-1}}\prod_{j=\ell+2}^{k-2}\frac{a(w_j)}{1-w_j}
\end{align}
after collecting the residue at $w_\ell=w_{\ell+1}$.
Enlarging now the $w_{\ell+1}$ contour on the second integral and then repeating the argument inductively, the last expression becomes
\begin{multline}
\textstyle\sum_{m=\ell}^{k-2}\frac1{(2\pi\I)^{k-m}}\oint_{\gamma^{k-m}_{\rrin,\rrin-(k-m-1)\ep,\dotsc,\rrin-\ep}}\d w^{k-m}_{m,\dotsc,k-1} \frac{a(w_m)^{m-\ell+1}a(w_{k-1})}{w_m^{\bar y_m-z-m+\ell}(1-w_m)^{m-\ell+1}(1-w_{k-1})^2}\\
\textstyle\times\prod_{j=m+1}^{k-1}\frac{1}{(w_{j}-w_{j-1})w_j^{\bar y_j-\bar y_{j-1}-1}}\prod_{j=m+1}^{k-2}\frac{a(w_j)}{1-w_j}+\frac1{2\pi\I}\oint_{\gamma_{\rrin-\ep}}\d w\,\frac{a(w)^{k-\ell}}{w^{\bar y_{k-1}-z-k+\ell+1}(1-w)^{k-\ell+1}}.
\end{multline}
Introducing the change of variables $w_m\longmapsto1/(1+v)$ in each summand of the first term and $w\longmapsto1/(1+v)$ in the last integral shows now that $p^n_k(\ell,z)$ equals the right hand side of \eqref{eq:pnk-ms} except that the $v$ contour in each summand is a circle $\tilde\gamma$ of radius $1/\rrin$ centered at $-1$.
To see that $\tilde\gamma$ can be shrunk to $\gamma_\delta$ (which is inside $\tilde\gamma$ by our assumption on $\delta$) we need to analyze the possible singularities of the integrand (other than $v=0$) inside the contour.
Note that the singularity at $v=-1+1/w_{m+1}$ in the first term is actually outside the contour thanks to our choices (this is precisely why we went through the trouble of enlarging the first contour in each integral above).
Next note that for $v$ inside $\tilde\gamma$, $(1+v)^{-1}$ lies outside $\gamma_\rrin$, where $a((1+v)^{-1})$ is analytic by Assum.~\ref{assum:apsi}.
So we only need to worry about the singularity at $v=-1$, but assuming now that $z\leq\bar y_\ell-\kappa$ and since $\bar y_\ell\leq\bar y_m-(m-\ell)\kappa$, using \eqref{eq:def-a} we see that the factor $(1+v)^{\bar y_{m}-z}$ is analytic at $v=-1$.
This proves \eqref{eq:pnk-ms}, while the fact that $\theta^zp^n_k(\ell,z)$ is a polynomial of degree at most $k-\ell$ for such $z$ follows directly from that and Cauchy's formula, since the only pole of each of the $v$ integrals is at $v=0$.
\end{proof}

\begin{cor}\label{cor:barpnk}
Let $\bar p^n_k(\ell,z)$ be the analytic extension of $p^n_k(\ell,z)$ from $z\leq y_{n-\ell}-\kappa$ to all $z$.
Then $\bar p^n_k(\ell,\cdot)\in\Wd{n}$, it has degree $k-\ell$, and for $\ell=0,\dotsc,k-1$ and all $z\in\zz$,
\begin{equation}\label{eq:Qtinvbarp}
(\Qt)^{-1}\bar p^n_k(\ell,z)=\bar p^n_k(\ell+1,z).
\end{equation}
Moreover, for $\ell=0,\dotsc,k$ and for $h^n_k$ as in Prop.~\ref{prop:pnkhnk} we have for all $z\in\zz$ that
\begin{equation}
h^n_k(\ell,z)=(1-\theta)\theta^{-1}(A^*)^{-1}\bar p^n_k(\ell,z)\label{eq:fnkRAbp}
\end{equation}
and, in particular, $\Phi^n_k(z)=(1-\theta)\theta^{-1}(\R^*)^{-1}(A^*)^{-1}\bar p^n_k(0,z)$, i.e.,
\begin{align}
&\textstyle\Phi^n_k(z)=\sum_{m=0}^{k-2}\frac{\alpha^{k}\theta^{\bar y_{k}-z}}{(2\pi\I)^{k-m}}\oint_{\gamma_\delta}\d v\oint_{\gamma^{k-m-1}_{\rrin-(k-m-1)\ep,\dotsc,\rrin-\ep}}\d w^{k-m-1}_{m+1,\dotsc,k-1}\prod_{j=m+1}^{k-2}\frac{a(w_j)}{1-w_j}\\
&\hspace{0.7in}\textstyle\times\frac{a((1+v)^{-1})^{m}a(w_{k-1})(1+v)^{\bar y_m-z}}{v^{m+1}\psi((1+v)^{-1})(1-w_{k-1})^2((1+v)w_{m+1}-1)w_{m+1}^{\bar y_{m+1}-\bar y_m-1}}\prod_{j=m+2}^{k-1}\frac{1}{(w_{j}-w_{j-1})w_j^{\bar y_j-\bar y_{j-1}-1}}\\
&\hspace{2.4in}\textstyle+\frac{\alpha^{k}\theta^{\bar y_{k}-z}}{2\pi\I}\oint_{\gamma_{\delta}}\d v\,\frac{a((1+v)^{-1})^{k-1}(1+v)^{\bar y_{k-1}-z}}{v^{k+1}\psi((1+v)^{-1})},
\end{align}
where $\ep$ and $\delta$ are as in Lem.~\ref{lem:pnk}.
\end{cor}

\begin{proof}
The right hand side of \eqref{eq:pnk-ms} is in fact $\theta^{-z}$ times a polynomial in $z$ of degree $k-\ell$ for all $z\in\zz$, so $\bar p^n_k(\ell,z)$ equals that expression for all $z$, it is in $\Wd{n}$, and it has degree $k-\ell$.
Using that formula and the notation from the proof of Lem.~\ref{lem:pnk} and computing as in the proof of Lem.~\ref{lem:conv} we get
\begin{align}
&\textstyle(\Qt)^{-1}\bar p^n_k(\ell,z)=\sum_{m=\ell}^{k-2}\frac{\alpha^{k-\ell}\theta^{\bar y_k-z+1}}{(1-\theta)(2\pi\I)^{k-m}}\oint_{\gamma_\delta}\d v\oint_{\gamma^{k-m-1}_{\rrin-(k-m-1)\ep,\dotsc,\rrin-\ep}}\d w^{k-m-1}_{m+1,\dotsc,k-1}\\
&\hspace{0.45in}\textstyle\times\frac{a((1+v)^{-1})^{m-\ell}a(w_{k-1})(1+v)^{\bar y_m-z}}{v^{m-\ell}(1-w_{k-1})^2((1+v)w_{m+1}-1)w_{m+1}^{\bar y_{m+1}-\bar y_m-1}}\prod_{j=m+2}^{k-1}\frac{1}{(w_{j}-w_{j-1})w_j^{\bar y_j-\bar y_{j-1}-1}}\prod_{j=m+1}^{k-2}\frac{a(w_j)}{1-w_j}\\
&\hspace{2.1in}\textstyle+\frac{\alpha^{k-\ell}\theta^{\bar y_k-z+1}}{(1-\theta)2\pi\I}\oint_{\gamma_{\delta}}\d v\,\frac{a((1+v)^{-1})^{k-\ell-1}(1+v)^{\bar y_{k-1}-z}}{v^{k-\ell}}
\end{align}
for $\ell\leq k-1$.
The integrand in the first integral is analytic at $v=0$ for $m=\ell$, so that term disappears from the sum and we recover the formula for $\bar p^n_k(\ell+1,z)$, which gives \eqref{eq:Qtinvbarp}.

We turn now to \eqref{eq:fnkRAbp}. By Prop.~\ref{prop:pnkhnk}, in order to prove it, it is enough to show that the right hand side, i.e., $g^n_k(\ell,z)\coloneqq(1-\theta)\theta^{-1}(A^*)^{-1}\bar p^n_k(\ell,z)$, satisfies \eqref{bhe}.
\eqref{bhe1} follows from \eqref{eq:Qtinvbarp}.
For \eqref{bhe2} we use $\bar p^n_k(k,z)=\alpha\tts(1-\theta)^{-1}\theta^{y_{n-k}-z+1}$ and compute again as in the proof of Lem.~\ref{lem:conv} 
to get $g^n_k(k,z)=(1-\theta)\theta^{-1}(A^*)^{-1}\bar p^n_k(k,z)=\theta^{y_{n-k}-z}$ as desired.

What remains is to prove \eqref{bhe3}, which translates into showing that
$(A^*)^{-1}\bar p^n_k(\ell,\bar y_\ell)=0$
for $\ell=0,\dotsc,k-1$, where $\bar y_\ell=y_{n-\ell}$. 
Proceeding as above we get
\begin{equation}
\textstyle(A^*)^{-1}\bar p^n_k(\ell,z)=\sum_{m=\ell}^{k-2}I_m(z)+J_{k-1}(z)\label{eq:Apnk}
\end{equation}
with
\begin{align}
&I_m(z)\textstyle=\frac{\alpha^{k-\ell}\theta^{\bar y_{k}-z+1}}{(1-\theta)(2\pi\I)^{k-m}}\oint_{\gamma_\delta}\d v\oint_{\gamma^{k-m-1}_{\rrin-(k-m-1)\ep,\dotsc,\rrin-\ep}}\d w^{k-m-1}_{m+1,\dotsc,k-1}\\
&\quad\textstyle\times\frac{a((1+v)^{-1})^{m-\ell}a(w_{k-1})(1+v)^{\bar y_m-z}}{v^{m-\ell+1}(1-w_{k-1})^2((1+v)w_{m+1}-1)w_{m+1}^{\bar y_{m+1}-\bar y_m-1}}\prod_{j=m+2}^{k-1}\frac{1}{(w_{j}-w_{j-1})w_j^{\bar y_j-\bar y_{j-1}-1}}\prod_{j=m+1}^{k-2}\frac{a(w_j)}{1-w_j},\hskip-14pt\\
&J_{k-1}(z)=\textstyle\frac{\alpha^{k-\ell}\theta^{\bar y_{k}-z+1}}{(1-\theta)2\pi\I}\oint_{\gamma_{\delta}}\d v\,\frac{a((1+v)^{-1})^{k-\ell-1}(1+v)^{\bar y_{k-1}-z}}{v^{k-\ell+1}}.
\end{align}
We focus on the first term on the right hand side of \eqref{eq:Apnk}, $I_{\ell}(z)$.
Computing the residue at the (simple) pole $v=0$ yields (recalling $a(1)=1$)
\begin{multline}
\textstyle I_{\ell}(z)=-\frac{\alpha^{k-\ell}\theta^{\bar y_{k}-z+1}}{(1-\theta)(2\pi\I)^{k-\ell-1}}\oint_{\gamma^{k-\ell-1}_{\rrin-(k-\ell-1)\ep,\dotsc,\rrin-\ep}}\d w^{k-\ell-1}_{\ell+1,\dotsc,k-1}\\
\textstyle\times\frac{a(w_{k-1})}{(1-w_{k-1})^2(1-w_{\ell+1})w_{\ell+1}^{\bar y_{\ell+1}-\bar y_{\ell}-1}}\prod_{j=\ell+2}^{k-1}\frac{1}{(w_{j}-w_{j-1})w_j^{\bar y_j-\bar y_{j-1}-1}}\prod_{j=\ell+1}^{k-2}\frac{a(w_j)}{1-w_j}.
\end{multline}
Now we proceed similarly to the proof of Lem.~\ref{lem:pnk}.
Expanding the $w_{\ell+1}$ contour to $\gamma_\rrin$ and picking up the residue at $w_{\ell+1}=w_{\ell+2}$ the right hand side becomes 
\begin{align}
&-\textstyle\frac{\alpha^{k-\ell}\theta^{\bar y_{k}-z+1}}{(1-\theta)(2\pi\I)^{k-\ell-1}}\oint_{\gamma^{k-\ell-1}_{\rrin,\rrin-(k-\ell-2)\ep,\dotsc,\rrin-\ep}}\d w^{k-\ell-1}_{\ell+1,\dotsc,k-1}\\
&\hspace{0.9in}\textstyle\times\frac{a(w_{k-1})a(w_{\ell+1})}{(1-w_{k-1})^2(1-w_{\ell+1})^2w_{\ell+1}^{\bar y_{\ell+1}-\bar y_{\ell}-1}}\prod_{j=\ell+2}^{k-1}\frac{1}{(w_{j}-w_{j-1})w_j^{\bar y_j-\bar y_{j-1}-1}}\prod_{j=\ell+2}^{k-2}\frac{a(w_j)}{1-w_j}\\
&\qquad-\textstyle\frac{\alpha^{k-\ell}\theta^{\bar y_{k}-z+1}}{(1-\theta)(2\pi\I)^{k-\ell-2}}\oint_{\gamma^{k-\ell-2}_{\rrin-(k-\ell-2)\ep,\dotsc,\rrin-\ep}}\d w^{k-\ell-2}_{\ell+2,\dotsc,k-1}\\
&\hspace{0.9in}\textstyle\times\frac{a(w_{k-1})a(w_{\ell+2})^2}{(1-w_{k-1})^2(1-w_{\ell+2})^3w_{\ell+2}^{\bar y_{\ell+2}-\bar y_{\ell}-2}}\prod_{j=\ell+3}^{k-1}\frac{1}{(w_{j}-w_{j-1})w_j^{\bar y_j-\bar y_{j-1}-1}}\prod_{j=\ell+3}^{k-2}\frac{a(w_j)}{1-w_j}.
\end{align}
After changing variables $w_{\ell+1}\longmapsto1/(1+v)$, the first of the two terms yields exactly $-\theta^{\bar y_\ell-z}I_{\ell+1}(\bar y_\ell)$.
Proceeding inductively to compute the second term by expanding the first contour, changing variables and so on, yields the terms $-\theta^{\bar y_\ell-z}I_{\ell+2}(\bar y_\ell),\dotsc,-\theta^{\bar y_\ell-z}I_{k-2}(\bar y_\ell)$ and finally $-\theta^{\bar y_\ell-z}J_{k-1}(\bar y_\ell)$.
We have shown that $I_{\ell}(z)=-\theta^{\bar y_\ell-z}\sum_{m=\ell+1}^{k-2}I_{m}(\bar y_\ell)-\theta^{\bar y_\ell-z}I_{k-1}(\bar y_\ell)$, and thus in view of \eqref{eq:Apnk} we have
\[\textstyle (A^*)^{-1}\bar p^n_k(\ell,z)=\sum_{m=\ell+1}^{k-2}(I_m(z)-\theta^{\bar y_\ell-z}I_m(\bar y_\ell))+J_{k-1}(z)-\theta^{\bar y_\ell-z}J_{k-1}(\bar y_\ell).\]
This gives $(A^*)^{-1}\bar p^n_k(\ell,\bar y_\ell)=0$ as desired.

The explicit formula for $\Phi^n_k(z)$ follows directly from applying $(\R^*)^{-1}$ to \eqref{eq:Apnk} and computing in the same way.
\end{proof}

\begin{proof}[Proof of Prop.~\ref{prop:G0n-formula}]
Using \eqref{eq:fnkRAbp} in \eqref{eq:defG0n} yields
\begin{equation}\label{eq:G-in-the-proof}
G_{0,n}A(z_1,z_2)=(1-\theta)\theta^{-1}\sum_{k=0}^{n-1}Q^{n-k}(z_1,y_{n-k})\bar p^n_k(0,z_2)
\end{equation}
and then from \eqref{eq:pnkrw} and the definition of $\bar p^n_k(0,z_2)$ we get for $z_2\leq y_n-\kappa$ that
\begin{equation}
G_{0,n}A(z_1,z_2)=(1-\theta)\theta^{-1}\sum_{k=0}^{n-1}Q^{n-k}(z_1,y_{n-k})\pp_{B^*_{-1}=z_2}(\tau^{0,n}=k).\label{eq:G0nA}
\end{equation}

We turn now to computing $\bar G_{0,n}(z_1,z_2)$ for $z_2\leq y_n-\kappa$.
By definition we have $\bar G_{0,n}(z_1,z_2)=\sum_{k=0}^{n-1}\ee_{B_0=z_1}\big[\mQ^{(n-k)}(B_k,z_2)\uno{\tau=k}\big]$, and thanks to the assumption on $z_2$, inside the expectation we have $B_k-z_2\geq y_{k+1}+1 - y_n+\kappa\geq(n-k)\kappa+1$, so by \eqref{eq:QQext} we may replace $\mQ^{(n-k)}$ by $Q^{n-k}$ to get
\[\bar G_{0,n}(z_1,z_2)=\pp_{B_0=z_1}\!\big(\tau<n,\,B_n=z_2\big)=\pp_{B_{-1}^*=z_2}\!\big(\tau^{0,n}<n,\,B_{n-1}^*=z_1\big).\]
Observe that for $\eta>y\geq \eta'+\kappa$, if we ask $B^*$ to jump from $\eta'$ to a location strictly above $y$ then the jump is of size at least $\kappa+1$, so only the geometric part of the definition of $Q$ in \eqref{eq:defQ} is seen and then we have $\pp_{B^*_{k-1}=\eta'}(B^*_k>y)=\theta^{y-\eta+1}(1-\theta)^{-1}\pp_{B^*_{k-1}=\eta'}(B^*_k=\eta)$. 
Therefore for any $\eta>y_{n-k}$, and since $y_{n-k}-y_{n-k+1}\geq\kappa$, we have for $k\in\{0,\dotsc,n-1\}$
\begin{equation}\label{eq:memoryless}
\begin{aligned}
\pp_{B^*_{-1}=z}\big(&\tau^{0,n}=k,\,B^*_k=\eta\big)=\sum_{\eta'\leq y_{n-k+1}}\pp_{B^*_{-1}=z}\big(\tau^{0,n}>k-1,\,B^*_{k-1}=\eta'\big)\Qt(\eta',\eta)\\
&=\sum_{\eta'\leq y_{n-k+1}}\pp_{B^*_{-1}=z}\big(\tau^{0,n}>k-1,\,B^*_{k-1}=\eta'\big)(1-\theta)\theta^{\eta-y_{n-k}-1}\pp_{B^*_{k-1}=\eta'}(B_k^*>y_{n-k})\\
&=(1-\theta)\theta^{\eta-y_{n-k}-1}\pp_{B^*_{-1}=z}\big(\tau^{0,n}=k\big).
\end{aligned}
\end{equation}
We deduce then that, for $z_2\leq y_n-\kappa$,
\begin{align}
\bar G_{0,n}(z_1,z_2)&=\sum_{k=0}^{n-1}\sum_{\eta>y_{n-k}}\pp_{B_{-1}^*=z_2}(\tau^{0,n}=k,\,B_{k}^*=\eta)(\Qt)^{n-k-1}(\eta,z_1)\\
&=\sum_{k=0}^{n-1}\sum_{\eta>y_{n-k}}\pp_{B_{-1}^*=z_2}(\tau^{0,n}=k)(1-\theta)\theta^{\eta-y_{n-k}-1}(\Qt)^{n-k-1}(\eta,z_1)\\
&=\sum_{\ell>0}\theta^{\ell}Q^{-1}G_{0,n}A(z_1-\ell,z_2),
\end{align}
where in the third equality we used \eqref{eq:G0nA}.
Using now \eqref{eq:Q0} and \eqref{eq:AQ0} we conclude that
\[\bar G_{0,n}(z_1,z_2)=Q_0Q^{-1}G_{0,n}A(z_1,z_2)=A^{-1}G_{0,n}A(z_1,z_2).\]
We have proved this identity for $z_2\leq y_{n}-\kappa$, but both sides are in $\Wd{n}$ as functions of $z_2$ (for the left hand side this holds by \eqref{eq:barG0n}, while for the right hand side we use \eqref{eq:G-in-the-proof}, which gives $A^{-1}G_{0,n}A(z_1,z_2)=(1-\theta)\theta^{-1}\sum_{k=0}^{n-1}A^{-1}Q^{n-k}(z_1,y_{n-k})\bar p^n_k(0,z_2)$, together with the fact that $\bar p^n_k(0,\cdot)$ is in $\Wd{n}$ by definition), so the identity extends to all $z_2$ as needed.
\end{proof}

\subsection{Main result and application to particle systems}
\label{sec:main-proof}

Prop.~\ref{prop:G0n-formula} expresses the main part of the kernel in terms of the hitting times of a random walk.
Using this in \eqref{eq:KnG0n} leads to our main result:

\begin{thm}\label{thm:kernel-rw}
Suppose that the functions $a$ and $\psi$ satisfy Assums.~\ref{assum:q} and \ref{assum:apsi}.
For $n\in\set{\bn}$ define
\begin{align}
\SM_{-n}(z_1,z_2) &= (\R Q^{-n}A)^*(z_1,z_2) \\ 
&=\frac{\alpha^{-n+1}}{2\pi\I}\oint_{\gamma_{\rrin}}\d w\,\frac{\theta^{z_2-z_1}}{w^{z_2-z_1+n+1}}\frac{(1-w)^n}{a(w)^{n-1}}\,\psi(w),\label{def:sm}\\
\SN_{n} (z_1,z_2) &= \mQ^{(n)}\R^{-1}A^{-1}(z_1,z_2)\\
&= \frac{\alpha^{n-1}}{2\pi\I}\oint_{\gamma_\delta}\d w\,\frac{(1-w)^{z_2-z_1+n-1}}{\theta^{z_2-z_1}w^n}\frac{a(1-w)^{n-1}}{\psi(1-w)},
\label{def:sn}
\end{align}
where $\delta>0$ is so that $a(1-w)$ and $\psi(1-w)^{-1}$ are analytic inside $\gamma_\delta$. Suppose also that $y_j-y_{j+1}\geq\kappa$ for each $j=1,\dotsc,\bn-1$.
Then the kernel $K$ defined in \eqref{eq:K-schutz} can be expressed as
\begin{equation}\label{eq:Kn-RW}
K(n_i,x_i;n_j,x_j)=- Q^{n_j-n_i}(x_i,x_j)\uno{n_i<n_j}+(\SM_{-n_i})^*\SN^{\epi(\vec y)}_{n_j}(x_i,x_j)
\end{equation}
for any $n_i,n_j\in\set{\bn}$, where
\begin{equation}\label{eq:defSNepi}
	\SN^{\epi(\vec y)}_n(z_1,z_2) = \ee_{B_{0}=z_1}\!\left[\SN_{n-\tau}(B_\tau,z_2)\uno{\tau<n}\right]
\end{equation}
with $B$ the random walk with transition matrix $Q$ defined in \eqref{eq:defQ} and with $\tau$ the hitting time defined in~\eqref{eq:deftau}.
Moreover, in \eqref{def:sm} the contour $\gamma_r$ can be replaced by $\gamma_{r'}$ for any radius $r'\in[r,1)$ (with $r$ as coming from Assum.~\ref{assum:apsi}).
\end{thm}

\begin{proof}
From \eqref{eq:KextKn}, \eqref{eq:KnG0n} and Prop.~\ref{prop:G0n-formula} we get
\[K(n_i,x_i;n_j,x_j)=-Q^{n_j-n_i}(x_i,x_j)\uno{n_i<n_j}+Q^{n_j-n_i}\R Q^{-n_j}A\bar G_{0,n_j}A^{-1}\R^{-1}(x_i,x_j),\]
and then \eqref{eq:Kn-RW} follows directly from the definitions of $\bar G_{0,n_j}$, $\SM_{-n}$, $\SN_{n}$, and $\SN^{\epi(\vec y)}_n$.

It remains to prove the contour integral formulas given in \eqref{def:sm} and \eqref{def:sn}.
The first one follows directly from Lem.~\ref{lem:conv} and the definitions of $\R$, $Q^{-n}$ and $A$.
The integrand is analytic on $A_{r,1}$ by Assum.~\ref{assum:apsi}, so at this stage the radius of the contour can be changed to any $r'\in(r,1)$ without changing the value of the kernel, which gives the last statement of the result.
For the second formula, Lem.~\ref{lem:conv} gives $\R^{-1}A^{-1}(x,y)=\frac{\alpha^{-1}}{2\pi\I}\oint_{\gamma_\rrin}du\,\frac{\theta^{x-y}}{u^{x-y+1}}\frac{1}{a(u)\psi(u)}$ and then computing again as in the proof of that lemma and using \eqref{eq:QExt} we get, for small $\delta>0$,
\[\textstyle\R^{-1}A\bar Q^{(n)}(x,y)=\frac{\alpha^{n-1}}{2\pi\I}\oint_{\gamma_\delta}dv\,\frac{\theta^{x-y}(1+v)^{x-y + 1}}{v^n}\frac{a(1/(1+v))^{n-1}}{\psi(1/(1+v))}.\]
Changing variables $w=v/(1+v)$ leads to the integrand in \eqref{def:sn}, and the resulting contour can be adjusted to lie on any circle of radius $\delta$ small enough so that $a(1-w)$ and $\psi(1-w)^{-1}$ are analytic.
\end{proof}

\begin{rem}\label{rem:assumext}
\leavevmode
\begin{enumerate}[label=(\alph*)]
\item The main advantage of the (very simple) extension of the choice of contour $\gamma_r$ in \eqref{def:sm} given in the last sentence of the theorem is to lift the restriction $r<\theta$ which comes from Assum.~\ref{assum:apsi}.
The restriction can actually be relaxed a bit more: the contour can be chosen to be $\gamma_{r'}$ with any $r'\in(0,1)$ so that $\psi(w)/a(w)^{n-1}$ is analytic on $A_{r',1}$, again since the contour can be deformed without crossing any singularities.
In principle, this could be useful in situations where zeros of $\psi(w)$ cancel singularities of $1/a(w)^{n-1}$ (but the situation does not arise in any of our examples).
\item From its definition in \eqref{eq:K-schutz}, it is easy to see that the parameter $\theta$ enters $K$ simply as a conjugation $\theta^{x_i-x_j}$ (so, in particular, the Fredholm determinants $\det(I-\bP_aK\bP_a)$ do not depend on $\theta$), and it is natural to wonder about how this plays out on the right hand side of \eqref{eq:Kn-RW}.
So consider another choice of the parameter $\theta$, call it $\hat\theta$, and let $\hat K$ be the kernel in \eqref{eq:Kn-RW} defined using $\hat\theta$.
We also put hats on top of other quantities defined using $\hat\theta$ in order to distinguish them from those defined using $\theta$. Using \eqref{eq:defQ}, the Radon-Nikodym derivative of the law of the random walk $\hat B$ up to time $n$ with respect to the law of the random walk $B$ up to time $n$ equals $(\frac{\hat\alpha}{\alpha})^{n}(\frac{\hat\theta}{\theta})^{B_0-B_{n}}$, and hence $\hat\SN^{\epi(\vec y)}_n(z_1,z_2)=(\hat\alpha/\alpha)^{2n-1}({\hat\theta}/{\theta})^{z_1-z_2}\SN^{\epi(\vec y)}_n(z_1,z_2)$.
Similarly $\hat\SM_{-n}(z_1,z_2)=(\frac{\hat\alpha}{\alpha})^{-n+1}(\frac{\hat\theta}{\theta})^{z_2-z_1}\SM_{-n}(z_1,z_2)$ and $\hat Q^n(z_1,z_2)=(\frac{\hat\alpha}{\alpha})^{n}(\frac{\hat\theta}{\theta})^{z_1-z_2}Q^n(z_1,z_2)$.
Hence $\hat K(n_i,x_i;n_j,x_j)=(\frac{\hat\alpha}{\alpha})^{n_j-n_i}(\frac{\hat\theta}{\theta})^{x_i-x_j}K(n_i,x_i;n_j,x_j)$.
From this one sees that the effect of changing $\theta$ on the right hand side of \eqref{eq:Kn-RW} is to introduce a conjugation \emph{and} change the parameter used to define the random walk $B_n$.
\end{enumerate}
\end{rem}

\begin{proof}[Proof of Thm.~\ref{thm:main}]
By Cor.~\ref{cor:biorth_caterpillars}, the left hand side of \eqref{eq:probability-main-kappa=0} is given by the Fredholm determinant on right hand side of \eqref{eq:M_formula_caterpillars} with $\kappa=0$.
We will apply Thm.~\ref{thm:kernel-rw} to the kernel inside that determinant.
To this end we let $a(w)\equiv1$ and $\psi(w) = \varphi(w)^t$.
The properties of the function $\varphi$ listed in Assum.~\ref{a:phi} together with the choice $\theta\in(\rhoin,1)$ in Sec.~\ref{sec:main} imply that $\psi$ satisfies Assum.~\ref{assum:apsi} with $\rin=\rhoin$ and $\rout=\rhoout$.
The assumption also holds (trivially) for $a$, while Assum.~\ref{assum:q} holds (trivially) since $a(w)\equiv1$ corresponds to $q_i=\uno{i\geq1}$.
Hence if we define the functions $Q$ and $\R$ by \eqref{eq:contourQ} and \eqref{eq:defR} for the above choice of functions $a(w)$ and $\psi(w)$ and use them to construct the kernel appearing in \eqref{eq:K_equal_speeds}, the theorem applies and \eqref{eq:probability-main-kappa=0} follows.
\end{proof}

\begin{proof}[Proof of Thm.~\ref{thm:main2}]
The case $\kappa=0$ is already covered by Thm.~\ref{thm:main}, so let $\kappa\geq1$.
Then Cor.~\ref{cor:biorth_caterpillars} implies that, if $t\geq\kappa(n_m-1)$ or if $t\geq0$ and condition \eqref{eq:backward-in-time-one} holds, the left hand side of \eqref{eq:probability-main} is given by the right hand side of \eqref{eq:M_formula_caterpillars} with the given choice of $\kappa$.
In order to apply Thm.~\ref{thm:kernel-rw} in this case we let $a(w)=\varphi(w)^\kappa$ and $\psi(w) = \varphi(w)^t$.
As in the previous proof, the properties of the function $\varphi$ (now listed in Assum.~\ref{a:kappa}(b.i) as well as Assum.~\ref{a:phi}) together with the choice $\theta\in(\rhoin,1)$ in Sec.~\ref{sec:main} imply that $\psi$ and $a$ satisfy Assum.~\ref{assum:apsi} with $\rin=\rhoin$, $\rout=\rhoout$.
Assum.~\ref{a:kappa}(b.ii), on the other hand, implies that $a(w)=\sum_{i\leq\kappa}b_i^{*\kappa}w^i$ where $b_i^{*\kappa}$ stands for the $\kappa$-fold convolution of the $b_i$'s from \eqref{eq:varphiappl}, so defining 
\begin{equation}\label{eq:qbkappaconv}
\textstyle q_i=1-\sum_{j=i}^\kappa b_j^{*\kappa}
\end{equation}
we have $a(w)=\sum_{i\leq\kappa}(q_{i+1}-q_i)w^i$ as in \eqref{eq:def-a}.
We claim that these $q_i$'s satisfy Assum.~\ref{assum:q}.
Condition (i) is straightforward in view of the definition.
For (ii), in the case $\kappa=1$ one sums by parts to write $\sum\theta^iq_i=\theta(1-\theta)^{-1}\sum(q_{i+1}-q_i)\theta^i$, which is finite because the sum equals $a(\theta)=\varphi(\theta)$ and $\varphi$ is analytic on an annulus containing $\theta$; the case $\kappa>1$ follows similarly.
Then Thm.~\ref{thm:kernel-rw} applies in this situation, and we get \eqref{eq:probability-main} as in the previous proof.
\end{proof}

The formulas given in \eqref{eq:SMSNQ} follow directly from \eqref{def:sm} and \eqref{def:sn}, using the setting of the last proof (so that $\psi(w)a(w)=\varphi(w)^{t+\kappa}$) and a change of variables as above.

\appendix 

\section{Convolution of determinantal functions}
\label{app:convolutions}

We are going to prove results, which allow one to compute convolutions of determinantal functions of the type \eqref{eq:G}.

Fix $N \in \nn$ and $\vec v = (v_i)_{i \in \set{N}}$ such that $v_i > 0$ for each $i$.
For each $i, j \in \set{N}$ we consider a function $L_{i, j}\!: \zz^2 \longrightarrow \rr$ such that there are constants $C>0$ and $r > \max_i v_i$ so that $|L_{i, j}(x,y)| \leq C r^{x - y}$.
Then, using the kernels defined in \eqref{eq:kernelsQ} and \eqref{eq:operatorsE}, for $\vec x, \vec y \in \Omega_N$ we define a determinantal function
\begin{equation}\label{eq:FRkernel}
\H_{\!L} (\vec y, \vec x) = \det\bigl[\bigl(\E_{i} \Q^{[i]} L_{i, j} \Q^{[-j]} \E_{-j}\bigr)(y_{i}, x_{j})\bigr]_{i, j \in \set{N}}.
\end{equation}
The sums involved in the compositions of kernels inside the determinant are all absolutely convergent by the same argument as the one provided below \eqref{eq:F_def}.
The following result, which is a generalization of \cite[Lem.~3.2]{MR2594587} shows that, in a particular case, convolutions of such functions preserve their structure.

\begin{prop}\label{prop:Cauchy-Binet_general}
Consider two families of kernels $R_{i}$ and $S_{j}$ on $\zz^2$, for $i,j\in\set{N}$, and write $(R\cdot1)_{i,j}=R_i$, $(1\cdot S)_{i,j}=S_j$ and $(R \cdot S)_{i, j} = R_{i}S_{j}$. If all these kernels satisfy the properties listed above, then for $\vec x, \vec y \in \Omega_N$
\begin{equation}\label{eq:Cauchy-Binet1}
\sum_{\vec z \in \Omega_N} \H_{\!R\cdot1} (\vec y, \vec z) \H_{1\cdot S} (\vec z, \vec x) = \H_{R \cdot S} (\vec y, \vec x).
\end{equation}
\end{prop}

As a particular case (c.f. \eqref{eq:F_def}/\eqref{eq:F_formula}) we get the following:

\begin{cor}\label{cor:MCconv}
Consider two Markov chains on $\Omega_N$ with transition probabilities $G^{(\ell)}_t(\vec y,\vec x)$ of the form \eqref{eq:G-main}, i.e., for $\ell=1,2$, $G^{(\ell)}_t(\vec y,\vec x)=\det[F^{(\ell)}_{i-j}(X_{N+1-i}-y_{N+1-j},t)]_{i,j\in\set{N}}$ where $F^{(\ell)}$ has the form \eqref{eq:F-main} with $\varphi=\varphi_\ell$ for some complex functions $\varphi_1,\varphi_2$.
Assume that these last two functions satisfy Assum.~\ref{a:phi} for a common choice of $\rhoin,\rhoout$.
Then for each $t_1,t_2\geq0$ and each $\vec x,\vec y\in\Omega_N$,
\[\sum_{\vec z\in\Omega_N}G^{(1)}_{t_1}(\vec y,\vec z)G^{(2)}_{t_2}(\vec z,\vec x)=\bar G_{t_1,t_2}(\vec y,\vec x)\]
with $\bar G_{t_1,t_2}(\vec y,\vec x)$ again of the form \eqref{eq:G-main} with the right hand side of \eqref{eq:F-main} now defined using $t=1$ and $\varphi(w)=\varphi_1(w)^{t_1}\varphi_2(w)^{t_2}$.
\end{cor}

Before proving Prop.~\ref{prop:Cauchy-Binet_general} we need a version of the generalized Cauchy-Binet/Andr\'eief identity.

\begin{lem}
For a measure space $(\Lambda, \mathcal{B}, \lambda)$, let $\varphi_i, \psi_i : \Lambda \to \rr$ be measurable functions such that $\varphi_i \psi_j$ is integrable for any $i, j \in \set{N}$. 
Assume also that $\Lambda$ is a totally ordered set, and define the Weyl chamber $\Omega^\Lambda_N = \{\vec x \in \Lambda^N : x_1 > x_2 > \dotsm > x_N\}$.
Then
\begin{equation}\label{eq:CB}
\det\left[\int_\Lambda \varphi_i(x) \psi_j(x) \d \lambda(x) \right]_{i, j \in \set{N}} = \int_{\Omega^\Lambda_N} \det[\varphi_i(x_j)]_{i, j \in \set{N}} \det[\psi_i(x_j)]_{i, j \in \set{N}} \d \lambda^N(\vec x).
\end{equation}
\end{lem}

The identity is usually stated (see e.g. \cite[Prop.~2.10]{johanssonRMandDetPr}) with the integral on the right hand side over $\Lambda^N$ and an additional factor of $1/N!$; \eqref{eq:CB} follows from this by antisymmetry of determinant.

\begin{proof}[Proof of Prop.~\ref{prop:Cauchy-Binet_general}]
Applying the Cauchy-Binet identity \eqref{eq:CB} we get
\[\textstyle\H_{\!1\cdot S} (\vec z, \vec x) = \sum_{\vec u \in \Omega_N} \det\bigl[\E_i \Q^{[i]}(z_{i}, u_{j})\bigr]_{i,j \in \set{N}} \det\bigl[S_j \Q^{[-j]} \E_{-j} (u_{i}, x_{j})\bigr]_{i,j \in \set{N}}.\]
The key will be to prove that
\begin{equation}\label{eq:CB2}
\textstyle\sum_{\vec z \in \Omega_N} \H_{\!R\cdot1} (\vec y, \vec z) \det\bigl[\E_i \Q^{[i]}(z_{i}, u_{j})\bigr]_{i,j \in \set{N}} = \det\bigl[\E_i \Q^{[i]} R_{i}(y_{i}, u_{j})\bigr]_{i,j \in \set{N}}.
\end{equation}
In fact, using these two identities we may write
\[\textstyle\sum_{\vec z \in \Omega_N} \H_{\!R\cdot1} (\vec y, \vec z) \H_{\!1\cdot S} (\vec z, \vec x) = \sum_{\vec z \in \Omega_N} \det\bigl[\E_{i} \Q^{[i]} R_{i}(y_{i}, z_{j})\bigr]_{i,j \in \set{N}} \det\bigl[ S_j \Q^{[-j]} \E_{-j}(z_{i}, x_{j})\bigr]_{i,j \in \set{N}}\]
which, after another application of \eqref{eq:CB}, equals $\det\bigl[\E_{i} \Q^{[i]}R_iS_j \Q^{[-j]} \E_{-j}(y_{i}, x_{j})\bigr]_{i,j \in \set{N}}$ as desired.

So we need to prove \eqref{eq:CB2}. 
To have a shorter notation we write $A_{i}(x,y) = \E_i \Q^{[i]} R_{i} (x,y)$.
Then using the definitions in \eqref{eq:operatorsE}, the left hand side of \eqref{eq:CB2} can be written as
\begin{equation}\label{eq:CB22}
\textstyle\sum_{\vec z \in \Omega_N} \det\bigl[A_{i} \Q^{[-j]}(y_{i}, z_{j})\bigr]_{i,j \in \set{N}} \det\bigl[\Q^{[i]}(z_{i}, u_{j})\bigr]_{i,j \in \set{N}}.
\end{equation}
We will use the summation by parts formula, which follows from \eqref{eq:Q_inverse_def}, 
\begin{equation}\label{eq:sum_by_parts}
\textstyle\sum_{u = a}^b (f \Q_i^{-1})(u) g(u) = \sum_{u = a}^b f(u) (\Q_i^{-1}g)(u) + v_i f(a) g(a-1) - v_i f(b+1) g(b).
\end{equation}
Using multilinearity, the first determinant in \eqref{eq:CB22} can be written as
\begin{equation}
 \textstyle\sum_{x \in \zz} \det\bigl[A_{\boldsymbol{\cdot}}\!\Q^{[-1]}(y_{\boldsymbol{\cdot}}, z_{1}), \dotsc, A_{\boldsymbol{\cdot}}\!\Q^{[-N+1]}(y_{\boldsymbol{\cdot}}, z_{N-1}), A_{\boldsymbol{\cdot}}\!\Q^{[-N+1]}(y_{\boldsymbol{\cdot}}, x)\bigr] \Q^{-1}_{N}(x, z_N),
\end{equation}
where we wrote $A_{\boldsymbol{\cdot}}\!\Q^{[-j]}(y_{\boldsymbol{\cdot}}, z_{j})$ for the $j^{\text{th}}$ column of the matrix $\big(A_{i}\!\Q^{[-\ell]}(y_{i}, z_{\ell})\big)_{i,\ell\in\set{N}}$. Recalling that $z_N$ is summed from $-\infty$ to $z_{N-1}$ and applying \eqref{eq:sum_by_parts}, \eqref{eq:CB22} becomes
\begin{multline}\label{eq:CB3}
\textstyle\sum_{\vec z \in {\Omega}_N} \det\bigl[A_{\boldsymbol{\cdot}}\!\Q^{[-1]}(y_{\boldsymbol{\cdot}}, z_{1}), \dotsc, A_{\boldsymbol{\cdot}}\!\Q^{[-N+1]}(y_{\boldsymbol{\cdot}}, z_{N-1}), A_{\boldsymbol{\cdot}}\!\Q^{[-N+1]}(y_{\boldsymbol{\cdot}}, z_N)\bigr]\\
\textstyle\times \det\bigl[\Q^{[1]}(z_{1}, u_{\boldsymbol{\cdot}}), \dotsc, \Q^{[N-1]}(z_{N-1}, u_{\boldsymbol{\cdot}}), \Q^{[N-1]}(z_{N}, u_{\boldsymbol{\cdot}})\bigr].
\end{multline}
To see this we need to check that the last two terms in \eqref{eq:sum_by_parts} do not contribute:
for the first of the two terms this holds because for every $z_N$ sufficiently small $\det\bigl[\Q^{[i]}(z_{i}, u_{j})\bigr]_{i,j \in \set{N}} = 0$ (this follows readily from the definition \eqref{eq:kernelsQ_formulas} and the residue theorem), while for the second one it holds because in the case $z_N = z_{N-1}$, the matrix in the first determinant in \eqref{eq:CB3} has two equal columns and hence the determinant vanishes. 
Applying the same operations for the variables $z_{N-1}, \dotsc, z_2, z_1$, then for $z_{N}, \dotsc, z_3, z_2$ and so on, \eqref{eq:CB3} turns to 
\begin{equation}
\textstyle\sum_{\vec z \in {\Omega}_N} \det\bigl[A_{i}(y_i, z_{j})\bigr]_{i,j \in \set{N}} \det\bigl[\Q^{[0]}(z_{i}, u_{j})\bigr]_{i,j \in \set{N}} =  \det\bigl[A_{i}(y_i, u_{j})\bigr]_{i,j \in \set{N}},
\end{equation}
which is exactly \eqref{eq:CB2}.
\end{proof}

The following two results extend Prop.~\ref{prop:Cauchy-Binet_general} to a setting where the matrices in the determinants have different sizes; we need this in order to handle the setting of Thm.~\ref{thm:main2}.
For $2 \leq k \leq N$ using \eqref{eq:kernelsQ} we define $\Q^{[2, k]} = \Q_{2} \dotsm \Q_{k}$ and $\Q^{[-k,-2]} = \Q_{k}^{-1} \dotsm \Q_{2}^{-1}$.
Then for functions $R_{i,j} $ as in the beginning of this section we define, for $\vec x, \vec y \in \Omega_{N-1}$,
\begin{equation}
\bar{\H}_{\!R} (\vec y, \vec x) = \det\bigl[\bigl(\E_{i+1} \Q^{[2, i+1]} R_{i, j} \Q^{[-j-1, -2]} \E_{-j - 1}\bigr)(y_{i}, x_{j})\bigr]_{i, j \in \set{N-1}}.
\end{equation}
For a vector $\vec z$ and a scalar $\tilde y$ we write $\tilde y \sqcup \vec z$ for the vector obtained from $\vec z$ by adding $\tilde y$ as the first entry.

\begin{prop}\label{prop:Cauchy-Binet2}
Consider kernels $(R_i)_{i\in\set{N-1}}$ and $(S_i)_{i\in\set{N}}$ with properties as in Prop.~\ref{prop:Cauchy-Binet_general}, such that $R_i$ and $\Q_1$ commute.
Then for $\vec x \in \Omega_N$, $\vec y \in \Omega_{N-1}$ and $\tilde y \in \zz$ one has
\begin{equation}\label{eq:Cauchy-Binet2}
\sum_{\vec z \in \Omega_{N-1}} \bar{\H}_{\!R\cdot1} (\vec y, \vec z) \H_{\!1\cdot S} (\tilde y \sqcup \vec z, \vec x) = \H_{\!U} (\tilde y \sqcup \vec y, \vec x),
\end{equation}
where $U_{1, j} = S_{j}$ and $U_{i, j} = R_{i-1}S_{j}$ for $2 \leq i \leq N$ and $j \in \set{N}$.
\end{prop}

\begin{proof}
Repeating the argument in the proof of \eqref{eq:CB2}, we can write the left hand side of \eqref{eq:Cauchy-Binet2} as
\begin{equation}\label{eq:CB5}
\textstyle\sum_{\vec z \in \Omega_{N-1}} \det\bigl[\E_{i+1} \Q^{[2, i+1]} R_i(y_{i}, z_{j})\bigr]_{i,j \in \set{N-1}} \det\bigl[\Q_1 S_j \Q^{[-j]} \E_{-j}(\tilde z_{i}, x_{j})\bigr]_{i,j \in \set{N}},
\end{equation}
where $\tilde z_1 = \tilde y$ and $\tilde z_{i} = z_{i-1}$ for $i = 2, \dotsc,N$.
The second determinant on the right hand side can be expanded as $\sum_{k = 1}^N (-1)^{1 + k}\ts\Q_1 S_k \Q^{[-k]} \E_{-k}(\tilde y, x_{k}) \det\bigl[\Q_1 S_j \Q^{[-j]} \E_{-j}(z_{i}, x_{j})\bigr]_{i \in \set{N-1},\, j \in \set{N} \setminus\{k\}}$, and plugging this into \eqref{eq:CB5} and then applying the Cauchy-Binet identity \eqref{eq:CB} we get
\[\textstyle\sum_{k = 1}^N (-1)^{1 + k}\ts \Q_1 S_k \Q^{[-k]} \E_{-k}(\tilde y, x_{k}) \det\bigl[\E_{i+1} \Q^{[2, i+1]} R_{i} \Q_1 S_j \Q^{[-j]} \E_{-j}(y_{i}, x_{j})\bigr]_{{ i \in \set{N-1},j \in \set{N} \setminus\{k\} }}.\]
Since $\Q_1$ commutes with $R_i$ and the other $\Q_k$'s commute, this is just the cofactor expansion of the right hand side of \eqref{eq:Cauchy-Binet2} along its first row.
\end{proof}

The following results can be proved similarly.

\begin{prop}\label{prop:Cauchy-Binet3}
Given kernels $(R_i)_{i\in\set{N}}$ and $(S_i)_{i\in\set{N-1}}$ with properties as in Prop.~\ref{prop:Cauchy-Binet_general}, such that $S_i$ and $\Q^{-1}_N$ commute, for $\vec x \in \Omega_{N-1}$, $\vec y \in \Omega_{N}$ and $\tilde y \in \zz$ one has
\begin{equation}\label{eq:Cauchy-Binet3}
\sum_{\vec z \in \Omega_{N-1}} \H_{\!R\cdot1} (\vec y, \vec z \sqcup \tilde y) \tilde \H_{\!1\cdot S} (\vec z, \vec x) = \H_{\!V} (\vec y, \vec x \sqcup \tilde y),
\end{equation}
where $V_{i, j} = R_{i}S_{j}$ and $V_{i, N} = R_{i}$ for $i \in \set{N}$ and $j\in \set{N-1}$.
\end{prop}

\begin{prop}\label{prop:Cauchy-Binet5}
Let $R$ and $S$ be as in Prop.~\ref{prop:Cauchy-Binet3}. Then for $\vec x \in \Omega_{N-1}$, $\vec y \in \Omega_{N}$ and $\tilde y \in \zz$ one has
\begin{equation}
\sum_{\vec z \in \Omega_{N-1}} \H_{\!R\cdot1} (\vec y, \tilde y \sqcup \vec z) \bar \H_{\!1\cdot S} (\vec z, \vec x) = \H_{\!\bar V} (\vec y, \tilde y \sqcup \vec x),
\end{equation}
where $\bar V_{i, 1} = R_{i}$ and $\bar V_{i, j} = R_{i}S_{j-1}$ for $i \in \set{N}$ and $2 \leq j \leq N$.
\end{prop}

\section{Proof of the biorthogonal characterization of the kernel}
\label{app:biorth}

In this section we prove Thm.~\ref{thm:biorth_general}. Before that we provide a sketch of the proof and comment on the relation to previous work. The key step of the proof is to express the function \eqref{eq:G_TS} as a projection of a signed determinantal point process, which we do in Prop.~\ref{prop:G+_expansion}. The correlation kernel of this process can be obtained from the Eynard-Mehta theorem \cite{eynardMehta}, which due to a special form of the domain (a triangular array) can be written in our case in a biorthogonal form; this is the content of Thm.~\ref{thm:biorth_very_general}. Then the formula \eqref{eq:M_formula} follows from a standard result sometimes referred to as the ``gap probability'' of a determinantal point process. A formula like \eqref{eq:M_formula} was first derived in \cite{sasamoto, borFerPrahSasam} for continuous time TASEP with equal starting and ending times. Our proof will follow the generalization of this result to space-like paths derived in \cite{bp-push}. However, as we described after Thm.~\ref{thm:biorth_general}, there are some differences with the latter result.

Throughout the section we fix a space-like path $\cS = \{(n_1, t_1), \dotsc, (n_m, t_m)\} \in \SLP_N$.
Then we have $n_1 \leq n_2 \leq \dotsm \leq n_m$ and $t_1 \geq t_2 \geq \dotsm \geq t_m$. 
It will be convenient to change the order of elements by introducing $\un_i = n_{m - i + 1}$ and $\ut_i = t_{m - i + 1}$, so that $\un_1 \geq \un_2 \geq \dotsm \geq \un_m$ and $\ut_1 \leq \ut_2 \leq \dotsm \leq \ut_m$. 
We will also write $\ut_0=0$.
Respectively, for a vector $\vec x = (x_1, \dotsc, x_m) \in \zz^m$, let $\vec{\ul{x}} $ denote the reversed vector $(x_m, \dotsc, x_1)$.
Then $\G_{\vT, \cS}$, defined in \eqref{eq:G_TS}, can be rewritten as
\begin{equation}\label{eq:G+_new}
\G_{\vT, \cS} (\vec y, \vec x) = \sum_{\vec x(0) \in \Omega_{N}} \sum_{\substack{\vec x(\ut_i) \in \Omega_{\un_i}: \\ x_{\un_i}(\ut_i) = \ul{x}_i, i \in \set{m}}} \G^{-}_{\vT} (\vec y, \vec x(0)) \prod_{i=1}^{m} \G_{\ut_{i-1}, \ut_{i}}(\vec x_{\leq \un_i}(\ut_{i-1}), \vec x(\ut_i)).
\end{equation}
The key fact is that the function \eqref{eq:G+_new} can be written as a marginal of a signed determinantal measure on a larger space.
To this end we define a triangular array of integer variables $\D_{n} = \bigl\{\x^\ell_{k} \in \zz : \ell \in \set{n},\; k \in \set{\ell}\bigr\}$, whose generic element we denote by $\X$. We will also use ``virtual'' variables $\x^{\ell - 1}_{\ell}$ which can be thought of as having fixed values $\infty$. We define the \emph{Gelfand-Tsetlin cone} of size $n \in \nn$ as
\begin{equation}
\GT_{n} = \bigl\{\x^\ell_k \in \zz: \ell \in \set{n}, k \in \set{\ell},\; \x^{\ell-1}_k < \x^\ell_k \leq \x^{\ell+1}_{k+1} \bigr\} \subset \D_{n}.
\end{equation}
As in Sec.~\ref{sec:measures} we parametrize variables by time points, $\x^\ell_{k}(t)$ (see also footnote \ref{ft:timepoints}).
Then the respective arrays of time-dependent variables are $\D_{n}(t)$ and $\GT_{n}(t)$, with a generic element $\X(t)$. 

\subsection{Determinantal measure on triangular arrays}

We begin by stating some results about the function $F_{k, \ell}$ defined in \eqref{eq:F_formula}. 
It will actually be more convenient to work with the function 
\begin{equation}\label{eq:F_tilde}
\tilde{F}_{k, \ell}(x_1, x_2; t) = F_{k, \ell}(x_1, x_2; t) v_{k}^{x_1} / v_{\ell}^{x_2}.
\end{equation}
Note that from \eqref{eq:F_formula} we get $\tilde{F}_{k, \ell}(x_1, x_2; t) = \tilde{F}_{k, \ell}(x_2 - x_1; t)\coloneqq\tilde{F}_{k, \ell}(0, x_2-x_1; t)$. Define also
\begin{equation}\label{eq:tilde_phi_simple}
\tilde{\phi}_\ell(x_1, x_2) = v_\ell^{x_2 - x_1} \uno{x_1 \leq x_2}, \qquad \phi_\ell(x_1, x_2) = v_\ell^{x_2 - x_1} \uno{x_1 > x_2}.
\end{equation}
Then we have the following recurrence relations for $\tilde{F}_{k, \ell}$, which follow directly from \eqref{eq:F_tilde} and \eqref{eq:F_formula}:
\begin{equation}\label{eq:F_prop}
\tilde{F}_{k, \ell-1}(x; t) = \tilde{\phi}_{\ell} * \tilde{F}_{k, \ell}(x; t), \qquad \tilde{F}_{k + 1, N}(x; t) = \phi_{k+1} * \tilde{F}_{k, N}(x; t),
\end{equation}
with $\phi*F(x;t)=\sum_{y \in \zz}\phi(x,y)F(y;t)$. The three results that follow will be useful later on.

\begin{lem}\label{lem:phi_appears}
For $\X \in \GT_{n}$ one has
\begin{equation}\label{eq:phi_appears}
\prod_{\ell = 2}^{n} \prod_{k = 1}^{\ell - 1} \tilde{\phi}_\ell(\x_k^{\ell - 1}, \x_{k+1}^{\ell}) = \prod_{j = 1}^{n} v_{j}^{-\x^j_1} \det \bigl[\phi_j(\x^{j-1}_k, \x^{j}_\ell)\bigr]_{k,\ell \in \set{j}},
\end{equation}
where the functions $\tilde \phi_\ell$ and $\phi_\ell$ are defined in \eqref{eq:tilde_phi_simple}, and where $\x^{\ell-1}_\ell$ are ``virtual'' variables, for which we postulate $\phi_\ell(\x^{\ell-1}_\ell,y) = v_\ell^{y}$. Moreover, if $\vec x \in \Omega_{n}$ and $\X \in \D_{n}$ is such that $\x_{1}^{\ell} = x_\ell$ for $\ell \in \set{n}$, then the right hand side of \eqref{eq:phi_appears} is non-zero only if $\X \in \GT_{n}$.
\end{lem}

\begin{proof}
The case $n = 2$ is easy to check, and both statements can be proved by induction over $n \geq 2$.
\end{proof}

\begin{lem}\label{lem:GT_appears}
For $\vec x, \vec y \in \Omega_{n}$ and for arbitrary time points $t_1, \dotsc, t_n \in \T$ we have
\begin{align}
\det \bigl[\tilde F_{k, \ell}(y_{k}, x_{\ell}; t_k)\bigr]_{k,\ell \in \set{n}} = (-1)^{\lfloor n / 2\rfloor}\!\!\! \sum_{\substack{\X \in \GT_{n}: \\ \x_{1}^{\ell} = x_\ell, \ell \in \set{n}}} &\left(\prod_{j = 1}^{n} v_{j}^{-\x^j_1} \det \bigl[\phi_j(\x^{j-1}_k, \x^{j}_\ell)\bigr]_{k,\ell \in \set{j}} \right) \\
& \hspace{1cm} \times \det \bigl[\tilde F_{k, n}(y_{k}, \x^{n}_\ell; t_k)\bigr]_{k,\ell \in \set{n}}.\label{eq:GT_appears}
\end{align}
\end{lem}

\begin{proof}
Changing the index $\ell \longmapsto n - \ell + 1$ and applying the first identity in \eqref{eq:F_prop} multiple times, we get that the left hand side of \eqref{eq:GT_appears} equals 
\begin{equation}\label{eq:GT_appears1}
(-1)^{\lfloor n / 2\rfloor} \det \bigl[\tilde{\phi}_{n - \ell + 2} * \tilde{\phi}_{n - \ell + 3} * \dotsm * \tilde{\phi}_{n} * \tilde F_{k, n}(x_{n - \ell + 1} - y_{k}; t_k)\bigr]_{k,\ell \in \set{n}}.
\end{equation}
We write the convolution inside the determinant explicitly by introducing new variables $\x^{n - \ell + j}_j$ for $2 \leq j \leq \ell$ such that $\x^{n - \ell + 1}_1 = x_{n - \ell + 1}$ for each $\ell \in \set{n}$:
\[\textstyle\sum_{{\x^{n - \ell + j}_j \in \zz, \, 2 \leq j \leq \ell}} \left( \prod_{j = 1}^{\ell - 1} \tilde{\phi}_{n - \ell + j + 1} \bigl(\x^{n - \ell + j}_j, \x^{n - \ell + j + 1}_{j+1}\bigr) \right) \tilde F_{k, n}(\x^{n}_\ell - y_{k}; t_k).\]
Using the multilinearity of the determinant to take the summation outside of the determinant in \eqref{eq:GT_appears1} we get
\begin{equation}\label{eq:GT_appears2}
\textstyle (-1)^{\lfloor n / 2\rfloor} \sum_{{\x^{\ell}_j \in \zz, \, 2 \leq j \leq \ell \leq n}} \left( \prod_{\ell = 2}^{n} \prod_{j = 1}^{\ell -1} \tilde{\phi}_\ell(\x_j^{\ell - 1}, \x_{j+1}^{\ell}) \right) \det \bigl[\tilde F_{k, n}(\x^{n}_\ell - y_{k}; t_k)\bigr]_{k,\ell \in \set{n}},
\end{equation}
where $\x^{\ell}_1 = x_{\ell}$ for $\ell \in \set{n}$. Applying Lem.~\ref{lem:phi_appears}, expression \eqref{eq:GT_appears2} can be then written as \eqref{eq:GT_appears}. \end{proof}

\begin{lem}\label{lem:GT_to_det}
For $\vec x, \vec y \in \Omega_{n}$ and for arbitrary time points $t_1, \dotsc, t_n \in \T$ we have
\begin{align}\label{eq:GT_to_det}
\det &\bigl[ \tilde F_{n, n}(y_k, x_\ell; t_\ell) \bigr]_{k,\ell \in \set{n}} \\
&\quad = (-1)^{\lfloor n / 2\rfloor}\!\!\! \sum_{\substack{\X \in \GT_{n}: \\ \x_k^n = y_k, k \in \set{n}}} \left(\prod_{j = 1}^{n} v_j^{-\x^{j}_{1}} \det \bigl[\phi_j(\x^{j-1}_k, \x^{j}_\ell)\bigr]_{k,\ell \in \set{j}}\right) \det \bigl[ \tilde F_{k, n}(\x_1^{k}, x_\ell; t_\ell) \bigr]_{k,\ell \in \set{n}}.
\end{align}
\end{lem}

\begin{proof}
The proof is similar to that of Lem.~\ref{lem:GT_appears}.
We change the order of rows $k \longmapsto n - k + 1$ and apply the second identity in \eqref{eq:F_prop} to write the left hand side of \eqref{eq:GT_to_det} as
\begin{equation}\label{eq:GT_to_det1}
(-1)^{\lfloor n / 2\rfloor} \det \bigl[\tilde{\phi}_{n} * \dotsm * \tilde{\phi}_{k + 1} * \tilde F_{k, n}(x_{\ell} - y_{n - k + 1}; t_\ell)\bigr]_{k,\ell \in \set{n}}.
\end{equation}
Denoting $\x^n_{n - k + 1} = y_{n - k + 1}$ and introducing new variables $\x_{n - k - j + 2}^{n - j + 1}$ for $2 \leq j \leq n - k$, the $(k, \ell)^{\text{th}}$ entry of the matrix in \eqref{eq:GT_to_det1} can be written as
\[\textstyle\sum_{{\x_{n - k - j + 2}^{n - j + 1} \in \zz \, 2 \leq j \leq n - k}} \left( \prod_{j = 1}^{n - k} \tilde{\phi}_{n - j + 1} (\x_{n - k - j + 1}^{n - j}, \x_{n - k - j + 2}^{n - j + 1}) \right) \tilde F_{k, n}(x_{\ell} - \x^k_{1}; t_\ell).\]
Then multilinearity of determinant allows to write \eqref{eq:GT_to_det1} as
\[\textstyle(-1)^{\lfloor n / 2\rfloor} \sum_{{\X \in \D_{n}: \, \x^n_k = y_k, k \in \set{n}}} \left(\prod_{\ell = 2}^{n} \prod_{k = 1}^{\ell - 1} \tilde{\phi}_\ell(\x_k^{\ell - 1}, \x_{k+1}^{\ell})\right) \det \bigl[ \tilde F_{k, n}(\x_1^{k}, x_\ell; t_\ell) \bigr]_{k,\ell \in \set{n}}.\]
Applying Lem.~\ref{lem:phi_appears}, we can write this expression as \eqref{eq:GT_to_det}.
\end{proof}

We turn now to the main goal of this section, which is to write $\G_{\vT, \cS}$ as a marginal of a determinantal measure on triangular arrays.
Fix a vector $\vec y \in \Omega_N$; some of the functions below will depend on $\vec y$ but we will not indicate it in our notation.
For $s, t \in \T$ and for $k\leq n$ in $\nn$ we define the functions 
\begin{align}
\textstyle \cT_{t, s}(x_1, x_2) &\textstyle = \frac{1}{2\pi\I}\oint_{\gamma_\rin}\d w\,\frac{\varphi(w)^{t - s}}{w^{x_1 - x_2 + 1}}, \label{eq:T_def}\\
\textstyle \Psi^{n}_{n - k} (x) &\textstyle = \frac{1}{2\pi\I}\oint_{\gamma_\rin}\d w\,\frac{\prod_{i = k+1}^{n} (v_{i} - w)}{w^{x - y_k +n - k + 1}} \varphi(w)^{-T_k}. \label{eq:Psi_def}
\end{align}
Furthermore, for the space-like path $\cS$ fixed above, we define the domain
\begin{align}\label{eq:DomainD}
\D_{\cS} &= \bigl\{\x^{\un_0}_{\ell}(\ut_0) \in \zz : \ell \in \set{\un_0}\bigr\} \\
&\qquad \cup \bigcup_{i \in \set{m}} \Bigl\{\x^n_{\ell}(\ut_i) \in \zz : \un_{i+1} \leq n \leq \un_i, \ell \in \set{n} ~\text{ such that }~ \x^{n+1}_\ell(\ut_i) < \x^n_\ell(\ut_i) \leq \x^{n+1}_{\ell+1}(\ut_i)\Bigr\}, 
\end{align}
where $\ut_0 = 0$, $\un_0 = N$ and $\un_{m+1} = 0$. Then we define a signed measure $\cW$ on $\X \in \D_{\cS}$ through
\begin{align}\label{eq:W_def}
\cW(\X) &= \det \bigl[\Psi^{\un_0}_{\un_0 - k} (\x^{\un_0}_\ell(\ut_{0}))\bigr]_{k, \ell \in \set{\un_0}} \prod_{j = \un_1+1}^{\un_0} \det \bigl[\phi_j (\x^{j-1}_k(\ut_0), \x^{j}_\ell(\ut_0))\bigr]_{k,\ell \in \set{j}}\\
&\qquad \times \prod_{i=1}^m \det \bigl[\cT_{\ut_{i}, \ut_{i-1}} (\x_k^{\un_i}(\ut_i), \x_\ell^{\un_i}(\ut_{i-1}))\bigr]_{k, \ell \in \set{\un_i}} \prod_{j = \un_{i+1} + 1}^{\un_i} \det \bigl[\phi_j (\x^{j-1}_k(\ut_i), \x^{j}_\ell(\ut_i))\bigr]_{k,\ell \in \set{j}},
\end{align}
where in the case $\un_{i+1} = \un_{i}$ the product $\prod_{j = \un_{i+1} + 1}^{\un_i} a_n$ is by definition $1$. 
The following result gives a formula for $\G_{\vT, \cS}$ as a marginal of $\cW$.

\begin{prop}\label{prop:G+_expansion}
For any $\vec y \in \Omega_N$ and $\vec x \in \Omega_m$ the function \eqref{eq:G+_new} can be written as
\begin{equation}\label{eq:G+_expansion}
\G_{\vT, \cS}(\vec y, \vec x) = C  \sum_{{\X \in \D_{\cS}:\, \x_1^{\un_i}(\ut_i) = x_i, i \in \set{m}}}\cW(\X),
\end{equation}
where $C = \bigl( \prod_{j = 1}^{N} v_j^{ - y_j} \bigr) \prod_{k = 1}^{\un_0} \varphi(v_k)^{T_k} \prod_{i = 1}^m \prod_{j = 1}^{\un_i} \varphi(v_j)^{\ut_{i-1} - \ut_{i} }$.
\end{prop}

\begin{proof}
Using formulas \eqref{eq:G-} and \eqref{eq:G} in \eqref{eq:G+_new}, we can write 
\begin{multline}\label{eq:G+_new1}
\G_{\vT, \cS} (\vec y, \vec x) := \tilde C_1 \sum_{\vec{\x}_1(\ut_0) \in \Omega_{N}} \sum_{\vec \x_1(\ut_i) \in \Omega_{\un_i}: \, \x_1^{\un_i}(\ut_i) = \ul{x}_i, i \in \set{m}} \det \bigl[F_{k, \ell}(y_{k}, \x_1^{\ell}(\ut_0); -T_k)\bigr]_{k, \ell \in \set{\un_0}} \\
\times \prod_{i=1}^{m} \det \bigl[F_{k, \ell}(\x^k_1(\ut_{i-1}), \x^{\ell}_1(\ut_i); \ut_{i} - \ut_{i-1})\bigr]_{k, \ell \in \set{\un_i}},
\end{multline}
where $\tilde C_1 = \prod_{k = 1}^{\un_0} \varphi(v_k)^{T_k} \prod_{i = 1}^m \prod_{j = 1}^{\un_i} \varphi(v_j)^{\ut_{i-1} - \ut_{i} }$. 
Now using \eqref{eq:F_tilde} to replace $F_{k, \ell}$ with $\tilde F_{k, \ell}$ and applying \eqref{eq:GT_appears} to the determinant involving $\vec y$ we get
\begin{align}
\det \bigl[F_{k, \ell}(y_{k}, \x^{\ell}_1(\ut_0); -T_k)\bigr]_{k, \ell \in \set{\un_0}} &= C_0 \!\!\! \sum_{{\X \in \GT_{\un_0}(\ut_0), \, \text{fixed } \x_{1}^{\ell}(\ut_0)}}~ \prod_{j = 1}^{\un_0} \det \bigl[\phi_j (\x^{j-1}_k(\ut_0), \x^{j}_\ell(\ut_0))\bigr]_{k,\ell \in \set{j}} \\
& \hspace{3cm} \times \det \bigl[\tilde F_{k, \un_0}(y_k, \x^{\un_0}_\ell(\ut_{0}); -T_k)\bigr]_{k,\ell \in \set{\un_0}},
\end{align}
where $C_0 = (-1)^{\lfloor \un_0 / 2\rfloor} \prod_{j = 1}^{\un_0} v_j^{ - y_j}$. Similarly, the $i^{\text{th}}$ factor in the second line of \eqref{eq:G+_new1} equals
\begin{multline}
 C_i\hspace{-25pt} \sum_{\X \in \GT_{\un_i}(\ut_i),\,\text{fixed }\x_{1}^{\ell}(\ut_i)} \prod_{j = 1}^{\un_i}\det \bigl[\phi_j(\x^{j-1}_k(\ut_i), \x^{j}_\ell(\ut_i))\bigr]_{k,\ell \in \set{j}}
 \det \bigl[\tilde F_{k, \un_i}(\x^{k}_1(\ut_{i-1}), \x^{\un_i}_\ell(\ut_{i}); \ut_{i} - \ut_{i-1})\bigr]_{k,\ell \in \set{\un_i}},\hspace{-10pt}
\end{multline}
where $C_i = (-1)^{\lfloor \un_i / 2\rfloor} \prod_{j = 1}^{\un_i} v_j^{ - \x^j_1(\ut_{i-1})}$. 
Substituting these expansions into \eqref{eq:G+_new1}, we obtain
\begin{align}
\tilde C_2 &\sum_{\X} \left(\det \bigl[\tilde F_{k, \un_0}(y_k, \x^{\un_0}_\ell(\ut_{0}); -T_k)\bigr]_{k,\ell \in \set{\un_0}} \prod_{j = 1}^{\un_0} \det \bigl[\phi_j (\x^{j-1}_k(\ut_0), \x^{j}_\ell(\ut_0))\bigr]_{k,\ell \in \set{j}} \right) \label{eq:G+_expansion_proof}\\
& \times \prod_{i = 1}^{m} \det \bigl[\tilde F_{k, \un_i}(\x^{k}_1(\ut_{i-1}), \x^{\un_i}_\ell(\ut_{i}); \ut_{i} - \ut_{i-1})\bigr]_{k,\ell \in \set{\un_i}} \prod_{j = 1}^{\un_i} \det \bigl[\phi_j(\x^{j-1}_k(\ut_i), \x^{j}_\ell(\ut_i))\bigr]_{k,\ell \in \set{j}},
\end{align}
where $\tilde C_2 = \tilde C_1 \prod_{i = 0}^{m} C_i$ and where the sum runs over $\X \in \bigcup_{i = 0}^m \GT_{\un_i}(\ut_i)$ such that $\x_1^{\un_i}(\ut_i) = \ul{x}_i$ for $i \in \set{m}$.

Our next aim is to reduce the sum in \eqref{eq:G+_expansion_proof} to the domain $\D_\cS$, defined in \eqref{eq:DomainD}. To this end, for each $i = 1, \dotsc, m$ we sum over the variables $\x_\ell^{k}(\ut_{i-1})$ for $k \in \set{\un_i - 1}$ and $\ell \in \set{k}$ applying Lem.~\ref{lem:GT_to_det}. Then the functions $\tilde F_{k, \un_i}(\x^{k}_1(\ut_{i-1}), \x^{\un_i}_\ell(\ut_{i}); \ut_{i} - \ut_{i-1})$ get replaced by $\tilde F_{\un_i, \un_i}(\x^{\un_i}_{k}(\ut_{i-1}), \x^{\un_i}_\ell(\ut_{i}); \ut_{i} - \ut_{i-1})$, the products $\prod_{j = 1}^{\un_i}$ get replaced by the products $\prod_{j = \un_{i+1} + 1}^{\un_i}$, and we obtain 
\begin{align}
\tilde C_3 &\sum_{\X} \left(\det \bigl[\tilde F_{k, \un_0}(y_k, \x^{\un_0}_\ell(\ut_{0}); -T_k)\bigr]_{k,\ell \in \set{\un_0}}  \prod_{j = \un_1+1}^{\un_0} \det \bigl[\phi_j (\x^{j-1}_k(\ut_0), \x^{j}_\ell(\ut_0))\bigr]_{k,\ell \in \set{j}}\right) \label{eq:G+_expansion_proof1}\\
& \times \prod_{i = 1}^{m} \det \bigl[\tilde F_{\un_i, \un_i}(\x^{\un_i}_{k}(\ut_{i-1}), \x^{\un_i}_\ell(\ut_{i}); \ut_{i} - \ut_{i-1})\bigr]_{k,\ell \in \set{\un_i}} \prod_{j = \un_{i+1} + 1}^{\un_i} \det \bigl[\phi_j(\x^{j-1}_k(\ut_i), \x^{j}_\ell(\ut_i))\bigr]_{k,\ell \in \set{j}},
\end{align}
where $\un_{m+1} = 0$, where the sum runs over $\X \in \D_\cS$ such that $\x_1^{\un_i}(\ut_i) = \ul{x}_i$ for $i \in \set{m}$, and where $\tilde C_3 = \tilde C_2 \prod_{i = 1}^m (-1)^{\lfloor \un_i / 2\rfloor} \prod_{j = 1}^{\un_i} v_j^{ \x^j_1(\ut_{i-1})}$. In the case $\un_{i+1} = \un_{i}$ the product $\prod_{j = \un_{i+1} + 1}^{\un_i} a_n$ is, by definition, $1$. Definitions \eqref{eq:T_def}, \eqref{eq:Psi_def} and \eqref{eq:F_tilde} yield 
\begin{align}
\cT_{\ut_{i}, \ut_{i-1}} (\x_\ell^{\un_i}(\ut_i), \x_k^{\un_i}(\ut_{i-1})) &= \tilde F_{\un_i, \un_i}(\x^{\un_i}_{k}(\ut_{i-1}), \x^{\un_i}_\ell(\ut_{i}); \ut_{i} - \ut_{i-1}),\\
\Psi^{\un_0}_{\un_0 - k} (\x^{\un_0}_\ell(\ut_{0})) &= (-1)^{\un_0 - k} \tilde F_{k, \un_0}(y_k, \x^{\un_0}_\ell(\ut_{0}); -T_k).
\end{align}
Then \eqref{eq:G+_expansion_proof1} can be written as \eqref{eq:G+_expansion} with the constant multiplier $\tilde C_3 \prod_{k = 1}^{\un_0} (-1)^{\un_0 - k}$, which is exactly as in the statement of this proposition.
\end{proof}

\subsection{Proof of the biorthogonalization formula}

Our goal is to show how Thm.~\ref{thm:biorth_general} can be deduced from Prop.~\ref{prop:G+_expansion}. In order to swap the products in the two lines of \eqref{eq:W_def}, for every $n \in \nn_0$ we define $c(n) = \#\{0 \leq i \leq m : \un_i = n\}$ (note that $0 \leq c(n) \leq m + 1$). Furthermore, for each $n$ such that $c(n) \neq 0$ we introduce the time variables $t_1^n < \dotsm < t_{c(n)}^n$ such that the space-like path $S$ contains the pairs $(n, t_1^n)$, $\dotsc$, $(n, t_{c(n)}^n)$.
Moreover, we let $t_{0}^n = t_{c(n+1)}^{n+1}$, $t_{0}^{N} = 0$ and $t_0^0 = t_1$. Then, recalling that $\un_0 = N$, \eqref{eq:W_def} can be written as
\begin{align}
\cW(\X) &= \prod_{j = 1}^{N} \left( \det \bigl[\phi_j (\x^{j-1}_k(t_0^{j-1}), \x^{j}_\ell(t_{c(j)}^j))\bigr]_{k,\ell \in \set{j}}  \prod_{i=1}^{c(j)} \det \bigl[\cT_{t^{j}_{i}, t^{j}_{i-1}} (\x_k^{j}(t_i^j), \x_\ell^{j}(t_{i-1}^j))\bigr]_{k, \ell \in \set{j}} \right) \\
&\hspace{7.5cm} \times \det \bigl[\Psi^{N}_{N - k} (\x^{N}_\ell(t_0^N))\bigr]_{k, \ell \in \set{N}}. \label{eq:G+_expansion_new}
\end{align}

In order to proceed we need to introduce several functions, which depend on the values $n$ and $t^n_i$. As a consequence of \eqref{eq:T_def} we have $\cT_{t^{n}_{c(n)}, t^{n}_{0}} = \cT_{t^{n}_{c(n)}, t^{n}_{c(n) - 1}} * \dotsm * \cT_{t^{n}_{1}, t^{n}_{0}}$, which we denote for brevity by $\cT^{n} = \cT_{t^{n}_{c(n)}, t^{n}_{0}}$, and where we write $A * B (x,y) = \sum_{z \in \zz} A(x,z) B(z, y)$.
For two pairs $\fn_i = (n_i, t_{a_i}^{n_i})$ and $\fn_j = (n_j, t_{a_j}^{n_j})$ such that $\fn_i \prec \fn_j$ we define
\begin{equation}
\phi^{(\fn_i, \fn_j)} = \cT_{t^{n_i}_{a_i}, t^{n_i}_{0}} * \phi_{n_i + 1} * \cT^{n_i + 1} * \dotsm * \phi_{n_j} * \cT_{t^{n_j}_{c(n_j)} ,t^{n_j}_{a_j}}.
\end{equation}
Then using definitions \eqref{eq:T_def} and \eqref{eq:tilde_phi_simple} we can write explicitly 
\begin{equation}\label{eq:phi_general}
 \phi^{(\fn_i, \fn_j)}(x_i, x_j) = \frac{1}{2\pi\I}\oint_{\gamma_{\rin}}\d w\,\frac{\varphi(w)^{t^{n_i}_{a_i} - t^{n_j}_{a_j}}}{w^{x_i-x_j - n_j + n_i + 1}} \prod_{k = n_i + 1}^{n_j} (v_k - w)^{-1}.
\end{equation}
Using \eqref{eq:Psi_def}, for $\fn = (n, t^n_a)$ such that $\fn \prec (N, 0)$ and for $1 \leq k \leq N$, we define 
\begin{equation}\label{eq:Psi_relation}
\Psi^{\fn}_{n - k} = \phi^{(\fn, (N, 0))} * \Psi^{N}_{N - k},
\end{equation}
which can be written explicitly as 
\begin{equation}\label{eq:Psi_integral}
 \Psi^{\fn}_{n - k}(x) = \frac{1}{2\pi\I}\oint_{\gamma_{\rin}}\d w\,\frac{\varphi(w)^{t^{n}_{a}}}{w^{x - y_k + n - k + 1}} \frac{\prod_{i = 1}^n (v_i - w)}{\prod_{i = 1}^k (v_i - w)} \varphi(w)^{-T_k}.
\end{equation}
Finally, we define a matrix $M = (M_{k, \ell})_{k, \ell \in \set{N}}$ with entries 
\begin{equation}\label{eq:matrix_M}
M_{k, \ell} = (\phi_k * \cT^{k} * \dotsm * \phi_{N} * \cT^{N} * \Psi^{N}_{N - \ell}) (\x^{k-1}_k).
\end{equation}
We will use the following result, which is \cite[Thm.~4.2]{bp-push}.

\begin{thm}\label{thm:biorth_very_general}
Suppose that the matrix $M$ is non-singular and upper triangular and define $\hat \cW = \det[M^{-1}] \cW$.
Then $\sum_{\X \in \D_{\cS}}\hat\cW(\X)=1$.
Furthermore the measure $\hat \cW$, interpreted as a (possibly signed) point process is determinantal, with correlation kernel given, for any $\fn_i = (n_i, t^{n_i}_{a_i}), \fn_j = (n_j, t^{n_j}_{a_j}) \in S$ and $x_i, x_j \in \zz$, by
\begin{equation}\label{eq:KernelBP}
K (\fn_i, x_i; \fn_j, x_j) = - \phi^{(\fn_i, \fn_j)}(x_i, x_j) \uno{\fn_i \prec \fn_j} + \sum_{k = 1}^{n_j} \Psi^{\fn_i}_{n_i - k}(x_i) \Phi^{\fn_j}_{n_j - k}(x_j),
\end{equation}
where the functions $\Phi^{(n, t^n_a)}_{n-k}$ for all $n \in \set N$ and $k \in \set n$ are given by
\begin{equation}\label{eq:Phi_general}
\Phi^{(n, t^n_a)}_{n-k}(x) = \sum_{\ell = 1}^n [M^{-1}]_{k, \ell} \bigl( \phi_\ell * \phi^{((\ell, t^\ell_{c(\ell)}), (n, t_{a}^{n}))} \bigr)(\x_{\ell}^{\ell-1}, x).
\end{equation}
In particular, these functions are uniquely defined by the following two conditions:
\begin{enumerate}[label=\uptext{(\arabic*)}]
\item for $k, \ell \in \set{n}$ the biorthogonalization relation $\sum_{x \in \zz} \Psi^{(n, t^{n}_{a})}_{n - k}(x) \Phi^{(n, t^{n}_{a})}_{n - \ell}(x) = \uno{k =\ell}$ holds,
\item $\{x \in \zz \longmapsto \Phi^{(n, t^n_a)}_{n-k}(x) : k \in \set{n}\}$ is a basis of the linear span of the functions
\begin{equation}\label{eq:span_of_phis}
\bigl\{ x \in \zz \longmapsto \phi_k * \phi^{((k, t^k_{c(k)}), (n, t_{a}^{n}))}(\x_{k}^{k-1}, x) : k \in \set{n} \bigr\}.
\end{equation}
\end{enumerate}
Moreover, for any $\fn_i \prec \fn_j$ and for respective values of $k$ one has the identity $\phi^{(\fn_i, \fn_j)} * \Phi^{\fn_j}_{k} = \Phi^{\fn_i}_{k}$.
\end{thm}

\begin{proof}[Proof of Thm.~\ref{thm:biorth_general}]
We start with the case of different values $v_N > v_{N-1} > \dotsm > v_1 > 0$, for which we apply Theorem~\ref{thm:biorth_very_general} to our measure $\cW$ given by \eqref{eq:G+_expansion_new}.
Our first task is to show that the matrix $M$ in \eqref{eq:matrix_M} is non-singular and upper-triangular. 
Let us denote for brevity $\fn_k = (k, t^k_{c(k)})$. Then using \eqref{eq:Psi_relation}, the entry \eqref{eq:matrix_M} can be written as $M_{k, \ell} = (\phi_k * \Psi^{\fn_k}_{k - \ell}) (\x^{k-1}_k)$.
Therefore \eqref{eq:Psi_integral} yields 
\begin{equation}\label{eq:M_integral}
 M_{k, \ell} = \sum_{x \in \zz} v_k^x \,\frac{1}{2\pi\I}\oint_{\gamma_\rin}\d w\,\frac{\varphi(w)^{t^k_{c(k)}}}{w^{x - y_\ell + k - \ell + 1}} \frac{\prod_{i = 1}^k (v_i - w)}{\prod_{i = 1}^\ell (v_i - w)} \varphi(w)^{-T_\ell}.
\end{equation}
Since $|w| < v_k$ for $w\in\gamma_\rin$, the sum over $x<0$ can be computed directly, using $\sum_{x < 0} (v_k / w)^x = w / (v_k - w)$.
For $x \geq 0$ we may enlarge the contour to a circle of radius a bit larger than $|v_k|$ because for $k\geq\ell$ the integrand has no singularities at any of the $v_i$'s, while if $k < \ell$ then the singularities occur only at the points $v_{k+1}, \dotsc, v_\ell$, which are strictly larger than $v_k$ by assumption; in this case we use $\sum_{x \geq 0} (v_k / w)^x = - w / (v_k - w)$.
Putting both sums together yields 
\begin{equation}\label{eq:M_explicit}
 M_{k, \ell} = - \frac{1}{2\pi\I}\oint_{\Gamma_{v_k}}\d w\,\frac{\varphi(w)^{t^k_{c(k)}}}{w^{- y_\ell + k - \ell}} \frac{\prod_{i = 1}^{k-1} (v_i - w)}{\prod_{i = 1}^\ell (v_i - w)} \varphi(w)^{-T_\ell},
\end{equation}
where the contour $\Gamma_{v_k}$ encloses only the singularity at $v_k$. 
If $\ell < k$ then $M_{k, \ell} = 0$, so the matrix $M$ is upper-triangular, with diagonal entries given by
 $M_{k, k} = - \frac{1}{2\pi\I}\oint_{\Gamma_{v_k}}\d w\,\frac{\varphi(w)^{t^k_{c(k)}}}{w^{- y_k}} \frac{\varphi(w)^{-T_k}}{(v_k - w)} = \varphi(v_k)^{t^k_{c(k)}-T_k} v_k^{y_k}$.
 Then we can compute the determinant $\det [M] = \prod_{k = 1}^N \varphi(v_k)^{t^k_{c(k)}-T_k} v_k^{y_k} \neq 0$.
One can readily check that $\det [M^{-1}]$ is exactly the constant $C$ from Prop.~\ref{prop:G+_expansion}.

Hence, defining the normalized measure $\hat \cW = \det[M^{-1}] \cW$, expression \eqref{eq:G+_expansion} can be written as
\begin{equation}\label{eq:G+_expansion_hat}
\G_{\vT, \cS}(\vec y, \vec x) = \sum_{{\X \in \D_{\cS}: \, \x_1^{\un_i}(\ut_i) = x_i, i \in \set{m}}} \hat \cW(\X).
\end{equation}
Thm.~\ref{thm:biorth_very_general} implies that the measure $\hat \cW$ is determinantal with  correlation kernel given in \eqref{eq:KernelBP}. Furthermore, using \eqref{eq:phi_general} for $\fn = (n, t) \in \cS$ such that $\fn_k \prec \fn$ we can compute 
\[\textstyle\phi_k * \phi^{(\fn_k, \fn)}(\x_{k}^{k-1}, x) = \sum_{y \in \zz} v_k^y \, \frac{1}{2\pi\I}\oint_{\gamma_{\rin}}\d w\,\frac{\varphi(w)^{t^k_{c(k)} - t}}{w^{y - x - n + k + 1}} \prod_{i = k + 1}^{n} (v_i - w)^{-1}.\] Computing this sum in the same way as we did for \eqref{eq:M_integral}, we obtain 
\begin{equation}\label{eq:phi_explicit}
\textstyle\phi_k * \phi^{(\fn_k, \fn)}(\x_{k}^{k-1}, x) = - \frac{1}{2\pi\I}\oint_{\Gamma_{v_k}}\!\d w\,\frac{\varphi(w)^{t^k_{c(k)} - t}}{w^{ - x - n + k}} \prod_{i = k}^{n} (v_i - w)^{-1} = \frac{\varphi(v_k)^{t^k_{c(k)} - t}}{v_k^{ - x - n + k}} \prod_{i = k+1}^{n} (v_i - v_k)^{-1}.
\end{equation}
Hence, the set \eqref{eq:span_of_phis} is the span of $\{ x \in \zz \longmapsto v_k^x : k \in \set{n}\}$, which is exactly the set $\V{n}(\vec v)$ defined in \eqref{eq:space}, and thus the correlation kernel \eqref{eq:KernelBP} coincides with \eqref{eq:KernelK}.
We deduce the identity \eqref{eq:M_formula} as a standard consequence of \eqref{eq:G+_expansion_hat}, which expresses $\G_{\vT, \cS}$ as a marginal of the distribution of a determinantal process (see e.g. \cite[Prop.~2.9]{johanssonRMandDetPr} for a version of this in the one point case).

The other values of $v_i$ can be treated by analytic continuation. More precisely, consider values $0 < v_1, \dotsc, v_N < v_+$, for some fixed $v_+ > 0$.
We will first show that the left hand side of \eqref{eq:M_formula} is analytic with respect to the $v_i$'s in this domain and then show that the right hand side can be analytically extended to this domain, which will give the claim \eqref{eq:M_formula} for any choice of the parameters $v_i$.

From \eqref{eq:Psi_def} we conclude that the functions $ \Psi^{N}_{N - k} (x)$ are analytic with respect to the $v_i$'s and satisfy $| \Psi^{N}_{N - k} (x) | \leq C a^{|x|}$, for any $a > 0$ and for any $v_i$ in the compact set as above. Hence, \eqref{eq:tilde_phi_simple} and \eqref{eq:T_def} imply that the measure $\cW(\X)$ in \eqref{eq:G+_expansion_new} can be bounded by a power series in the values $v_i$. The same can be shown for the left hand side of \eqref{eq:M_formula}, because it is a sum of $\cW(\X)$ over a suitable domain for $\X$ (see \eqref{eq:mu_fixed} and \eqref{eq:G+_expansion}).

In order to show that the right hand side of \eqref{eq:M_formula} is analytic in the $v_i$'s, we will show that the correlation kernel is so. In view of \eqref{eq:M_explicit} and \eqref{eq:phi_explicit}, $\Phi^{(n, t^n_a)}_{n-k}(x)$ in \eqref{eq:Phi_general} is analytic in $v_+ > v_N > \dotsm > v_1 > 0$. Analyticity of the other functions implies that the kernel \eqref{eq:KernelBP} is analytic. Hence, we can conclude that the right hand side of \eqref{eq:M_formula} can be extended analytically to all $0 < v_1, \dotsc, v_N < v_+$. 

Since $v_+$ was chosen arbitrarily, identity \eqref{eq:M_formula} holds for any strictly positive values $v_i$. Finally, one can readily check that for general values $v_i$ the integral in \eqref{eq:phi_explicit} equals $v_k^x P_k(x)$, where $P_k(x)$ is a polynomial in $x$ of degree $\#\{i > k : v_i = v_k\}$. Hence, the span of the functions \eqref{eq:span_of_phis} equals $\V{n}(\vec v)$, defined in \eqref{eq:space}.
\end{proof}

\section{Proof of Assum.~\ref{a:kappa} for right Bernoulli jumps}
\label{sec:rightBernoulli-assumptions}

\subsection{Proof of Assum.~\ref{a:kappa}\ref{it:one}}\label{sec:kappaone}

Throughout the section we use $\{\vec e_{i}\}_{i \in \set N}$ to denote the vectors from the canonical basis of $\rr^N$. Let $\G^{\rB}_{0,t}$ be the right hand side of \eqref{eq:rightBernoulliBlock}. We begin by deriving some of its algebraic properties. Although we used this function only on the Weyl chamber $\Omega_N$, it is defined on all of $\zz^N$. The following result can be obtained by direct computations, using properties of the function $F^{\rB}_n(x,t)$ defined in \eqref{eq:rBerBlockF}.

\begin{lem}\label{lem:G_seq}
\leavevmode
\begin{enumerate}[label=\uptext{(\roman*)}]
\item\label{it:G_seq1} Let $\vec x, \vec y \in \Omega_N$ and $k \in \set{N}$ be such that $\vec y = (y_1, y_1 - 1, \dotsc, y_1 - k, y_{k+1}, \dotsc, y_N)$, $\vec x = (y_1, y_1 - 1, \dotsc, y_1 - k, x_{k+1}, \dotsc, x_N)$ and $y_{k+1} < y_1 - k - 1$. Then
\begin{equation}\label{eq:G_seq1}
\textstyle q^{-1} \G^{\rB}_{0,1} (\vec y, \vec x) = \sum_{\substack{\vec z \in \Omega_k : \\ z_i \geq x_i, i \in \set{k}}} \G^{\rB}_{0,1} \bigl(\vec y, (z_1, \dotsc, z_k, x_{k+1}, \dotsc, x_N)\bigr).
\end{equation}
\item\label{it:G_seq2} If $y_k = y_{k+1}$ for some $k \in \set{N-1}$, then $\G^{\rB}_{0,t}(\vec y, \vec x) = \G^{\rB}_{0,t}(\vec y - \vec e_{k+1}, \vec x)$.
\end{enumerate}
\end{lem}

Fix $\vT = (T_i)_{i \in \set N}$ with $T_i = -\kappa (i-1)$ and $\vec y \in \Omega_N(\kappa)$, where $\kappa \geq 1$ and $\Omega_N(\kappa)$ is defined in \eqref{eq:Omega-kappa}. Then for $m \in \set{N}$, $\vec x \in \Omega_{N - m + 1}$ and $T_{m} \leq t < T_{m - 1}$ (with the convention $T_0 = \infty$) we define 
\begin{equation}\label{eq:SchuetzSEQ_0}
	\G^{[m, N]}_{\vT,t} (\vec y, \vec x) = \pp \bigl(\xx^{\rB}_{t}(m + i - 1) = x_i,\,i\in\set{N - m + 1} \big| \xx^{\rB}_{T_i}(i) = y_i, i \in \set{N}\bigr),
\end{equation}
which is the transition probability (at time $t$), for the particles $\xx^{\rB}(m), \dotsc, \xx^{\rB}(N)$ when the system starts at locations $y_1, \dotsc, y_N$ at respective times $T_1, \dotsc, T_N$, to the locations $\vec x$.
Let also, for $s \geq 0$ and $\vec a, \vec b \in \Omega_{N - m + 1}$,
\begin{equation}\label{eq:SchuetzSEQ_m}
\G^{[m, N]}_{0,s} (\vec a, \vec b) = \det \bigl[F^{\rB}_{i - j}(b_{N - m + 2 - i} - a_{N - m + 2 - j}, s)\bigr]_{i, j \in \set{N - m + 1}},
\end{equation}
which is the analog of \eqref{eq:rightBernoulliBlock} for the system which only has the particles $\xx^{\rB}(m), \dotsc, \xx^{\rB}(N)$.

\begin{lem}\label{lem:SchuetzTASEP2}
In above setting, and with the notation $y\sqcup\vec z$ from Appdx.~\ref{app:convolutions},
\begin{equation}\label{eq:SchuetzTASEP2}
	\textstyle \G^{[m, N]}_{\vT,t} (\vec y, \vec x) = \sum\limits_{\substack{\vec z(T_i) \in \Omega_{N - i} \\ m \leq i < N}} \left(\prod_{i = m + 1}^{N} \G^{[i, N]}_{T_{i}, T_{i - 1}} (y_{i} \sqcup \vec z(T_i), \vec z(T_{i-1}))\right) \G^{[m, N]}_{T_{m}, t} (y_{m} \sqcup \vec z(T_m), \vec x),
\end{equation}
where $\vec z(T_N)$ is an empty vector.
\end{lem}

Up to multipliers, the Markov property yields the formula \eqref{eq:SchuetzTASEP2} with the additional condition in the sum that all entries of $\vec z (T_i)$ be strictly smaller than $y_i$. 
However, this restriction precludes the application of Prop.~\ref{prop:Cauchy-Binet2} to the convolutions of determinants in \eqref{eq:SchuetzTASEP2}. 
We will show that the properties of $\G^{\rB}_{0,t}$ provided in Lem.~\ref{lem:G_seq} imply that this restriction can be omitted.

\begin{proof}
We will prove \eqref{eq:SchuetzTASEP2} by induction over $m = N, N - 1, \dotsc, 1$.
For the base case, $m = N$, note that on the time interval $T_{N} \leq t < T_{N - 1}$ only the $N^{\text{th}}$ particle moves. Therefore  \eqref{eq:SchuetzSEQ_m} with $N = 1$ and $\vec x = (x_1)$ yields $\G^{[N, N]}_{\vT,t} (\vec y, \vec x) = \G^{[N, N]}_{T_{N}, t} (y_N, x_1)$, which is \eqref{eq:SchuetzTASEP2}.

Assuming now that \eqref{eq:SchuetzTASEP2} holds for some $2 \leq m \leq N$, we will prove it for $m-1$. For $T_{m-1} \leq t < T_{m - 2}$ and $\vec x \in \Omega_{N - m+2}$ the Markov property yields
\begin{align}\label{eq:prob0}
\textstyle \G^{[m - 1, N]}_{\vT,t} (\vec y, \vec x) &\textstyle = \sum_{{\vec u, \vec a \in \Omega_{N - m + 1}\!:\, a_1 < y_{m-1}}} \pp \bigl(\xx^{\rB}_{T_{m-1} - 1} = \vec u \big| \xx^{\rB}_{T_i}(i) = y_i, i \in \set{N}\bigr)\\
 &\textstyle \qquad\qquad \times \pp \bigl(\xx^{\rB}_{T_{m-1}} = y_{m-1} \sqcup \vec a \big| \xx^{\rB}_{T_{m-1} - 1} = \vec u\bigr) \\
 &\textstyle \qquad\qquad \times \pp \bigl(\xx^{\rB}_{t}(m + i - 2) = x_i,\,i\in\set{N - m + 2} \big| \xx^{\rB}_{T_{m-1}} = y_{m-1} \sqcup \vec a\bigr).
\end{align} 
The induction hypothesis yields $\pp \bigl(\xx^{\rB}_{T_{m-1} - 1} = \vec u \big| \xx^{\rB}_{T_i}(i) = y_i, i \in \set{N}\bigr) = \G^{[m, N]}_{\vT,T_{m-1} - 1} (\vec y, \vec u)$, where the latter is given by the right hand side of \eqref{eq:SchuetzTASEP2}. Moreover, \eqref{eq:SchuetzSEQ_m} gives 
\begin{align}
	\textstyle \pp \bigl(\xx^{\rB}_{t}(m + i - 2) = x_i,\,i\in\set{N - m + 2} \big| \xx^{\rB}_{T_{m-1}} = y_{m-1} \sqcup \vec a\bigr) = \G^{[m-1, N]}_{T_{m-1}, t} (y_{m-1} \sqcup \vec a, \vec x).
\end{align}
Now, we will write explicitly the transition probability from $\vec u$ to $\vec a$ in \eqref{eq:prob0}. Our assumption on the initial state $\vec y \in \Omega_N(\kappa)$ guarantees that $\xx^{\rB}_{T_{m-1} - 1}(m) < \xx^{\rB}_{T_{m-1}}(m-1)$. However, it can happen that $\xx^{\rB}_{T_{m-1} - 1}(m) = \xx^{\rB}_{T_{m-1}}(m-1) - 1$ and on the next step the $m^{\text{th}}$ particle can try to jump on top of the $(m-1)^{\text{st}}$, which should be prevented. We will consider these cases more precisely. 

If $u_{1} < y_{m-1} - 1$, then we have $\pp \bigl(\xx^{\rB}_{T_{m-1}} = y_{m-1} \sqcup \vec a \big| \xx^{\rB}_{T_{m-1} - 1} = \vec u\bigr) = \G^{[m, N]}_{0,1} (\vec u, \vec a)$, and this probability is non-zero for $a_1 < y_{m-1}$ and zero otherwise. Then we can write 
\begin{align}
\textstyle \sum_{\substack{\vec a \in \Omega_{N - m + 1} : \\ a_1 < y_{m-1}}} &\textstyle \pp \bigl(\xx^{\rB}_{T_{m-1}} = y_{m-1} \sqcup \vec a \big| \xx^{\rB}_{T_{m-1} - 1} = \vec u\bigr) \G^{[m-1, N]}_{T_{m-1}, t} (y_{m-1} \sqcup \vec a, \vec x) \\
&\textstyle = \sum_{\vec a \in \Omega_{N -m + 1}} \G^{[m, N]}_{0,1} (\vec u, \vec a) \G^{[m-1, N]}_{T_{m-1}, t} (y_{m-1} \sqcup \vec a, \vec x).\label{eq:prob1}
\end{align}
If $u_{1} = y_{m-1} - 1$, let $1 \leq k \leq N - m + 1$ be so that $\vec u = (y_{m-1} - 1, \dotsc, y_{m-1} - k, u_{k+1}, \dotsc, u_{N - m+1})$, where $u_{k+1} < y_{m-1} - k - 1$ in the case $k \leq N - m$. Then the transition probability from $\vec u$ to $y_{m-1} \sqcup \vec a$ is non-zero only if $a_{i} = u_{i}$ for each $1 \leq i \leq k$. In this case we have $\pp \bigl(\xx^{\rB}_{T_{m-1}} = y_{m-1} \sqcup \vec a \big| \xx^{\rB}_{T_{m-1} - 1} = \vec u\bigr) = q^{-1} \G^{[m, N]}_{0,1} (\vec u, \vec a)$ (i.e. in the probability measure $\G^{[m, N]}_{0,1} (\vec u, \cdot)$ on $\Omega_{N - m + 1}$, we change the probability for the $m^{\text{th}}$ particle to stay put from $q$ to $1$).
Applying \eqref{eq:G_seq1} we obtain
\begin{equation}
\textstyle \pp \bigl(\xx^{\rB}_{T_{m-1}} = y_{m-1} \sqcup \vec a \big| \xx^{\rB}_{T_{m-1} - 1} = \vec u\bigr) = \sum_{\!\!\!\substack{\vec z \in \Omega_k : \\ z_i \geq a_i, 1 \leq i \leq k}} \G^{[m, N]}_{0,1} \bigl(\vec u, (z_1, \dotsc, z_k, a_{k+1}, \dotsc, a_N)\bigr).
\end{equation}
This yields, for $u_{1} = y_{m-1} - 1$,
\begin{align}
&\textstyle \sum_{\substack{\vec a \in \Omega_{N - m + 1} : \\ a_1 < y_{m-1}}} \pp \bigl(\xx^{\rB}_{T_{m-1}} = y_{m-1} \sqcup \vec a \big| \xx^{\rB}_{T_{m-1} - 1} = \vec u\bigr) \G^{[m-1, N]}_{T_{m-1}, t} (y_{m-1} \sqcup \vec a, \vec x)\\
&\textstyle = \sum_{\substack{\vec a \in \Omega_{N - m + 1} : \\ a_i = u_i, 1 \leq i \leq k}} \sum_{\!\!\!\!\!\!\substack{\vec z \in \Omega_k : \\ z_i \geq a_i, 1 \leq i \leq k}} \G^{[m, N]}_{0,1} \bigl(\vec u, (z_1, \dotsc, z_k, a_{k+1}, \dotsc, a_{N-m+1})\bigr) \G^{[m-1, N]}_{T_{m-1}, t} (y_{m-1} \sqcup \vec a, \vec x)\\
&\textstyle = \sum_{\substack{\vec z \in \Omega_{N - m + 1} : \\ z_i \geq u_i, 1 \leq i \leq k}} \G^{[m, N]}_{0,1} (\vec u, \vec z) \G^{[m-1, N]}_{T_{m-1}, t} \bigl((y_{m-1}, u_1, \dotsc, u_k, z_{k+1}, \dotsc, z_{N-m+1}), \vec x\bigr).\label{eq:prob2}
\end{align}
The terms in this sum vanish unless $z_i - u_i \in \{0,1\}$ for each $1\in\set{k}$, and one can see that there is a $k^*\in\set{k}$ such that $(z_1, \dotsc, z_k) = (u_1 + 1, \dotsc, u_{k_*}+1, u_{k_*}, \dotsc, u_k)$. Moreover, $u_{i} + 1 = y_{m-1} - i + 1$ for each $1 \leq i \leq k_*$. Then applying Lem.~\ref{lem:G_seq}\ref{it:G_seq2} consecutively to the entries $z_1$, $z_2$, $\dotsc$, $z_{k_*}$, we get $\G^{[m-1, N]}_{T_{m-1}, t} (y_{m-1} \sqcup \vec z, \vec x) = \G^{[m-1, N]}_{T_{m-1}, t} \bigl((y_{m-1}, u_1, \dotsc, u_k, z_{k+1}, \dotsc, z_{N-m+1}), \vec x\bigr)$.
Furthermore, if $z_i < u_i$ for some $1 \leq i \leq k$, then the function $\G^{[m, N]}_{0,1} (\vec u, \vec z)$ vanishes, which means that \eqref{eq:prob2} can be written as
\begin{equation}\label{eq:prob3}
\textstyle \sum_{\vec z \in \Omega_{N - m + 1}} \G^{[m, N]}_{0,1} (\vec u, \vec z) \G^{[m-1, N]}_{T_{m-1}, t} (y_{m-1} \sqcup \vec z, \vec x).
\end{equation}

Combining identities \eqref{eq:prob1} and \eqref{eq:prob3}, formula \eqref{eq:prob0} can be written as
\begin{equation}\label{eq:prob4}
\textstyle \G^{[m - 1, N]}_{\vT,t} (\vec y, \vec x) = \sum_{\vec u, \vec a \in \Omega_{N - m + 1}} \G^{[m, N]}_{\vT,T_{m-1} - 1} (\vec y, \vec u) \G^{[m, N]}_{0,1} (\vec u, \vec a) \G^{[m-1, N]}_{T_{m-1}, t} (y_{m-1} \sqcup \vec a, \vec x).
\end{equation}
The induction hypothesis implies that the function $\G^{[m, N]}_{\vT,T_{m-1} - 1}$ has the required form \eqref{eq:SchuetzTASEP2}. Moreover, direct computations show that the functions $\G^{[m, N]}_{0,1}$ and $\G^{[m-1, N]}_{T_{m-1}, t}$ are in the form \eqref{eq:FRkernel}, which allows one to apply Prop.~\ref{prop:Cauchy-Binet_general} to their convolution.
Then the last expression turns to 
\begin{equation}
\textstyle \sum_{\vec a \in \Omega_{N - m + 1}} \sum_{\substack{\vec z^{\,i} \in \Omega_{N - i} : \\ m \leq i < N}} \prod_{i = m + 1}^{N} \G^{[i, N]}_{T_{i}, T_{i - 1}} (y_{i} \sqcup \vec z^{\,i}, \vec z^{\,i-1}) \G^{[m, N]}_{T_{m}, T_{m-1}} (y_{m} \sqcup \vec z^{\,m}, \vec a) \G^{[m-1, N]}_{T_{m-1}, t} (y_{m-1} \sqcup \vec a, \vec x),
\end{equation} 
which is exactly \eqref{eq:SchuetzTASEP2} for $m-1$.
\end{proof}

\begin{lem}\label{lem:G-formula}
For $\kappa \geq 1$, $\vec y\in\Omega_N(\kappa)$ and $\vec x\in\Omega_N$,
\begin{equation}\label{eq:G-formula}
	\G^{\rB}_{\vT,0} (\vec y, \vec x) = \det \bigl[F^{\rB}_{i - j}(x_{N + 1 - i} - y_{N + 1 - j}, \kappa (j-1))\bigr]_{i, j \in \set{N}}.
\end{equation}
\end{lem}

\begin{proof}
This follows from applying Prop.~\ref{prop:Cauchy-Binet2} consecutively to the determinants in \eqref{eq:SchuetzTASEP2}.
\end{proof}

\subsection{Proof of Eqn.~\ref{eq:backward-in-time-one}}

\begin{lem}\label{lem:rB-back-in-time}
Identity \eqref{eq:backward-in-time-one} holds for the model \eqref{eq:rightBernoulliBlock} with right Bernoulli jumps, where in the definition of the function \eqref{eq:rBerBlockF} with negative time the singularity $-q/p$ should be excluded from the contour.
\end{lem}

\begin{proof}
If we exclude the singularity from the contour, then the function \eqref{eq:rBerBlockF} satisfies $F^{\rB}_{i - N + 1}(z_{N + 1 - i} - x_{2}, -t) = 0$ if $x_2 > z_2$, where we use the variables as in \eqref{eq:backward-in-time-one}. Hence, the restriction $x_2 < x_1$ in the sum in \eqref{eq:backward-in-time-one} can be omitted (this is because each term of the sum may be non-vanishing only when $x_2 \leq z_2$, and the dynamics implies $y_1 \leq x_1$, which by the assumptions on the variables yields $x_2 < x_1$). Applying then \eqref{eq:rightBernoulliBlock} and Prop.~\ref{prop:Cauchy-Binet5}, the left hand side of \eqref{eq:backward-in-time-one} turns to
\begin{equation}
\det \bigl[F_{i - j}(x_1 \cdot \uno{i = N} + z_{N + 1 - i} \cdot \uno{i < N} - y_{N + 1 - j}, t \cdot \uno{i = N})\bigr]_{i, j \in \set{N}}.
\end{equation}
One can prove that this determinant equals the right hand side of \eqref{eq:backward-in-time-one} in the same way as the initial condition is checked in \cite[Prop.~2.1]{bp-push}.
\end{proof}

\section{Proofs for right geometric jumps with sequential update}
\label{sec:rGeometric-proof}

As we described in Sec.~\ref{sec:rightGeometric}, TASEP with right geometric jumps is different from the other models in that section, and in particular Assum.~\ref{a:kappa} and hence Thm.~\ref{thm:main2} do not hold for it.
Because of this we need to prove the formula \eqref{eq:RG_distribution} directly in the case of sequential update ($\kappa=-1$).

We start with an auxiliary result.
Let $\G^{[1, N]}_{0, t} (\vec y, \vec x)$ be the function on the right hand side of \eqref{eq:rightGeometric}. Then for $\vT = (T_i)_{i \in \set N}$ with $T_i = i - N$,
for $m \in \set{N}$, $\vec x,\vec y \in \Omega_{m}$ and $T_{m} \leq t < T_{m + 1}$ (with the convention $T_{N+1} = \infty$) we define
\begin{equation}\label{eq:SchuetzRG_0}
	\G^{[1, m]}_{\vT,t} (\vec y, \vec x) = \pp \bigl(\xx^{\rG}_{t}(i) = x_i,\,i\in\set{m} \big| \xx^{\rG}_{T_i}(i) = y_i, i \in \set{m}\bigr),
\end{equation}
which is the transition probability for the particles $\xx^{\rG}(1), \dotsc, \xx^{\rG}(m)$ from $y_1, \dotsc, y_{m}$ at times $T_1, \dotsc, T_m$ to the locations $\vec x$ at time $t$.

\begin{lem}\label{lem:SchuetzRG}
In above setting, and with the notation $\vec z \sqcup y$ from Appdx.~\ref{app:convolutions},
\begin{equation}\label{eq:SchuetzRG2}
\textstyle \G^{[1, m]}_{\vT,t} (\vec y, \vec x) = \sum_{\substack{\vec z(T_i) \in \Omega_{i} \\ 1 \leq i < m}} \left(\prod_{i = 1}^{m-1} \G^{[1, i]}_{T_{i}, T_{i + 1}} (\vec z(T_i) \sqcup y_{i}, \vec z(T_{i+1}))\right) \G^{[1, m]}_{T_{m}, t} (\vec z(T_m) \sqcup y_{m}, \vec x),
\end{equation}
where $\vec z(T_1)$ is an empty vector.
\end{lem}

\begin{proof}
We will prove \eqref{eq:SchuetzRG2} by induction over $m = 1, 2, \dotsc, N$. For the base case, $m = 1$, note that on the time interval $T_1 \leq t < T_{2}$ only the $1^{\text{st}}$ particle moves. Therefore $\G^{[1, 1]}_{\vT, t} (\vec y, \vec x) = \G^{[1, 1]}_{T_1, t} (y_1, x_1)$, which is \eqref{eq:SchuetzRG2}.
Now assuming that \eqref{eq:SchuetzRG2} holds for some $1 \leq m < N$, we will prove it for $m+1$.
For $T_{m+1} \leq t < T_{m + 2}$ and $\vec x \in \Omega_{m+1}$ the Markov property yields
\begin{align}\label{eq:probRG0}
&\textstyle \G^{[1, m + 1]}_{\vT,t} (\vec y, \vec x) \textstyle= \sum_{\!\substack{\vec u, \vec a \in \Omega_{m} : \\ a_m > y_{m+1}}} \pp \bigl(\xx^{\rG}_{T_{m+1} - 1} = \vec u \big| \xx^{\rG}_{T_i}(i) = y_i, i \in \set{N}\bigr)\quad\mbox{}\\
 &\textstyle\hspace{0.2in}\times \pp \bigl(\xx^{\rG}_{T_{m+1}} = \vec a \sqcup y_{m+1}\big| \xx^{\rG}_{T_{m+1} - 1} = \vec u\bigr)
 \pp \bigl(\xx^{\rG}_{t}(i) = x_i,\,i\in\set{m + 1} \big| \xx^{\rG}_{T_{m+1}} = \vec a \sqcup y_{m+1}\bigr).
\end{align} 
The definition of the model implies that the terms contributing to the sum have $a_m \geq y_m > y_{m+1}$. Hence, the restriction $a_m > y_{m+1}$ in the sum can be omitted. The induction hypothesis yields $\pp \bigl(\xx^{\rG}_{T_{m+1} - 1} = \vec u \big| \xx^{\rG}_{T_i}(i) = y_i, i \in \set{N}\bigr) = \G^{[1, m]}_{\vT,T_{m+1} - 1} (\vec y, \vec u)$, where the latter is given by the right hand side of \eqref{eq:SchuetzRG2}. Moreover, $\pp \bigl(\xx^{\rG}_{T_{m+1}} = \vec a \sqcup y_{m+1}\big| \xx^{\rG}_{T_{m+1} - 1} = \vec u\bigr) = \G^{[1, m]}_{0,1} (\vec u, \vec a)$ and
\begin{align}
\textstyle \pp \bigl(\xx^{\rG}_{t}(i) = x_i,\,i\in\set{m + 1} \big| \xx^{\rG}_{T_{m+1}} = \vec a \sqcup y_{m+1}\bigr) = \G^{[1, m+1]}_{T_{m+1}, t} (\vec a \sqcup y_{m+1}, \vec x).
\end{align}
Then \eqref{eq:probRG0} can be written as
\begin{align}
\textstyle \G^{[1, m + 1]}_{\vT,t} (\vec y, \vec x) = \sum_{\vec u, \vec a \in \Omega_{m}} \G^{[1, m]}_{\vT,T_{m+1} - 1} (\vec y, \vec u) \G^{[1, m]}_{0,1} (\vec u, \vec a) \G^{[1, m+1]}_{T_{m+1}, t} (\vec a \sqcup y_{m+1}, \vec x).
\end{align}
By the induction hypothesis, the function $\G^{[1, m]}_{\vT,T_{m+1} - 1} (\vec y, \vec u)$ has the necessary form \eqref{eq:SchuetzRG2}. Moreover, the functions $\G^{[1, m]}_{0,1}$ and $\G^{[1, m+1]}_{T_{m+1}, t}$ are in the form \eqref{eq:FRkernel}, and we can apply Prop.~\ref{prop:Cauchy-Binet_general} to their convolution.
Then the last expression turns into \eqref{eq:SchuetzRG2} for $m+1$.
\end{proof}

By analogy with \eqref{eq:E-event} in the case $\kappa = 1$, we define the event
\begin{equation}
\bar \CE = \bigcap_{i \in \set N} \bigl\{i^{\text{th}} ~\text{particle stays put till time}~ i - N\bigr\}.
\end{equation}
The following result is the analogue of identity \eqref{eq:second-property} which we are going to use for this model.

\begin{lem}\label{lem:RG-assumption}
For any $\vec x, \vec y \in \Omega_N$,
\begin{equation}\label{eq:second-property-RG}
\pp(X^{\rG}_0 = \vec x | X^{\rG}_{1 - N} = \vec y, \bar \CE) = \det \bigl[F^{\rG}_{i - j}(x_{N + 1 - i} - y_{N + 1 - j}, N-j)\bigr]_{i, j \in \set{N}}.
\end{equation}
\end{lem}

\begin{proof}
Formula \eqref{eq:second-property-RG} is obtained by applying Prop.~\ref{prop:Cauchy-Binet3} consecutively to the determinants in \eqref{eq:SchuetzRG2}.
\end{proof}

The proof of Lem.~\ref{lem:RG-assumption} does not use the fact that the particles have geometric jumps, and in fact \eqref{eq:second-property-RG} also holds for other models, e.g. right Bernoulli jumps.
However, only for $X_t^{\rG}$ the formula proves to be useful when considering different ending times $t + i - N$ for each particle $i$.
Indeed, if we consider such ending times for the model with right Bernoulli jumps and sequential update, then after the $i^{\text{th}}$ particle stops, the $(i+1)^{\text{st}}$ particle needs to make one step and could jump on top of its right neighbor. 
In contrast, for the model with right geometric jumps, since the basic update rule is parallel, when the $i^{\text{th}}$ particle stops, the $(i+1)^{\text{st}}$ particle cannot jump over it on the next step, because it is still blocked by the position of its neighbor at the previous time.
This suggests the following analog of Lem.~\ref{lem:Ber-TASEP_prll}:

\begin{lem}\label{lem:rightGeom}
For $t \geq N-1$ and $\vec x, \vec y \in\Omega_N$, and with $\CN(\vec x)$ as defined in Lem.~\ref{lem:Ber-TASEP_prll},
\begin{equation}\label{eq:rightGeom2}
\pp(X^{\rG}_{t + i - N}(i) = x_i, ~i \in \set{N} | X^{\rG}_{0} = \vec y) = p^{-\CN(\vec x)} \det \bigl[F^{\rG}_{i - j}(x_{N + 1 - i} - y_{N + 1 - j}, t + i - N)\bigr]_{i, j \in \set{N}}.
\end{equation}
\end{lem}

\begin{proof}
Write $S_i = t + i - N$ and let $\G^{[1,N]}_{0,\vS} (\vec y, \vec x)$ be the probability on the left hand side of \eqref{eq:rightGeom2}. At any time point $s$ we consider $N \geq 1$ particles $\xx^{\rG}_{s}(N) < \dotsm < \xx^{\rG}_{s}(1)$, such that on the time interval $s > S_{i-1}$ (with the convention $S_0 = 0$) only the particles $\xx^{\rG}_{s}(i)$, $\dotsc$, $\xx^{\rG}_{s}(N)$ move. 
Denote these moving particles by $\xx^{[i, N]}_{s} = (\xx^{\rG}_{s}(i) , \dotsc, \xx^{\rG}_{s}(N))$, and let $\G^{[i, N]}$ be their transition function.

We prove \eqref{eq:rightGeom2} by induction over $N \geq 1$. The base case $N = 1$ is trivial. Assuming that \eqref{eq:rightGeom2} holds for $N - 1 \geq 1$, we will prove it for $N$. 
From the Markov property we may write $\G^{[1,N]}_{0,\vS} (\vec y, \vec x)$ as
\begin{align}
& \textstyle\sum_{{\vec u \in \Omega_{N-1} : \, u_{1} < x_1}} \sum_{{\vec a \in \Omega_{N-1} : \, a_{1} = x_{2}}} \pp \bigl(\xx^{[1, N]}_{S_1} = x_1 \sqcup \vec u \big| \xx^{[1, N]}_{0} = \vec y\bigr) \pp \bigl(\xx^{[2, N]}_{S_2} = \vec a \big| \xx^{[1, N]}_{S_1} = x_1 \sqcup \vec u\bigr)\\[-12pt]
&\hspace{5.5cm}\times \pp \bigl(\xx^{\rG}_{S_k}(k) = x_k,\, 2 \leq k \leq N \big| \xx^{[2, N]}_{S_2} = \vec a\bigr).\label{eq:rightGeom-interm}
\end{align}
We have $\pp (\xx^{[1, N]}_{S_1} = x_1 \sqcup \vec u | \xx^{[1, N]}_{t} = \vec y) = \G^{[1, N]}_{t, S_1} (\vec y, x_1 \sqcup \vec u)$, $\pp (\xx^{[2, N]}_{S_2} = \vec a |  \xx^{[1, N]}_{S_1} = x_1 \sqcup \vec u) = p^{-\uno{x_1 - x_2 = 1}} \G^{[2, N]}_{0, 1} (\vec u, \vec a)$, and $\pp (\xx^{\rG}_{S_k}(k) = x_k,\, 2 \leq k \leq N | \xx^{[2, N]}_{S_2} = \vec a) = \G^{[2, N]}_{S_2, \vS_{> 1}} (\vec a, \vec x_{> 1})$, where $\vS_{> 1}$ and $\vec x_{> 1}$ are obtained from $\vS$ and $\vec x$ respectively by removing the first entries. The multiplier $p^{-\uno{x_1 - x_2 = 1}}$ is needed to change the jump probability of the $2^{\text{nd}}$ particle in the case $x_1 - x_2 = 1$.

The function $\G^{[2, N]}_{0, 1} (\vec u, x_2 \sqcup \vec a)$ equals the probability for $N-1$ particles to go from $\vec u$ to $x_2 \sqcup \vec a$ during a unit time interval. Hence, it can be non-zero only if $u_1 \leq x_2$, which yields $u_1 \leq x_2 < x_1$, so the restriction $u_{1} < x_1$ in the sum can be omitted. Moreover, for $a_1 \neq x_2$ the last probability in \eqref{eq:rightGeom-interm} vanishes, and the restriction $a_1 = x_2$ in the sum can be omitted. Therefore
\begin{equation}
\textstyle\G^{[1,N]}_{0,\vS} (\vec y, \vec x) = p^{-\uno{x_1 - x_2 = 1}} \sum_{\vec u \in \Omega_{N-1}} \sum_{\vec a \in \Omega_{N-1}} \G^{[1, N]}_{0, S_1} (\vec y, x_1 \sqcup \vec u) \G^{[2, N]}_{0, 1} (\vec u, \vec a) \G^{[2, N]}_{S_2, \vS_{> 1}} (\vec a, \vec x_{> 1}).
\end{equation}
Using the induction hypothesis for the function $\G^{[2, N]}_{S_2, \vS_{> 1}} (\vec a, \vec x_{> 1})$ and applying Prop.~\ref{prop:Cauchy-Binet_general} to the sum over $\vec a$ allows to write the preceding expression as
\begin{equation}
\textstyle p^{-\uno{x_1 - x_2 = 1}} \sum_{\vec u \in \Omega_{N-1}} \G^{[1, N]}_{0, S_1} (\vec y, x_1 \sqcup \vec u) \G^{[2, N]}_{S_2, \vS_{> 1}} (\vec u, \vec x_{> 1}).
\end{equation}
Using again the induction hypothesis to the function $\G^{[2, N]}_{S_2, \vS_{> 1}}$ and applying Prop.~\ref{prop:Cauchy-Binet5} to the sum, we obtain \eqref{eq:rightGeom2}.
\end{proof}

\begin{proof}[Proof of Prop.~\ref{prop:RG_distribution}]
In the case of parallel update, formula \eqref{eq:RG_distribution} follows from Thm.~\ref{thm:main}. 
From now on we consider the case of sequential update ($\kappa=-1$).
Then the proof goes along the lines of the proof of Thm.~\ref{thm:main2}, the only difference being the orientation of space-like paths, so we only provide a sketch of the proof. 

Define the set of \emph{space-like paths} for this model as
\begin{equation}
\bar \SLP_N = \bigcup_{m \geq 1} \bigl\{(\fn_i)_{i \in \set{m}}\!: \fn_i \in \set{N} \times \nn_0, \fn_i \overline{\prec} \fn_{i+1}\bigr\},
\end{equation}
where the relation $(n_1, t_1) \overline{\prec} (n_2, t_2)$ now means $n_1 \leq n_2$,  $t_1 \leq t_2$ and $(n_1, t_1) \neq (n_2, t_2)$. Then for $T_i = i - N$,  $\cS = \{(n_1, t_1), \dotsc, (n_m, t_m)\} \in \bar \SLP_N$ and for $\vec y \in \Omega_N$ and $\vec x \in \Omega_m$ we define 
\begin{equation}
\G^{\rG}_{\vT, \cS}(\vec y, \vec x) = \pp \bigl(\xx^{\rG}_{t_i}(n_i) = x_i,\,i\in\set{m} \big| \xx^{\rG}_{T_i}(i) = y_i, i \in \set{N}\bigr).
\end{equation}
Let the set $\Omega_{n, N}$ contain the vectors $(x_n, \ldots, x_N)$ such that $x_{n} < x_{n+1} < \cdots < x_{N}$. Then by analogy with \eqref{eq:G+_new} we can write 
\begin{equation}
\textstyle\G^{\rG}_{\vT, \cS} (\vec y, \vec x) = \sum_{\vec x(0) \in \Omega_{N}} \sum_{\!\!\substack{\vec x(t_i) \in \Omega_{n_i, N}: \\ x_{1}(t_i) = x_i, i \in \set{m}}} \G^{\rG}_{\vT,0} (\vec y, \vec x(0)) \prod_{i=1}^{m} \G^{\rG}_{t_{i} - t_{i-1}}(\vec x_{\geq n_i}(t_{i-1}), \vec x(t_i)),
\end{equation}
where $t_0 = 0$, the function $\G^{\rG}_{\vT,0}$ equals \eqref{eq:SchuetzRG2} in the case $t = 0$, and where $\G^{\rG}_{t}$ is the transition function given by the right hand side of \eqref{eq:rightGeometric}. 

Define $\un_i = N - n_i + 1$, so that $\un_1 \geq \un_2 \geq \dotsm \geq \un_m$. 
Analogously to \eqref{eq:DomainD} we define the domain
\begin{align}
\bar\D_{\cS} &= \bigl\{\x^{\un_0}_{\ell}(t_0) \in \zz : \ell \in \set{\un_0}\bigr\} \\
&\qquad \cup \bigcup_{i \in \set{m}} \Bigl\{\x^n_{\ell}(t_i) \in \zz : \un_{i+1} \leq n \leq \un_i, \ell \in \set{n} ~\text{ such that }~ \x^{n+1}_\ell(t_i) < \x^n_\ell(t_i) \leq \x^{n+1}_{\ell+1}(t_i)\Bigr\}, 
\end{align}
where $t_0 = 0$, $\un_0 = N$ and $\un_{m+1} = 0$.
Next we define a signed measure $\bar{\cW}$ on $\X \in \bar \D_{\cS}$ through \eqref{eq:W_def} using the time points $t_i$ in place of $\ut_i$ and where the functions \eqref{eq:T_def} and \eqref{eq:Psi_def} are defined with $\varphi(w) = p / (1 - q w)$. Then, as in Prop.~\ref{prop:G+_expansion}, we can write
\begin{equation}
\textstyle\G^{\rG}_{\vT, \cS}(\vec y, \vec x) = C\sum_{\,{\X \in \bar \D_{\cS}:  \x_{\un_i}^{\un_i}(t_i) = x_i, i \in \set{m}}}\bar \cW(\X),
\end{equation}
for a constant $C \neq 0$. As in Thm.~\ref{thm:biorth_general} we can compute the correlation kernel of the determinantal measure $\bar \cW$, which yields formula \eqref{eq:RG_distribution} with the kernel given by \eqref{eq:KernelK} with the value of $\kappa$ equal $-1$. Applying then Thm.~\ref{thm:kernel-rw} with the functions $\psi(w) = \varphi(w)^t$ and $a(w) = 1/\varphi(w)$, we get \eqref{eq:RG_distribution}.\end{proof}

\section{Formulas for discrete-time RSK-solvable models}
\label{app:DW}

In this appendix we derive the transition probabilities for the discrete-time variants of TASEP described in Sec.~\ref{sec:caterpillars}, by rewriting in the form \eqref{eq:G-main} the formulas which were derived in \cite{MR2469339} using the four basic variants of the Robinson-Schensted-Knuth (RSK) algorithm.

Fix a vector $\vec\alpha = (\alpha_1, \dotsc, \alpha_N) \in \rr^N$.
The $r^{\text{th}}$ \emph{complete homogeneous symmetric polynomial} and the $r^{\text{th}}$ \emph{elementary symmetric function} are given respectively by
\begin{equation}
	\textstyle h_r(\vec\alpha) = \sum_{\!\!\substack{k_1, \dotsc, k_N \geq 0 \\ k_1 + \dotsm + k_N = r}} \alpha_1^{k_1} \alpha_2^{k_2} \dotsm \alpha_N^{k_N}, \qquad e_r(\vec\alpha) = \sum_{k_1 < k_2 < \dotsm < k_r} \alpha_{k_1} \alpha_{k_2} \dotsm \alpha_{k_r},
\end{equation}
where by convention $h_0 \equiv e_0 \equiv 1$ and $h_r \equiv e_r \equiv 0$ for $r < 0$. For $0 \leq k < \ell \leq N$, let $\vec\alpha^{(k, \ell)} = (0, \dotsc, 0, \alpha_{k+1}, \dotsc, \alpha_\ell, 0, \dotsc, 0)$ be the vector obtained from $\vec\alpha$ by setting the first $k$ and the last $N-\ell$ entries to $0$. 
Write $h^{(k, \ell)}_r(\vec\alpha) = h_r(\vec \alpha^{(k, \ell)})$ and $e^{(k, \ell)}_r(\vec \alpha) = e_r( \vec\alpha^{(k, \ell)})$ with $h^{(k, k)}_r( \vec\alpha) = e^{(k, k)}_r( \vec\alpha) = \uno{r = 0}$ and then, for a fixed function $f$ on $\zz$, define
\[\textstyle f^{(ij)}_{\vec \alpha}(k)=\sum_{\ell=0}^{i-j}(-1)^\ell e_\ell^{(ji)}(\vec \alpha)f(k+\ell)\uno{i\geq j}+\sum_{\ell=0}^{\infty}h_\ell^{(ij)}(\vec \alpha)f(k+\ell)\uno{i<j},\]
provided the series converges absolutely,
and define $\hat f^{(ij)}_{\vec \alpha}(k)$ in the same way except that $f(k+\ell)$ is replaced by $f(k-\ell)$
The formulas in \cite{MR2469339} are written in terms of $f^{(ij)}_{\vec \alpha}(k)$ and $\hat f^{(ij)}_{\vec \alpha}(k)$.
Our first task is to find an alternative expression for them.

Taking all entries of $\vec \alpha$ to be non-zero, define, for $k,\ell\in\set{N}$ and $x,y\in\zz$,
\begin{equation}\label{eq:H_symm}
\textstyle H^{(\vec\alpha)}_{k, \ell}(x,y) = \frac{1}{2\pi\I}\oint_{\gamma} \frac{\d w}{w^{x-y+1}} \frac{\prod_{i = 1}^{\ell} (1 - \alpha_i w)}{\prod_{i = 1}^{k} (1 - \alpha_i w)},
\end{equation}
where the contour $\gamma$ encloses $0$ but not any pole $w=1/\alpha_i$.
The kernel $H^{(\vec\alpha)}_{k, \ell}$ is related to symmetric functions:

\begin{lem}
$H^{(\vec\alpha)}_{k, \ell}(x,y) = (-1)^{x-y} e^{(k, \ell)}_{x-y}(\vec\alpha)$ if $\ell \geq k$, and $H^{(\vec\alpha)}_{k, \ell}(x,y) = h^{(\ell, k)}_{x-y}(\vec\alpha)$ if $\ell < k$.
\end{lem}

\begin{proof}
Write $H^{(\vec\alpha)}_{k, \ell}(x)=H^{(\vec\alpha)}_{k, \ell}(x,0)$.
If $x < 0$, then $H^{(\vec\alpha)}_{k, \ell}(x) = 0$, because the contour in \eqref{eq:H_symm} does not enclose any poles of the function in the integral. This yields the required identities, because $e^{(k, \ell)}_{x}(\vec\alpha) = h^{(\ell, k)}_{x}(\vec\alpha) = 0$. 

Consider now $x \geq 0$. The case $\ell = k$ is trivial, because $H^{(\vec\alpha)}_{k, k}(x) = \uno{x = 0}$, which coincides with $(-1)^{x} e^{(k, k)}_{x}(\vec\alpha)$.
In the case $\ell > k$, the Cauchy residue theorem yields 
\begin{equation}
\textstyle H^{(\vec\alpha)}_{k, \ell}(x) = \frac{1}{x!} \frac{\d^{x}}{\d w^{x}} \prod_{i = k+1}^{\ell} (1 - \alpha_i w) \Big|_{w = 0} = (-1)^{x} \sum_{k+1 \leq j_1 < \dotsm < j_{x} \leq \ell} \alpha_{j_1} \dotsm \alpha_{j_x} = (-1)^{x} e^{(k, \ell)}_{x}(\vec\alpha).
\end{equation}
In the case $\ell < k$, choosing the contour so that $|w| < 1/\alpha_i$ for each $i$, we can write $(1 - \alpha_i w)^{-1} = \sum_{k_i \geq 0} (\alpha_i w)^{k_i}$, and the Cauchy residue theorem yields 
\[\textstyle H^{(\vec\alpha)}_{k, \ell}(x) = \sum_{\!\substack{j_{\ell + 1}, \dotsc, j_k \geq 0 \\ j_{\ell+1} + \dotsm + j_k = x}} \alpha_{\ell+1}^{j_{\ell+1}} \dotsm \alpha_k^{j_k} = h^{(\ell, k)}_{x}(\vec\alpha).
  \qedhere\]
\end{proof}

Using the lemma and the fact that $e^{(k,\ell)}_y(\vec\alpha)=0$ if $y>\ell-k$, the above functions can be written as
\begin{equation}\label{eq:newDW}
\textstyle f^{(\ell, k)}_{\vec\alpha} = (H^{(\vec\alpha)}_{k, \ell})^*f\qqand\hat f^{(\ell, k)}_{\vec\alpha} = H^{(\vec\alpha)}_{k, \ell}f.
\end{equation}

\subsection{Proof of Eqn.~\ref{eq:rightBernoulliBlock}}
\label{sec:Ber-TASEP}

For the process $X^{\rB}_t \in \Omega_N$ defined in Sec.~\ref{sec:RB_caterpillars}, set $Y^{\rB}_t(i) = X^{\rB}_t(i) + i$, so that $Y^{\rB}_t \in \bar \Omega_N$ (see \eqref{eq:OmegaBar} and the evolution of $Y^{\rB}_t$ coincides with the model from Case~B in \cite[Sec.~2]{MR2469339}.
If we denote $v_i = p_i / q_i$ and let the function $H^{(\vec v)}$ be defined by \eqref{eq:H_symm} with values $\alpha_i = v_i$, then in view of \eqref{eq:newDW} the formula from \cite[Thm.~1]{MR2469339} becomes
\begin{equation}\label{eq:ProbBernoulliBlock}
 \textstyle 	\pp (Y^{\rB}_t = \vec x \,|\, Y^{\rB}_0 = \vec y) = \left(\prod_{i=1}^N q_i^{t} v_i^{x_i - y_i}\right) \det \bigl[ (H^{(\vec v)}_{k, \ell})^* \nu_t(x_\ell - y_k - \ell + k) \bigr]_{k, \ell \in \set{N}},
\end{equation}
for two configurations $\vec x, \vec y \in \bar \Omega_N$ and with $\nu_t(x) = {t \choose x} \uno{0 \leq x \leq t}$.
Using \eqref{eq:H_symm} and the contour integral formula $\nu_t(x)=\frac{1}{2\pi \I} \oint_{\Gamma_{0}}\!\d w\frac{(1+ w)^t}{w^{x+1}}$ we get
\begin{equation}
\textstyle (H^{(\vec v)}_{k, \ell})^*\nu_t(x)
=\frac{1}{2\pi \I} \oint_{\gamma'} \frac{\d w}{w^{x + \ell - k +1}} \frac{\prod_{i=1}^\ell (w - v_i)}{\prod_{i=1}^k (w - v_i)} (1+ w)^t,
\end{equation}
where the integration contour $\gamma'$ includes $0$ and all entries of $\vec v$.
Changing back to $X^{\rB}_t(i) = Y^{\rB}_t(i) - i$ in \eqref{eq:ProbBernoulliBlock} and taking all speeds to be equal, we arrive at \eqref{eq:rightBernoulliBlock} after a simple change of variables.

\subsection{Proof of Eqn.~\ref{eq:Ber-pushTASEP}}
\label{sec:Ber-pushTASEP}

For $\xx^{\lB}_t$ as in Sec.~\ref{sec:LB}, we define the process $Y^{\lB}_t(i) = - X^{\lB}_t(i) - i$. Then $Y^{\lB}_t(1) \leq Y^{\lB}_t(2) \leq \dotsm \leq Y^{\lB}_t(N)$; we denote by $\tilde \Omega_N$ the set of such configurations. 
Proceeding as in the previous case, Case~D of \cite[Thm.~1]{MR2469339} yields, for $\nu_t$ as in the previous case,
\begin{equation}\label{eq:ProbBerPush}
\textstyle  \pp(Y^{\lB}_t = \vec x | Y^{\lB}_0 = \vec y) = \left(\prod_{i=1}^N q_i^{t} v_i^{y_i - x_i}\right) \det \bigl[H^{(\vec v)}_{k, \ell} \nu_t (x_\ell - y_k + \ell - k) \bigr]_{k, \ell \in \set{N}},
\end{equation}
where $\vec x, \vec y \in \tilde \Omega_N$ and $v_i = q_i / p_i$. Using \eqref{eq:H_symm} and the integral representation of $\nu_t$ from Sec.~\ref{sec:Ber-TASEP}, we get
\begin{equation}
\textstyle H^{(\vec v)}_{k, \ell} \nu_t(x)
=\frac{1}{2\pi \I} \oint_{\gamma'} \frac{\d w}{w^{-x + \ell - k + 1}} \frac{\prod_{i = 1}^{\ell} (w - v_i)}{\prod_{i = 1}^{k} (w - v_i)} (1 + 1 / w)^{t},
\end{equation}
where the integration contour $\gamma'$ includes $0$ and all entries of $\vec v$. Changing back to $X^{\lB}_t(i) = - Y^{\lB}_t(i) - i$ in \eqref{eq:ProbBerPush} and taking all speeds to be equal, we get \eqref{eq:Ber-pushTASEP}.

\subsection{Proof of Eqn.~\ref{eq:ProbGeomPush}}
\label{sec:ProbGeomPush}

For $\xx^{\lG}_t$ as in Sec.~\ref{sec:LG}, we define the process $Y^{\lG}_t(i) = -X^{\lG}_t(i) - i$ so that $Y^{\lG}_t$ lives in $\tilde \Omega_N$ as in the previous case.
We set $\mu_t(x) = {t + x - 1 \choose x} \uno{x \geq 0, t \geq 1} + \uno{t = x = 0}$, which can be written as $\mu_t(x) = \frac{1}{2\pi \I} \oint_{\tilde \gamma} \frac{(1 - w)^{-t}}{w^{x+1}} \d w$, where the contour $\tilde \gamma$ includes $0$, but not $1$. Then Case~A of \cite[Thm.~1]{MR2469339} yields 
\begin{equation}\label{eq:ProbGeomPush2}
\textstyle  \pp(Y^{\lG}_t = \vec x | Y^{\lG}_0 = \vec y) = \left( \prod_{i=1}^N p_i^t q_i^{x_i - y_i}\right) \det \bigl[ H^{(\vec v)}_{k, \ell} \mu_t(x_\ell - y_k + \ell - k) \bigl]_{k, \ell \in \set{N}},
\end{equation}
where $\vec x, \vec y \in \tilde \Omega_N$ and $v_i = 1/q_i$.
Then using \eqref{eq:H_symm} we can write 
\begin{align}
\textstyle H^{(\vec v)}_{k, \ell} \mu_t(x) = \frac{1}{2\pi \I} \oint_{\gamma} \frac{\d w}{w^{-x + \ell - k +1}} \frac{\prod_{i=1}^\ell (w - 1 / q_i)}{\prod_{i=1}^k (w - 1 / q_i)} (1- 1/w)^{-t},
\end{align}
where the contour $\gamma$ encloses $0$, $1$ and all values $v_i$. Changing back to $X^{\lG}_t(i) = -Y^{\lG}_t(i) - i$ in \eqref{eq:ProbGeomPush2} and taking all speeds to be equal, we get \eqref{eq:ProbGeomPush}.

\subsection{Proof of Eqn.~\ref{eq:rightGeometric}}
\label{sec:GeomPush}

For $\xx^{\rG}_t$ as in Sec.~\ref{sec:rightGeometric-prll}, let us define the process $Y^{\rG}_t(i) = X^{\rG}_t(i) + i$, so that $Y^{\rG}_t \in \bar \Omega_N$. 
Case~C of \cite[Thm.~1]{MR2469339} yields, for $\mu_t$ as in the previous case,
\begin{equation}\label{eq:ProbGeometricBlock}
\textstyle  \pp(Y^{\rG}_t = \vec x | Y^{\rG}_0 = \vec y) = \left(\prod_{i=1}^N p_i^{t} q_i^{x_i - y_i}\right) \det \bigl[ (H^{(\vec v)}_{k, \ell})^* \mu_t (x_\ell - y_k - \ell + k) \bigr]_{k, \ell \in \set{N}},
\end{equation}
where $\vec x, \vec y \in \bar \Omega_N$ and $v_i = q_i$. Using the integral representation of $\mu_t$ and \eqref{eq:H_symm}, we may write
\begin{align}
\textstyle (H^{(\vec v)}_{k, \ell})^* \mu_t(x) = \frac{1}{2\pi \I} \oint_{\gamma} \frac{\d w}{w^{x + \ell - k +1}} \frac{\prod_{i=1}^\ell (w - q_i)}{\prod_{i=1}^k (w - q_i)} (1- w)^{-t},
\end{align}
where the contour $\gamma$ includes $0$ and all entries of $\vec v$, but does not include $1$. Changing back to $X^{\rG}_t(i) = Y^{\rG}_t(i) - i$ in \eqref{eq:ProbGeometricBlock} and taking all speeds to be equal, we get \eqref{eq:rightGeometric}.

\vs\vs
\noindent{\bf Acknowledgements.}
The authors would like to thank Alexei Borodin for discussions several years ago which motivated some of the results of this paper; Patrik Ferrari for pointing out to us the connection between sequential and parallel updates at the level of Markov chains on Gelfand-Tsetlin patterns \cite{bf-tilings}; and Jeremy Quastel for many valuable discussions related to this work.
KM was partially supported by NSF grant DMS-1953859.
The also thank an anonymous referee for a very detailed and helpful report.
DR was supported by Centro de Modelamiento Matem\'{a}tico (CMM) Basal Funds FB210005 from ANID-Chile, by Fondecyt Grant 1201914, and by Programa Iniciativa Cient\'ifica Milenio grant number NC120062 through Nucleus Millennium Stochastic Models of Complex and Disordered Systems.

\printbibliography[heading=apa]

\end{document}

%% file: catp-ch.tex
\newcommand{\catp}{\raisebox{0.1em}{\tikz[y=0.80pt,x=0.80pt,yscale=-1.0,xscale=1.0,inner sep=0pt,outer sep=0pt]%
{
\path[draw=black,line join=bevel,line cap=round,miter limit=4.00,line width=0.280pt] (-240.5382,165.4074) circle (0.0476cm);
\path[draw=black,fill=white,line join=bevel,line cap=round,miter limit=4.00,line width=0.280pt] (-239.2864,165.0583) circle (0.0476cm);
\path[draw=black,fill=white,line join=bevel,line cap=round,miter limit=4.00,line width=0.280pt] (-238.0008,164.9891) circle (0.0476cm);
\path[draw=black,fill=white,line join=bevel,line cap=round,miter limit=4.00,line width=0.280pt] (-236.6079,164.6444) circle (0.0476cm);
\path[draw=black,fill=white,line join=bevel,line cap=round,miter limit=4.00,line width=0.280pt] (-235.2670,164.3894) circle (0.0476cm);
\path[draw=black,fill=white,line join=bevel,line cap=round,miter limit=4.00,line width=0.280pt] (-234.0488,163.8984) circle (0.0476cm);
\path[draw=black,fill=white,line join=bevel,line cap=round,miter limit=4.00,line width=0.280pt] (-233.9610,163.2756) ellipse (0.0096cm and 0.0141cm);
\path[draw=black,line join=round,line cap=round,miter limit=4.00,line width=0.130pt] (-234.3312,162.2005) .. controls (-234.3312,162.2005) and (-234.7325,160.8832) .. (-235.0441,161.2231);
\path[draw=black,line join=round,line cap=round,miter limit=4.00,line width=0.130pt] (-233.5987,162.2260) .. controls (-233.5987,162.2260) and (-233.1974,160.9087) .. (-232.8857,161.2486);
}}}
\newcommand{\catpl}{\raisebox{0.1em}{\tikz[y=0.80pt,x=0.80pt,yscale=-1.0,xscale=-1.0,inner sep=0pt,outer sep=0pt]%
{
\path[draw=black,line join=bevel,line cap=round,miter limit=4.00,line width=0.280pt] (-240.5382,165.4074) circle (0.0476cm);
\path[draw=black,fill=white,line join=bevel,line cap=round,miter limit=4.00,line width=0.280pt] (-239.2864,165.0583) circle (0.0476cm);
\path[draw=black,fill=white,line join=bevel,line cap=round,miter limit=4.00,line width=0.280pt] (-238.0008,164.9891) circle (0.0476cm);
\path[draw=black,fill=white,line join=bevel,line cap=round,miter limit=4.00,line width=0.280pt] (-236.6079,164.6444) circle (0.0476cm);
\path[draw=black,fill=white,line join=bevel,line cap=round,miter limit=4.00,line width=0.280pt] (-235.2670,164.3894) circle (0.0476cm);
\path[draw=black,fill=white,line join=bevel,line cap=round,miter limit=4.00,line width=0.280pt] (-234.0488,163.8984) circle (0.0476cm);
\path[draw=black,fill=white,line join=bevel,line cap=round,miter limit=4.00,line width=0.280pt] (-233.9610,163.2756) ellipse (0.0096cm and 0.0141cm);
\path[draw=black,line join=round,line cap=round,miter limit=4.00,line width=0.130pt] (-234.3312,162.2005) .. controls (-234.3312,162.2005) and (-234.7325,160.8832) .. (-235.0441,161.2231);
\path[draw=black,line join=round,line cap=round,miter limit=4.00,line width=0.130pt] (-233.5987,162.2260) .. controls (-233.5987,162.2260) and (-233.1974,160.9087) .. (-232.8857,161.2486);
}}}